\documentclass[reqno,12pt,letterpaper]{amsart}
\usepackage{amsmath,amssymb,amsthm,graphicx,mathrsfs,url}
\usepackage[usenames,dvipsnames]{color}
\usepackage[colorlinks=true,linkcolor=Red,citecolor=Green]{hyperref}
\usepackage{floatrow}
\usepackage[usenames,dvipsnames]{xcolor}
\usepackage{tikz} 
\usepackage{slashed}
\usepackage{graphicx}
\usepackage[colorlinks=true]{hyperref}
\usepackage{color}
\usepackage{amsmath,amsfonts}
\usepackage{calrsfs}
\usepackage{amsmath,environ}
\usepackage{floatrow}
\usepackage{autonum}
\usepackage{amssymb}
\usepackage{amsmath, amsthm}
\usepackage{dsfont}
\usepackage{fancybox}
\usepackage{amssymb}
        \usetikzlibrary{decorations.pathreplacing}

\DeclareMathAlphabet{\pazocal}{OMS}{zplm}{m}{n}

\numberwithin{equation}{section}

\newcommand{\Range}{{\operatorname{Range}}}

\newcommand{\JJ}{{\pazocal{J}}}

\newcommand{\R}{\mathbb{R}}
\newcommand{\N}{\mathbb{N}}

\newcommand{\p}{{\partial}}
\newcommand{\olambda}{{\overline{\lambda}}}

\newcommand{\dd}[2]{\dfrac{\partial #1}{\partial #2}}
\newcommand{\matrice}[1]{\left[ \begin{matrix}
#1
\end{matrix} \right]}

\newcommand{\DD}{{\pazocal{D}}}

\newcommand{\loc}{{\text{loc}}}
\newcommand{\epsi}{\varepsilon}
\newcommand{\sgn}{{\operatorname{sgn}}}
 
\newcommand{\Tr}{{\operatorname{Tr}}}

\newcommand{\lr}[1]{\langle #1 \rangle}
\newcommand{\blr}[1]{\left\langle #1 \right\rangle}

\newcommand{\FF}{\pazocal{F}}

\newcommand{\tchi}{{\tilde{\chi}}}
\newcommand{\GG}{\pazocal{G}}

\newcommand{\supp}{\mathrm{supp}}
\newcommand{\trho}{\tilde{\rho}}

\newcommand{\tP}{\widetilde{P}}

\newcommand{\Z}{\mathbb{Z}}
\newcommand{\Dd}{\mathbb{D}}

\newcommand{\TT}{\pazocal{T}}

\newcommand{\Pp}{\mathbb{P}}
\newcommand{\tPp}{\widetilde{\mathbb{P}}}
\newcommand{\tQq}{\widetilde{\mathbb{Q}}}
\newcommand{\Ww}{\mathbb{W}}

\newcommand{\Id}{{\operatorname{Id}}}
\newcommand{\BB}{\mathcal{B}}

\newcommand{\EE}{\pazocal{E}}

\newcommand{\WF}{{\operatorname{WF}}}

\newcommand{\II}{\pazocal{I}}

\newcommand{\LL}{{\pazocal{L}}}
\newcommand{\HH}{\pazocal{H}}
\newcommand{\SSS}{{\pazocal{S}}}
\newcommand{\ZZ}{{\pazocal{Z}}}

\newcommand{\vp}{{\varphi}}

\newcommand{\ove}[1]{{\overline{#1}}}
\newcommand{\systeme}[1]{\left\{ \begin{matrix} #1 \end{matrix} \right.}

\newcommand{\dive}{{\operatorname{div}}}

\newcommand{\tvp}{{\tilde{\vp}}}

\newcommand{\Res}{{\operatorname{Res}}}

\newcommand{\C}{\mathbb{C}}

\newcommand{\az}{\alpha}

\newcommand{\tg}{\tilde{g}}

\newcommand{\Qq}{{\mathbb{Q}}}

\newcommand{\rk}{{\operatorname{rk}}}
\newcommand{\Op}{{\operatorname{Op}}}

\renewcommand{\Re}{\operatorname{Re}}

\newcommand{\ess}{{\operatorname{ess}}}
\renewcommand{\Im}{\operatorname{Im}}

\newcommand{\tE}{\widetilde{E}}

\newcommand{\Tt}{{\mathbb{T}}}

\newcommand{\1}{\mathds{1}}

\newcommand{\de}{ \ \mathrel{\stackrel{\makebox[0pt]{\mbox{\normalfont\tiny def}}}{=}} \ }

\title[Microlocal analysis of the bulk-edge correspondence]{Microlocal analysis of the bulk-edge correspondence}

\author{Alexis Drouot}

\NewEnviron{equations}{%
\begin{equation}\begin{gathered}
  \BODY
\end{gathered}\end{equation}
}
\newtheorem{thm}{Theorem}
\newtheorem*{conj}{Conjecture}

\newtheorem{lem}{Lemma}[section]

\newtheorem{theorem}[thm]{Theorem}

\begin{document}

\vspace*{-5mm}

\maketitle

\begin{abstract}
The bulk-edge correspondence predicts that interfaces between topological insulators support robust currents. We prove this principle for PDEs that are periodic away from an interface. Our approach relies on semiclassical methods. It suggests novel perspectives for  the analysis of topologically protected transport.  
\end{abstract}

\section{Introduction}

In solid state physics, perfect insulators are modeled by periodic operators with spectral gaps. When two sufficiently distinct insulators are glued (imperfectly) along an edge, robust currents propagate along the interface. Strikingly, the existence of such currents depends on the bulk structure rather than on the nature of the interface. 

This phenomenon is called the bulk-edge correspondence. It is a universal principle that reaches beyond electronics, for instance in accoustics \cite{Yea:15}, photonics \cite{HR:07,RH:08}, fluid mechanics \cite{DMV:17,PDV:19} and molecular physics \cite{F:19}. While bulk and edge indices were introduced as early as \cite{Ha:82,TKNN:82,BES:94}, the mathematical formulation of the bulk-edge correspondence started with \cite{Ha:93}. It has been the object of various improvements, covering Landau Hamiltonians \cite{KRS:02,EG:02,KS:04a,KS:04b}, strong disorder \cite{EGS:05,GS:18,Ta:14}, $\Z_2$-topological insulators \cite{GP:13,ASV:13}, K-theoretic aspects \cite{BKR:17,Ku:17,BR:18,Br:19} and periodic forcing \cite{GT:18,ST:19}.

In this work, we derive the bulk-edge correspondence for PDEs that are periodic away from the interface. The most important characteristics of our approach is the use of microlocal techniques in a field traditionally dominated by K-theory and functional analysis.  It opens two promising perspectives:
\begin{itemize}
\item The quantitative analysis of topologically protected transport;
\item The geometric calculation of bulk/edge indices in terms of eigenvalue crossings.
\end{itemize}

\subsection{Setting and main result} We study the Schr\"odinger evolution of electrons in a two-dimensional material, $i\p_t \psi = P\psi$. The Hamiltonian $P$ is an elliptic selfadjoint second order  differential operator on $L^2(\R^2)$:
\begin{equation}\label{eq:5h}
P \de \sum_{|\az| \leq 2} a_\az(x) D_x^\az, \ \ \ a_\az(x)  \in C^\infty_b(\R^2,\C), \ \ \ D_x \de \dfrac{1}{i} \dd{}{x}, \ \ \ \az = (\az_1,\az_2) \in \N^2.
\end{equation}
The class \eqref{eq:5h} covers for instance Schr\"odinger operators with a potential $V(x) \in C^\infty_b(\R^2,\R)$ and a (transverse) magnetic field $\p_{x_1} A_2 - \p_{x_2} A_1$, where $A(x) \in C_b^\infty(\R^2,\R^2)$: 
\begin{equation}\label{eq:1t}
-\big(\nabla_{\R^2} + iA(x)\big)^2 + V(x).
\end{equation}
It also includes the stationary form of the wave equation
that appears in photonics and meta-material realizations of topological insulators \cite{HR:07,RH:08,Kea:13,LWZ:18}: 
\begin{equation}\label{eq:0l}
-\dive_{\R^2}\big( \sigma(x) \cdot \nabla_{\R^2} \big), \ \ \ \ \sigma(x) \in C^\infty_b\big(\R^2,M_2(\C)\big) \ \
 \text{Hermitian-valued.}
\end{equation}

In relation with solid state physics, we assume that for some $L > 0$, the coefficients $a_\az(x)$ behave like $\Z^2$-periodic functions in the bulk regions $x_2 \geq L$ and $x_2 \leq -L$, see \eqref{eq:5i}. The periodic structures above and below the strip $|x_2| \leq L$ may be different.
Hence, $P$ represents the junction along $|x_2| \leq L$ of two (potentially distinct) perfect crystals, respectively modeled by $\Z^2$-periodic operators
\begin{equation}
P_+ \de \sum_{|\az| \leq 2} a_{\az,+}(x) D_x^\az, \ \ \ \ 
P_- \de \sum_{|\az| \leq 2} a_{\az,-}(x) D_x^\az, \ \ \ \ a_{\az,\pm}(x) \in C^\infty_b(\R^2,\C) \ \ \ \Z^2\text{-periodic.}
\end{equation}
The class \eqref{eq:5h} inherently covers bulk materials with non-squared periodicity, see \S\ref{sec:21}. We refer to Figure \ref{fig:1} for a pictorial representation of $P$.

In the bulk ($|x_2| \geq L$), $P_+$ and $P_-$ govern the quantum dynamics. We assume that this region is insulating at energy $\lambda_0$. Mathematically, this means 
\begin{equation}\label{eq:5j}
\lambda_0 \notin \sigma_{L^2(\R^2)}(P_+) \ \textstyle{\bigcup}  \ \sigma_{L^2(\R^2)}(P_-).
\end{equation} 
In other words, $P_+$ and $P_-$ do not allow for plane wave-like propagation at energy $\lambda_0$. Thanks to periodicity and to \eqref{eq:5j}, the generalized eigenspace of $P_+$ with energy below $\lambda_0$ induces a (Bloch) vector bundle $\EE_+$ over the $2$-torus $(\Tt^2)^* = \R^2/(2\pi\Z)^2$, see \S\ref{sec:41}. The Chern integer $c_1(\EE_+)$ is a topological invariant of $\EE_+$, associated to $P_+$.  One defines similarly $\EE_-$ and $c_1(\EE_-)$ associated to $P_-$.

While $P$ behaves like an insulator at energy $\lambda_0$ in the bulk $|x_2| \geq L$, it may still support currents along the strip $|x_2| \leq L$. Following \cite{KRS:02,EG:02,EGS:05}, we define the interface conductivity as \vspace*{-1mm}
\begin{equation}\label{eq:5t}
\II_e(P) \ \de \ \Tr_{L^2(\R^2)} \Big( i\big[ P, f(x_1) \big] \cdot g'(P) \Big), \ \  
\end{equation}
where $f \in C^\infty(\R,\R)$ and $g \in C^\infty(\R,\R)$ are such that
\begin{equation}\label{eq:0j}
f(x_1) = \systeme{1 & \text{ for }  x_1 \geq \ell \ \  \\ 0 & \text{ for } x_1 \leq -\ell}, \ \ \ \ g(\lambda) = \systeme{1 \ \ \  \text{ for } \lambda \leq \lambda_0-\epsilon_0 \\ 0 \ \ \ \text{ for } \lambda \geq \lambda_0+\epsilon_0}.
\end{equation}
In \eqref{eq:0j}, $\ell$ is an arbitrary positive number; and $\epsilon_0$ is any positive number  such that $[\lambda_0-2\epsilon_0,\lambda_0+2\epsilon_0]$ does not intersect $\sigma_{L^2(\R^2)}(P_-) \cup \sigma_{L^2(\R^2)}(P_+)$. 

Since $-g'(\lambda)$ is a probability density, $-g'(P)$ is a density of states with energy near $\lambda_0$. The operator $e^{itP} \cdot  i [P,f(x_1)] \cdot e^{-itP} = \p_t e^{itP} f(x_1) e^{-itP}$ measures the quantum flux of charges from $\{ f(x_1) = 0\}$ to $\{f(x_1) = 1\}$: the current moving left to right. Hence, \eqref{eq:5t} represents a density of current per unit energy (near $\lambda_0$). In analogy with Ohm's law, $\II_e(P)$ is a (quantum) conductivity. See Figure \ref{fig:2}; and \S\ref{sec:22} for properties of $\II_e(P)$. 

Our main result is:

\begin{theorem}\label{thm:0} Let $P$ be an elliptic selfadjoint operator of the form \eqref{eq:5h} equal to $P_+$ for $x_2 \geq L$ and $P_-$ for $x_2 \leq -L$. If $\lambda_0$ satisfies \eqref{eq:5j} then
\begin{equation}\label{eq:5l}
2\pi \cdot \II_e(P) = c_1(\EE_+) - c_1(\EE_-).
\end{equation}
\end{theorem}

   \begin{figure}
{\begin{tikzpicture}
\node at (0,0) {\includegraphics[height=2in]{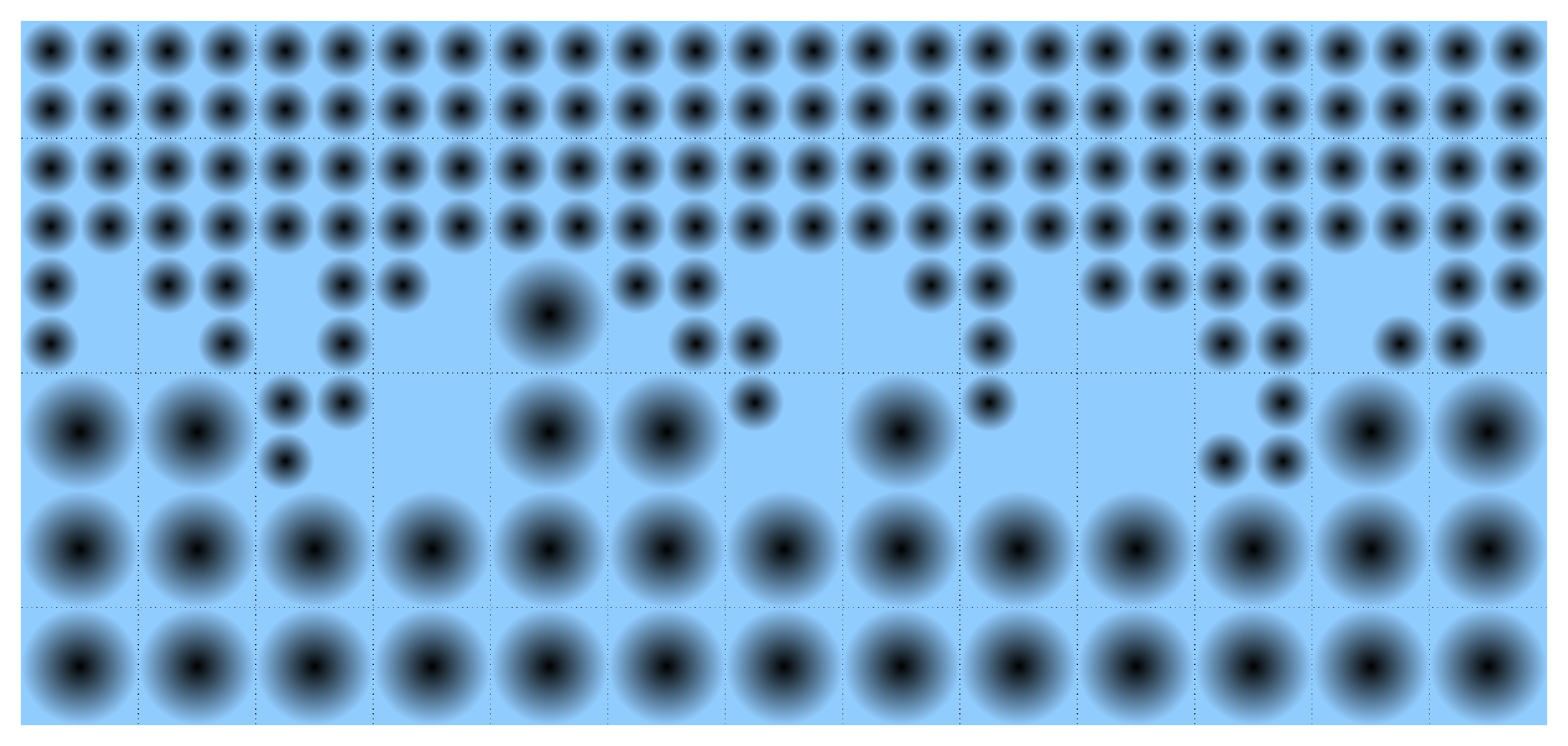}};
\draw[ultra thick,red,dashed] (-5.4,.8) -- (5.4,.8);
\draw[ultra thick,red,dashed] (-5.4,-.8) -- (5.4,-.8);
\node[red]  at (6.2,.8) {$x_2 = L$};
\node[red] at (6.4,-.8) {$x_2 = -L$};


\draw[decoration={brace,mirror,raise=5pt},decorate,thick]
  (5.2,.9) -- node[black,right=8pt] {$\text{bulk}$} (5.2,2.4);
\draw[decoration={brace,mirror,raise=5pt},decorate,thick]
  (5.2,-2.4) -- node[black,right=8pt] {$\text{bulk}$} (5.2,-.9);
\draw[decoration={brace,mirror,raise=5pt},decorate,thick]
  (-5.2,.7) -- node[black,left=8pt] {$\text{interface}$} (-5.2,-.7);

\draw[thick,->] (7.8-15,-.1+1.8-.8+.2) -- (7.8-15,1.2+1.8-.8+.2);
\draw[thick,->] (5.9+1.8-15,0+1.8-.8+.2) -- (7.2+1.8-15,0+1.8-.8+.2);

\node at (5.6+1.8-15,1.2+1.8-.8+.2) {$x_2$};
\node at (7.2+1.8-15,-.3+.5+1.8-.8+.2) {$x_1$};

   \end{tikzpicture}}
     \caption{\label{fig:1} Pictorial representation of a material covered in this work. A horizontal interface, $ |x_2| \leq L$, with arbitrary analytic structure, separates two distinct periodic medias (bulk), $x_2 \geq L$ and $x_2 \leq -L$. }
  \end{figure}
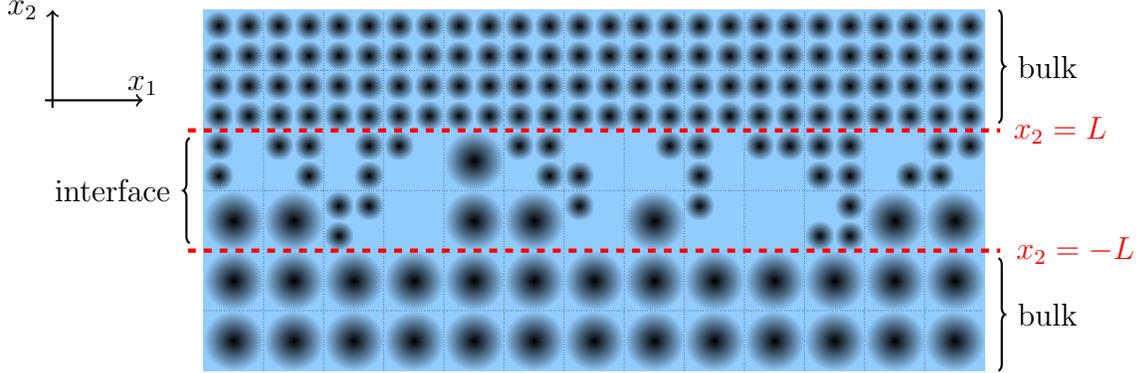

 Theorem \ref{thm:0} is the bulk-edge correspondence: the bulk and  edge indices are equal.
 Theorem \ref{thm:0} is non-trivial for systems with broken time-reversal symmetry (TRS):
$P \neq \ove{P}$. These include, for instance, Schr\"odinger operators with magnetic fields \eqref{eq:1t} and in meta-materials \eqref{eq:0l}. The formula \eqref{eq:5l} implies both quantization and topological robustness of $\II_e(P)$. Indeed, \eqref{eq:5l} shows that $2\pi \cdot \II_e(P) \in \Z$; and that an (even large) compact perturbation of $P$  preserves $P_+$ and $P_-$, hence the bundles $\EE_+$ and $\EE_-$ as well as the indices $c_1(\EE_+)$, $c_1(\EE_-)$ and $\II_e(P)$.

When $c_1(\EE_+) \neq c_1(\EE_-)$, Theorem \ref{thm:0} shows that $\II_e(P) \neq 0$; hence $g'(P) \neq 0$. A consequence is $\lambda_0 \in \sigma_{L^2(\R^2)}(P)$. Physically speaking,  the junction of two topologically distinct insulators must always be a conductor.

\subsection{Strategy} The proof  of Theorem \ref{thm:0} derives \eqref{eq:5l} starting from the formula \eqref{eq:5t} for $\II_e(P)$. At the most conceptual level, our inspiration comes from Fedosov's proof of the index theorem \cite{Fe:70} -- see also H\"ormander's account \cite[\S19.3]{Ho:85}. 

While our main result is not semiclassical (there is no asymptotic parameter $h \rightarrow 0$ in Theorem \ref{thm:0}), a key step of the proof is to deform $P$ to a semiclassical operator. Specifically, we construct in \S\ref{sec:24} an $h$-dependent operator
\begin{equation}\label{eq:5z}
P_h \de \sum_{|\az| \leq 2} c_\az(hx,x) D_x^\az \ : \ L^2(\R^2) \rightarrow L^2(\R^2)
\end{equation}
which is equal to $P$ when $h|x_2| \geq 1$, but whose coefficients admit a two-scale structure: $c_\az(x,y) \in C^\infty(\R^2 \times \Tt^2)$, $\Tt^2 = \R^2/\Z^2$. 
Lemma \ref{lem:1f} states that $\II_e(P)$ depends only on $P_+$ and $P_-$; in particular, $\II_e(P) = \II_e(P_h)$. The scaling \eqref{eq:5z} is semiclassical (in an unusual sense): if $U(x,y) \in C^\infty(\R^2 \times \Tt^2)$ and $u(x) = U(hx,x)$,
\begin{equation} \label{eq:1a}
P_h u(x) = \big(\Pp_h U\big)(hx,x) \ \ \ \ \text{where} \ \ \ \ 
\Pp_h \de \sum_{|\az| \leq 2} c_\az(x,y) (D_y+hD_x)^\az.
\end{equation}
The (leading) semiclassical symbol of $\Pp_h$ in $x$ is operator-valued. It equals
\begin{equation}\label{eq:1b}
\Pp(x,\xi) \de \sum_{|\az| \leq 2} c_\az(x,y) (D_y+\xi)^\az \ : \ L^2(\Tt^2) \ \rightarrow \ L^2(\Tt^2).
\end{equation}
We emphasize that $\II_e(P) = \II_e(P_h)$. Therefore, Theorem \ref{thm:0} reduces to a formula for a (semiclassically scaled) index $\II_e(P_h)$. This enables us to give a semiclassical proof.

\begin{figure}
\caption{The conductivity  $\II_e(P)$ measures the flux of particles of energy $\sim \lambda_0$, moving from $\{f(x_1)=0\}$ (left) to $\{f(x_1) = 1\}$ (right), along the interface, $|x_2| \leq L$. See Figure \ref{fig:1} for bulk and interface regions.}\label{fig:2}
{\hspace*{-2mm}\begin{tikzpicture}
\node at (0,0) {\includegraphics[height=2in]{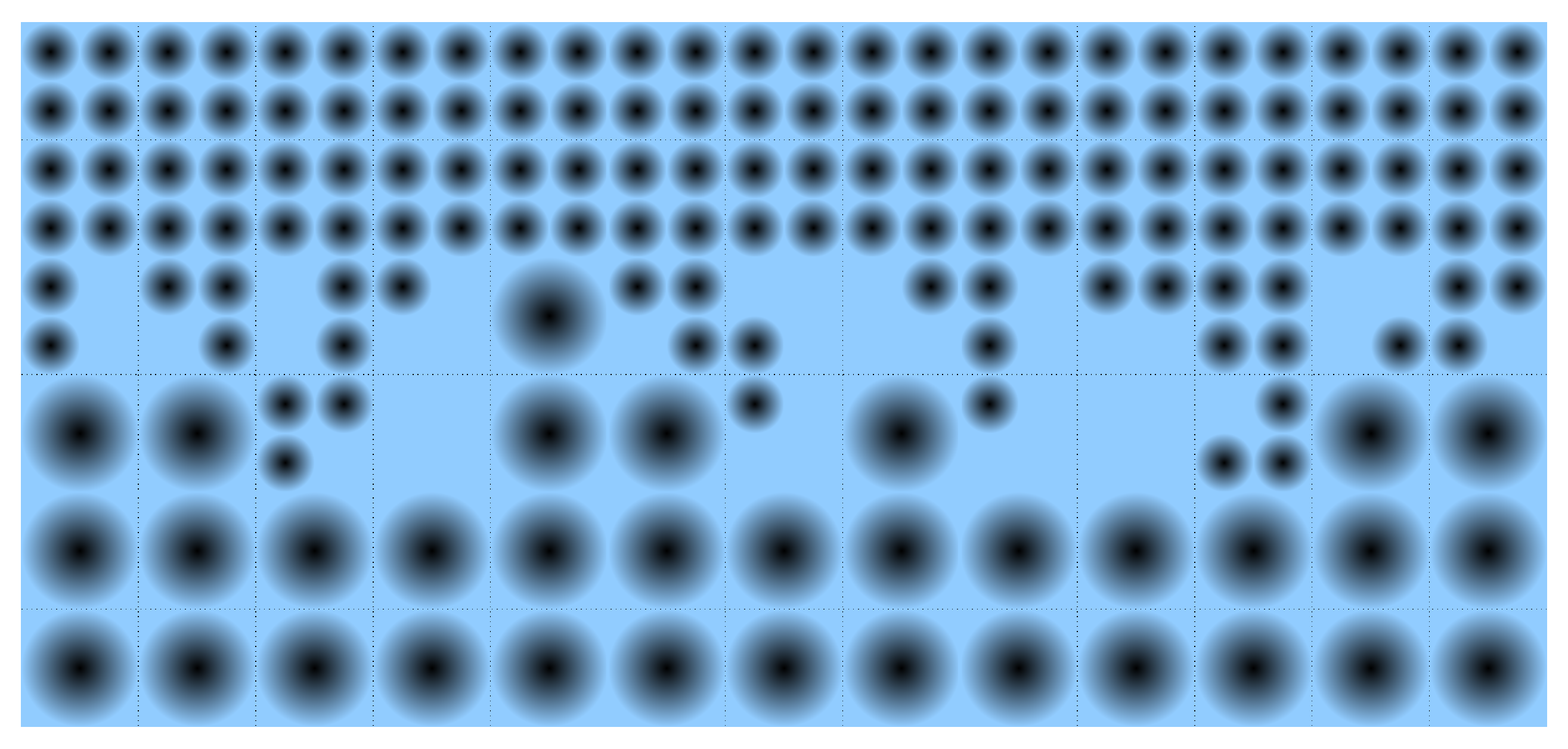}};

\draw[ultra thick, dashed] (-1.2,-2.5) -- (-1.2,2.7);
\draw[ultra thick, dashed] (1.2,-2.5) -- (1.2,2.7);

\node at (1.4,-2.7) {$x_1=\ell$};
\node at (-1.5,-2.7) {$x_1=-\ell$};

\draw[ultra thick, red, ->] (-2.3,0) -- (2.3,0);


\draw[decoration={brace,mirror,raise=5pt},decorate,thick]
  (5.2,2.4) -- node[black,above=6pt] {$\{f(x_1)=1\}$} (1.3,2.4);

\draw[decoration={brace,mirror,raise=5pt},decorate,thick]
  (-1.3,2.4) -- node[black,above=6pt] {$\{f(x_1)=0\}$} (-5.2,2.4);

\draw[ultra thick,red,dashed] (-5.4,.8) -- (5.3,.8);
\draw[ultra thick,red,dashed] (-5.4,-.8) -- (5.3,-.8);


   \end{tikzpicture}}
\end{figure}
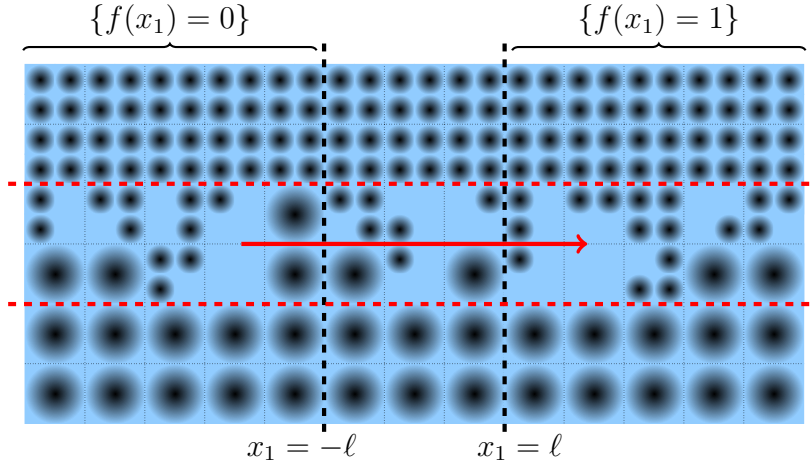

We rely on the spectral theory of two-scale operators \eqref{eq:5z}, developed in G\'erard--Martinez--Sj\"ostrand \cite{GMS:91}. It constructs a space $\HH_1 \subset \SSS'(\R^2 \times \Tt^2)$, isomorphic to $L^2(\R^2)$, such that $P_h$ on $L^2(\R^2)$ and $\Pp_h$ on $\HH_1$ are conjugated. Roughly speaking, distributions in $\HH_1$ are -- up to normalization -- $L^2(\R^2)$-multiples of the Dirac mass on 
\begin{equation}
\left\{ (x,y) \in \R^2 \times \Tt^2 :\  x = h(y+m) , \  m \in \Z^2\right\}.
\end{equation}
See \S\ref{sec:31}.  The equivalence between $P_h$ and $\Pp_h$ realizes $\II_e(P)$ as a semiclassical trace: we will see that
\begin{equations}\label{eq:0m}
\II_e(P) = \II_e(P_h) = \Tr_{L^2(\R^2)} \Big( \big[ P_h, f(hx_1) \big] g'(P_h) \Big)
\\
 = \Tr_{\HH_1} \Big( \big[ \Pp_h, f(x_1) \big] g'(\Pp_h) \Big) = \II_e(\Pp_h). 
\end{equations}

Semiclassical trace expansions have a long history. The most celebrated example is the semiclassical Weyl law, see e.g. \cite[\S9]{DS:99}, \cite[\S14]{Zw:12} and references given there. For the present work, the most relevant papers are due to Dimassi, Zerzeri and Duong \cite{Di:93,DZ:03,DD:14}. For instance, \cite{Di:93} shows that if $\vp \in C^\infty_0(\R)$ satisfies
\begin{equation}
\supp(\vp) \  \text{$\textstyle{\bigcap}$}  \bigcup_{\substack{|x| \geq M , \ \xi \in \R^2 }} \sigma_{L^2(\Tt^2)}\big( \Pp(x,\xi)  \big) \  = \ \emptyset
\end{equation}
for some $M>0$, 
then $\vp(P_h)$ is trace-class on $L^2(\R^2)$ and as $h \rightarrow 0$,
\begin{equation}\label{eq:0k}
\Tr_{L^2(\R^2)} \big( \vp(P_h) \big) \ \sim \ \sum_{j \geq 0} b_j \cdot h^{j-2}, \ \ \ \ b_0 =  \int_{\R^2 \times (\Tt^2)^*} \Tr_{L^2(\Tt^2)} \big( \vp\big( \Pp(x,\xi) \big) \ \dfrac{dxd\xi}{(2\pi)^2}.
\end{equation}

Since $P_h$ on $L^2(\R^2)$ and $\Pp_h$ on $\HH_1$ are conjugated, \eqref{eq:0k} also holds for $\Tr_{\HH_1} \big( \vp(\Pp_h) \big)$. Hence, Dimassi's result suggest that
\begin{equation}\label{eq:1x}
\II_e(\Pp_h) \ \sim \  \sum_{j=0}^\infty a_j \cdot h^{j-2}, \ \ \ \ h \rightarrow 0.
\end{equation}
But $\II_e(\Pp_h) = \II_e(P)$ does not depend on $h$. Thus, if \eqref{eq:1x} holds, then $a_j = 0$ for all $j \neq 2$; and $a_2 = \II_e(P)$. The framework of G\'erard--Martinez--Sj\"ostrand \cite{GMS:91} reduces $\Pp_h$ to a discrete  effective Hamiltonian $E_{22}(\lambda)$, realized as a pseudodifferential operator with matrix-valued symbol, $E_{22}(x,\xi;\lambda)$. The singular values of $E_{22}(\lambda)$ describe the spectral aspects of $\Pp_h$ relevant in the computation of $\II_e(\Pp_h)$. 

The bulk-edge correspondence for discrete Hamiltonians has been analyzed in \cite{KRS:02,EG:02,EGS:05}. Here we need an approach adapted to both our microlocal framework and to singular value problems. Section \ref{sec:33} extends arguments from \cite{EGS:05} to singular (versus eigenvalue) problems. We will eventually express $\II_e(P)$ in terms of the asymptotics $E_\pm(\xi;\lambda)$ of $E_{22}(x,\xi;\lambda)$:
\begin{equation}
\II_e(P) = \JJ(E_+) - \JJ(E_-),
\end{equation}
and provide an explicit formula for $\JJ(E_+)$ and $\JJ(E_-)$ -- see \eqref{eq:5d} below.

Without further considerations, recovering the Chern integers $c_1(\EE_\pm)$ from $\JJ(E_\pm)$ is technically difficult. In \S\ref{sec:4}, we design a specific effective Hamiltonian which tremendously simplifies the calculation. This completes the proof of Theorem \ref{thm:0}.

\subsection{Relation to earlier work} The bulk-edge correspondence has been intensely studied in relation to the integer quantum Hall effect, where the magnetic field is constant. We refer to \cite{KRS:02,EG:02,EGS:05} for discrete models; \cite{KS:04a,KS:04b} for K-theory proofs in the continuum; \cite{CG:05} for properties of the edge index; and \cite{Ta:14} for results covering strong disorder.

The analysis on continuous (versus discrete) models is more intricate. One reason is that after increasing the number of degrees of freedom, discrete periodic Hamiltonians are asymptotically constant. 

More importantly, there is a subtle situation that can only happen in the continuous setting. For discrete Hamiltonians on $\ell^2(\Z^2,\C^d)$, the Bloch eigenbundle below sufficiently high energy is necessarily trivial: it is simply $(\Tt^2)^* \times \C^d$. However, for continuous Hamiltonians, there may not exist $\lambda_2 > \lambda_0$ with $\lambda_2 \notin \sigma_{L^2(\R^2)}(P_+)$ such that the Bloch eigenbundle below energy $\lambda_2$ has trivial topology. 

This generates major difficulties in the proof. In particular, while we reduce our continuous Hamiltonian to a discrete system, the relevant spectral quantities become singular (characteristic) values rather than eigenvalues. In the discrete case, the Hamiltonian is already in reduced form; those singular values are eigenvalues; and our proof considerably simplifies. In that case it would resembles that of \cite{EGS:05}.

While most of the aforementioned bulk-edge correspondence works rely on functional analysis or K-theory, our approach to Theorem \ref{thm:0} is fully based on PDE techniques. It suggests a microlocal conjecture related to the quantitative aspects of transport in topological systems. We refer to \S\ref{sec:14} for the corresponding discussion.

For the sake of simplicity, Theorem \ref{thm:0} focuses on second-order operators. The microlocal aspects of the proof generalize to all elliptic pseudodifferential operators whose spectrum is bounded below. The present work relies on the effective Hamiltonian theory developed by G\'erard--Martinez--Sj\"ostrand \cite{GMS:91}. This paper also covers constant magnetic fields: $A(x) = [B_1 x_2, B_2 x_1]^\top$ in \eqref{eq:1t}. Consequently, we expect that our proof extends to this case.

\subsection{Perspectives}\label{sec:14} Our microlocal framework suggests exciting perspectives. In the discussion below, we assume that $\rk(\EE_+) = \rk(\EE_-) = n$.

 Let $\Pp_h$ be the semiclassical Hamiltonian \eqref{eq:1a} with symbol $\Pp(x,\xi) : L^2(\Tt^2) \rightarrow L^2(\Tt^2)$, see \eqref{eq:1b}. It does not depend on $x_1$, see \S\ref{sec:24}. Below, we emphasize this independence via the notations $\Pp(x_2,\xi) = \Pp(x,\xi)$ and $c_\az(x_2,y) = c_\az(x,y)$. The family $\big\{ \Pp(x_2,\xi) : \xi \in [0,2\pi]^2 \big\}$ forms the Floquet  decomposition (see e.g. \cite[\S16]{RS:78}) of
\begin{equation}
\Pp(x_2) \de 
\sum_{|\az| \leq 2} c_\az(x_2,y) D_y^\az \ : \ L^2(\R^2) \ \rightarrow \ L^2(\R^2).
\end{equation}

For $x_2$ sufficiently positive, $\Pp(x_2) = P_+$ hence $\lambda_0 \notin \sigma_{L^2(\R^2)} \big( \Pp(x_2) \big)$. Since $\rk(\EE_+) = n$, $\lambda_0$ lies precisely in the $n$-th gap of $\Pp(x_2)$, which is a fortiori open. Let $\lambda_1(x_2,\xi) \leq \dots \leq \lambda_j(x_2,\xi) \leq \dots$ be the eigenvalues of $\Pp(x_2,\xi)$. 
Suppose for a moment that 
\begin{equation}\label{eq:5x}
\forall (x_2,\xi) \in \R^2 \times [0,2\pi]^2, \ \ \ \ \lambda_n(x_2,\xi) < \lambda_{n+1}(x_2,\xi),
\end{equation}
Note that the condition \eqref{eq:5x} is weaker than assuming that $\Pp(x_2)$ has an open $n$-th $L^2(\R^2)$-gap for every $x_2 \in \R$. Yet, it allows for the construction of a smooth, $x_2$-parametrized family of (Bloch) bundles $\EE(x_2) \rightarrow (\Tt^2)^* = \R^2/(2\pi\Z)^2$, whose fibers at $\xi \in (\Tt^2)^*$ are the first $n$ eigenspaces of $\Pp(x_2,\xi)$. These bundles depend continuously on $x_2$. For $x_2$ sufficiently positive, $\EE(x_2) = \EE_+$. We deduce that
\begin{equations}
n = \rk(\EE_+) = \lim_{n\rightarrow +\infty} \rk\big(\EE(x_2)\big) = \lim_{n\rightarrow -\infty} \rk\big(\EE(x_2)\big).
\end{equations}
The assumption $\rk(\EE_-) = \rk(\EE_+) = n$ ensures that $\EE(x_2) = \EE_-$ for $x_2$ sufficiently negative. Therefore the family of bundles $\EE(x_2)$ interpolates smoothly between $\EE_+$ and $\EE_-$ as $x_2$ runs through $\R$. We deduce that when \eqref{eq:5x} holds, $c_1(\EE_+) = c_1(\EE_-)$. Theorem \ref{thm:0} yields $\II_e(P) = 0$.

In other words, if $\II_e(P) \neq 0$, then the $n$-th and $n+1$-th dispersion surfaces must intersect. The union of Bloch varieties $\left\{ (x,\xi;\lambda) \ : \ \lambda \in \sigma_{L^2}\big(\Pp(x_2,\xi) \big) \right\}$
must have singularities along the $n$-th and $n+1$-th dispersion surfaces.

  This observation motivates a conjecture that highlights microlocal aspects of the bulk-edge correspondence. Introduce the midpoint
\begin{equation}
\mu(x_2,\xi) \de \dfrac{\lambda_n(x_2,\xi) + \lambda_{n+1}(x_2,\xi)}{2}.
 \end{equation}
For $\delta > 0$, fix $G \in C^\infty\big(\R^2 \times (\Tt^2)^* \times \R\big)$ with
\begin{equation}
G(x,\xi;\lambda) = \systeme{1 \ \ \  \text{ for } \lambda \leq \mu(x_2,\xi)-\delta \\ 0 \ \ \ \text{ for } \lambda \geq \mu(x_2,\xi)+\delta}.
\end{equation}
Set $\Ww(x,\xi) = \p_\lambda G\big(x,\xi;\Pp(x_2,\xi)\big)$. This symbol is valued in linear operators on $L^2(\Tt^2)$; its support lies in the set
\begin{equation}
\ZZ_\delta \de 
\left\{ (x,\xi) \in \R^2 \times \R^2 \ : \ \big| \lambda_{n+1}(x_2,\xi) - \lambda_n(x_2,\xi) \big| \leq 2\delta \right\}.
\end{equation}
See Figure \ref{fig:3}. As $\delta \rightarrow 0$, $\ZZ_\delta$ converges to the set of eigenvalue crossings (singularities in the Bloch variety), $\ZZ_0 = \{ (x,\xi) : \lambda_n(x_2,\xi) = \lambda_{n+1}(x_2,\xi)\}$. Hence,
the quantization $\Ww_h$ of $\Ww(x,\xi)$ microlocalizes at arbitrarily small distance to $\ZZ_0$. It acts on the same space $\HH_1$ as $\Pp_h$. As $h \rightarrow 0$, we expect that $\Ww_h$ plays the role of $g' (\Pp_h)$ in \eqref{eq:0m}:

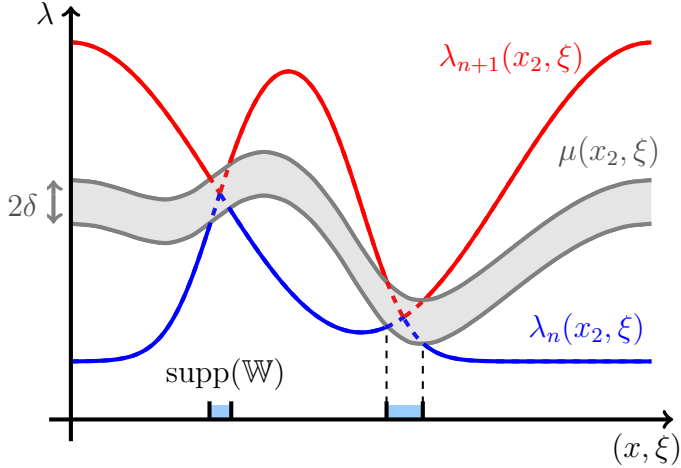
\begin{figure}
\floatbox[{\capbeside\thisfloatsetup{capbesideposition={right,center},capbesidewidth=3in}}]{figure}[\FBwidth]
{\hspace*{-5mm}\caption{Pictorial representation of dispersion surfaces, $\lambda_n$ (blue) and $\lambda_{n+1}$ (red), with midpoint area of width $2\delta$ (gray). The symbol $\Ww$ is supported where $\big|\lambda_{n+1}(x_2,\xi) - \lambda_n(x_2,\xi)\big| \leq 2\delta$, that is, where dispersion surfaces intersect the gray area.}\label{fig:3}}
{\begin{tikzpicture}

\node at (0,0) {\includegraphics[height=2.3in]{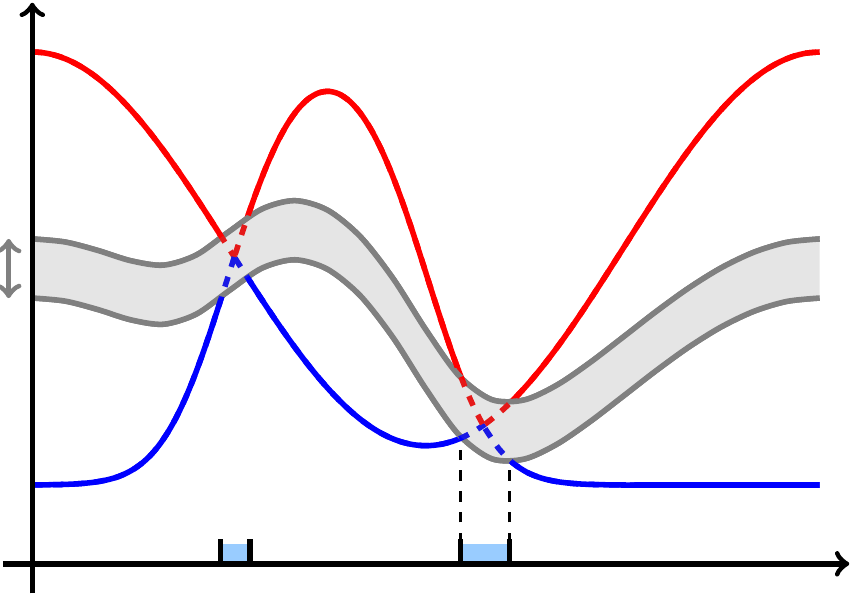}}; 

\node at (3.8,-3) {$(x,\xi)$}; 

\node at (-4.2,2.8) {$\lambda$}; 

\node at (-1.8,-2) {supp$(\mathbb{W})$}; 

\node[red]  at (2,2.2) {$\lambda_{n+1}(x_2,\xi)$}; 
\node[blue]  at (3,-1.4) {$\lambda_n(x_2,\xi)$};

\node[black!70] at (-4.5,.25) {$2\delta$};

\node[black!70]  at (3.3,.95) {$\mu(x_2,\xi)$};

   \end{tikzpicture}}
   \vspace*{-.3cm}
\end{figure}

\begin{conj}\label{conj:1} Assume $\rk(\EE_+) = \rk(\EE_-) =n$. There exists $\delta_0 > 0$ such that
\begin{equation}\label{eq:2b}
\delta \in (0,\delta_0) \ \ \ \Rightarrow \ \ \
\II_e(P) = 
\Tr_{\HH} \Big( \big[ \Pp_h, f(x_1) \big] \cdot \Ww_h \Big) + O(h^\infty).
\end{equation}
\end{conj}

Our conjecture predicts that, in the semiclassical limit, functions that are microlocalized away from the wavefront set $\WF_h(\Ww_h) \subset \ZZ_\delta$ do not contribute to $\II_e(P)$.  
It suggests that a set of edge states microlocalized on $\ZZ_0$ controls the fundamental aspects of topological transport. They should be WKB-type solutions of a normal form equation for $\Pp(x,\xi)$ near $\ZZ_0$ -- expressed as a pseudodifferential system. We refer to \cite{Bu:87,GRT:88,HS:90,DGR:02,PST:03a,PST:03b,DGR:04,FT:16} for WKB solutions in two-scale backgrounds of the form \eqref{eq:5z}. The work \cite{FT:16} also address Bloch bands with non-trivial topology, most relevant here. 

The fundamental role played by $\ZZ_0$ highlights the significance of eigenvalue crossings. When $\ZZ_0 = \emptyset$ -- equivalently, when \eqref{eq:5x} holds -- our conjecture is true. Indeed, on one hand, we saw that \eqref{eq:5x} implies $\II_e(P) = 0$; on the other hand, $\ZZ_\delta = \emptyset$ for $\delta$ sufficiently small, thus $\Ww_h = 0$ -- see Figure \ref{fig:4}.

\begin{figure}
\floatbox[{\capbeside\thisfloatsetup{capbesideposition={right,center},capbesidewidth=3in}}]{figure}[\FBwidth]
{\hspace*{-5mm}\caption{If the $n$-th and $n$-th dispersion surfaces do not intersect for all $(x,\xi)$, then for $\delta$ small enough they do not cross the midpoint area. The support of $\Ww$ is empty. }\label{fig:4}}
{\begin{tikzpicture}

\node at (0,0) {\includegraphics[height=2.3in]{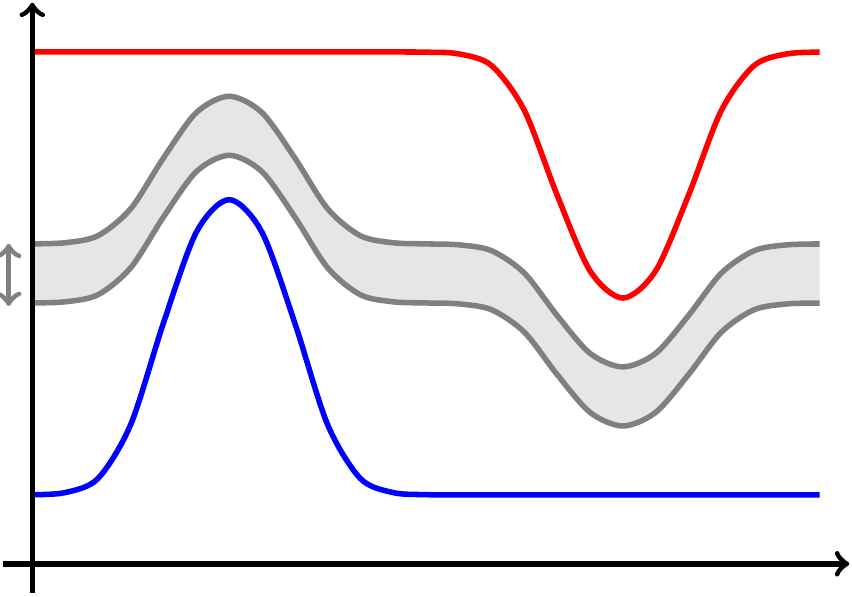}}; 


\node at (3.8,-3) {$(x,\xi)$}; 

\node at (-4.2,2.8) {$\lambda$};


\node[red]  at (3,2.75) {$\lambda_{n+1}(x_2,\xi)$}; 
\node[blue]  at (3.2,-1.6) {$\lambda_n(x_2,\xi)$};

\node[black!70]  at (3.3,.9) {$\mu(x_2,\xi)$};

\node[black!70] at (-4.5,.25) {$2\delta$}; 

\node at (-2,-2.2) {$\supp(\mathbb{W}) = \emptyset$};

   \end{tikzpicture}}
   \vspace*{-.3cm}
\end{figure}
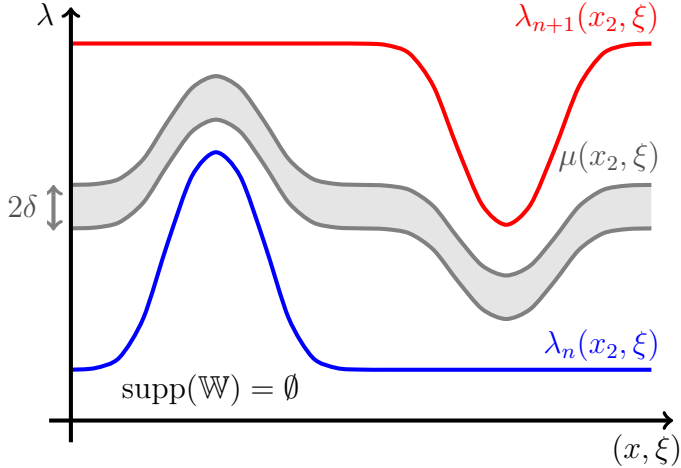

The simplest types of eigenvalue crossings are Dirac point, which typically appear in honeycomb structures \cite{FW:12}; see also \cite{BC:18,FLW:18,LWZ:18,Aea:18}. For small gap-opening perturbations and a scaling that is closer to homogenization than semiclassical analysis,  \cite{FLW:16,LWZ:18} constructed genuine edge states. They exhibit spectral concentration near the momentum associated to the conical crossing.\footnote{These are however not concentrated in position. This seems to be a feature of the homogenization -- rather than semiclassical -- scaling.} This supports the conjecture.

In \cite{Dr:19a,DW:19}, we completed the analysis via a full identification of edge states. These papers essentially provide a converse to \cite{FLW:16,LWZ:18}: all edge states are of the form derived there. This yields the explicit value of $\II_e(P)$;\footnote{In \cite{Dr:19a,DW:19}, we defined the edge index as a spectral flow. Modulo a factor $2\pi$, it equals \eqref{eq:5t} -- see e.g. \cite[Proposition 3]{ASV:13}.} see also \cite{Dr:19b} for the separate bulk index computation. For instance, if a weak magnetic field breaks TRS, 
\begin{equation}
2\pi \cdot \II_e(P) = \pm 2 = c_1(\EE_+) - c_1(\EE_-).
\end{equation}
The sign depends only on the orientation of the bulk magnetic field seen by Dirac point Bloch modes. Hence \cite{Dr:19a,Dr:19b,DW:19} supports our conjecture: bulk/edge indices can computed from a local understanding of the operator near eigenvalue crossings.

A more sophisticated case concerns eigenvalue crossings arising along a topologically non-trivial loop. For a 1D model built up from \cite{FLW:17,DFW:18}, we computed the bulk/edge indices as the winding number of a function from this loop to $\C\setminus \{0\}$ \cite{Dr:18}. This result also corroborates our conjecture.

\subsection{Notations}

\begin{itemize}
\item $\C^+$ denotes the upper half-plane $\{ \lambda \in \C : \Im \lambda > 0\}$.
\item $\Tt^2$ is the two-torus $\R^2/\Z^2$. We use the notations $\Tt^2_* = (\Tt^2)^* =  \R^2/(2\pi\Z)^2$ for its dual.
\item $C^\infty_b(\R^2)$ denotes smooth functions with bounded derivatives at any order; $\SSS(\R^2)$ denotes the Schwartz class. 
\item $P$ is the original Hamiltonian. It has bulk modeled by periodic operators $P_+$ and $P_-$ which have a spectral gap $[\lambda_0-2\epsilon,\lambda_0+2\epsilon]$ centered at $\lambda_0$, see \S\ref{sec:21}.
\item $P_h$ is a two-scale deformation of $P$ that preserves  $P_+$ and $P_-$, see \S\ref{sec:24}.
\item $\Pp_h$ is a semiclassical operator that is unitarily equivalent to $P_h$. It acts on a space $\HH^1 \subset \SSS'(\R^2 \times \Tt^2)$.  It has semiclassical symbol $\Pp(x,\xi)$ acting on $L^2(\Tt^2)$, see \S\ref{sec:32}. For $\pm x_2 \geq 1$, it equals $P_\pm(\xi)$, see \eqref{eq:1y}.
\item $Q$, $Q_h$, $\Qq_h$, $\Qq(x,\xi)$ and $\Qq_\pm(\xi)$ are respectively equal to $\psi(P)$, $\psi(P_h)$, $\psi(\Pp)$, $\psi\big(\Pp(x,\xi)\big)$ and $\psi\big(\Pp_\pm(\xi)\big)$, where $\psi$ satisfies \eqref{eq:3k}. 
\item $\II_e(P)$ is the edge index of $P$, eventually denoted $\II(P_-,P_+)$. Its definition requires two functions $f$ and $g$, see \S\ref{sec:21}. For convenience, we will use $\JJ_e(P) = -i \cdot \II_e(P)$ past \S\ref{sec:21}.
\item $\tg$ is an almost analytic extension of $g$; $\Omega$ is a bounded neighborhood of $\supp(\tg)$; and $\Omega' \subset \Omega$ satisfies \eqref{eq:1z}.
\item The real part of an operator $T$ is the selfadjoint operator $\Re(T) = \frac{T+T^*}{2}$.
\item If $u_1, \dots, u_n$ are vectors, $[u_1, \dots, u_n]$ denotes the subspace $\C u_1 \oplus \dots \oplus \C u_n$.
 \item Given an order function $m$ and $a \in S(m)$, the (classical) Weyl quantization of $a$ is $\Op(a) \in \Psi(m)$ (see \S\ref{sec:22}) and the semiclassical Weyl quantization of $a$ is $\Op_h(a) \in \Psi_h(m)$ (see \S\ref{sec:31}).
 \item We will use functional spaces $\HH_1$ and $\HH_2$ defined in \cite{GMS:91}, associated to classes of symbols $S^{(jk)}(m)$ and operators $\Psi_h^{(jk)}(m)$, see \S\ref{sec:314}.
 \item $\EE_\pm$ are vector bundles over $(\Tt^2)^*$, associated to $P_\pm$ and $\lambda_0$. They have Chern number $c_1(\EE_\pm)$, see \S\ref{sec:41}.
 \item $d^1\lambda$ denotes a one-dimensional line element in $\C$ and $d^2\lambda$ denotes the Lebesgue measure on $\C$.
 \item We use the notation $\pm$ when a statement is true for both $+$ and $-$. For instance, ``$\pm u(x) \geq 0$ or $\mp v(x) = 0$" means both ``$u(x) \geq 0$ or $-v(x) \geq 0$" and   ``$-u(x) \geq 0$ or $v(x) \geq 0$".
 \item In some statements, we use the exponent $-\infty$ to express that the statement holds for any exponent $-s$, $s > 0$.
\end{itemize}

\noindent{\bf Acknowledgments.} I am grateful to F. Faure, J. Shapiro, M. I. Weinstein and M. Zworski for valuable discussions. I thankfully acknowledge support from NSF DMS-1440140 (MSRI, Fall 2019) and DMS-1800086, and from the Simons Foundation through M. I. Weinstein's Math+X investigator award \#376319.

\section{The edge index}

We review here the properties of the edge index $\II_e(P)$,
owing to \cite{KRS:02,EG:02,EGS:05,CG:05,Ba:19a}. From Theorem \ref{thm:0}, we anticipate that $\II_e(P)$ depends only on
\begin{equation}
\Pi_\pm \de \1_{(-\infty,\lambda_0]}(P_\pm).
\end{equation}
This is supported by standard results, recalled in  \S\ref{sec:21}. We provide our own proofs in \S\ref{sec:23}, relying on pseudodifferential calculus (reviewed in \S\ref{sec:22}). They initiate the reader to the semiclassical techniques of \S\ref{sec:3}.

This independence property motivates the search for a formula expressing $\II_e(P)$ in terms of $\Pi_\pm$ only: the bulk-edge correspondence. On one hand, \eqref{eq:5t} defines $\II_e(P)$ as the trace of a classical ($h=1$) pseudodifferential operator. On the other hand, a widely developed area of spectral asymptotics expands semiclassical ($h \rightarrow 0$) traces in powers of $h$. This suggests to look for a deformation of $P$ to a semiclassical operator, that preserves the edge index. We realize this in \S\ref{sec:24}.

\subsection{Edge index}\label{sec:21} 

Let $P$ be a partial differential operator of order $2$:
\begin{equation}\label{eq:5s}
P \de \sum_{|\az| \leq 2} a_\az(x) D_x^\az, \ \ \ \ a_\az(x) \in C_b^\infty(\R^2), \ \ \  \ \text{such that:}
\end{equation}
\begin{itemize}
\item[\textbf{(a)}] $P$ is symmetric on $L^2(\R^2)$, i.e. $\lr{u,Pv}_{L^2} = \lr{Pu,v}_{L^2}$ when $u, v \in C^\infty_0(\R^2)$;\vspace*{.8mm}
\item[\textbf{(b)}] $P$ is elliptic, i.e. $\sum_{|\az|=2} \Re\big( a_\az(x) \big) \xi^\az \geq c |\xi|^2$ for some $c > 0$  and all $(x,\xi)$; \footnote{Note that from \textbf{(a)}, $a_\az(x) = \ove{a_\az(x)}$ thus the ellipticity condition is equivalent to the more standard one $\sum_{|\az|=2}  a_\az(x) \xi^\az \geq c |\xi|^2$.}
\item[\textbf{(c)}] $P$ has $\Z^2$-periodic coefficients for $|x_2| \geq L$:
\begin{equation}\label{eq:5i}
\exists a_{\az,\pm}(x) \in C^\infty_b(\R^2,\C), \ \text{ $\Z^2$-periodic with } \
a_\az(x) = \systeme{a_{\az,+}(x) & \text{ for }  x_2 \geq L \ \  \\ a_{\az,-}(x) & \text{ for } x_2 \leq -L}.
\end{equation}
\end{itemize}
Under these conditions, $P$ extends uniquely to a selfadjoint operator on $L^2(\R^2)$, with domain $H^2(\R^2)$. In the region $\pm x_2 \geq L$, $P$ is equal to the periodic operator
\begin{equation}
P_\pm \de \sum_{|\az| \leq 2} a_{\az,\pm}(x) D_x^\az \ : \ L^2(\R^2) \rightarrow L^2(\R^2). 
\end{equation}
We observe that the class of operators $P$ of the form \eqref{eq:5s}, satisfying \textbf{(a)} and \textbf{(b)} above, is invariant under linear substitutions. Such change of variables correspond to conjugating $P$ by a unitary transform. Hence the class \eqref{eq:5s} inherently models rational interfaces between materials that are (commensurately) periodic with respect to (non-necessarily squared) $\Z^2$-lattices .

In the rest of the paper, we fix $\lambda_0 \notin \sigma_{L^2(\R^2)}(P_+) \bigcup \sigma_{L^2(\R^2)}(P_-)$. The equation $(P_\pm - \lambda_0)u = 0$ has no bounded solutions. Physically, $\lambda_0$ is an insulating energy: there are no plane waves with energy $\lambda_0$ in systems modeled by $P_\pm$. 

In relation with solid state physics, the operator $P$ models the junction of two perfect insulators along an imperfect interface $|x_2| \leq L$. Even though such materials are insulating in $\pm x_2 \geq L$, they can still support currents at energy $\lambda_0$ along the interface $|x_2| \leq L$. Fix $\epsilon \in (0,1)$ with
\begin{equation}\label{eq:0a}
\sigma_{L^2(\R^2)}(P_\pm) \cap [\lambda_0-2\epsilon,\lambda_0+2\epsilon] =\emptyset
\end{equation}
and two functions $f(x_1) \in C^\infty(\R)$, $g(\lambda) \in C^\infty(\R)$ such that
\begin{equation}\label{eq:2a}
f(x_1) = \systeme{ 0 & \text{ for } x_1 \leq -\ell \\ \ell & \text{ for } x_1 \geq \ell \ \ }, \ \ \ \ g(\lambda) = \systeme{ 0 \ \text{ for } \lambda \geq \lambda_0+\epsilon \\ 1 \ \text{ for } \lambda \leq \lambda_0-\epsilon}.
\end{equation}
Following \cite{KRS:02,EG:02,EGS:05,CG:05}, 
we introduce the conductivity  at energy $\lambda_0$:
\begin{equation}\label{eq:5u}
\II_e(P) \de \Tr_{L^2(\R^2)} \big(i \big[P,f(x_1)\big] \cdot g'(P)\big).
\end{equation}

Before justifying that \eqref{eq:5u} is well-defined, we provide a physical interpretation of $\II_e(P)$ as a conductivity. Using cyclicity of the trace, we observe that for all $t \in \R$,
\begin{equation}
\II_e(P) = \Tr_{L^2(\R^2)} \Big( e^{itP}  i \big[P,f(x_1)\big] e^{-itP} \cdot g'(P)\Big) = \Tr_{L^2(\R^2)} \left(\dd{e^{itP} f(x_1) e^{-itP}}{t} \cdot g'(P)\right).
\end{equation}
From a quantum mechanics point of view:
\begin{itemize}
\item $\p_t e^{itP} f(x_1) e^{-itP}$ measures the quantum flux between $\{f(x_1) = 0\}$ and $\{f(x_1) = 1\}$, per unit time. Indeed, it is the time derivative of the Heisenberg evolution of $f(x_1)$ -- which measures the probability of a particle to sit in $\{f(x_1) = 1\}$.
\item $-g'(P)$ is a density of quantum states with energy near $\lambda_0$. Indeed, $-g'(\lambda)$ is a probability density.
\item Taking the trace corresponds to summing over all quantum states.
\end{itemize}
Therefore, $\II_e(P)$ measures the number of particles moving left to right per unit time and per unit energy (near $\lambda_0$). This is the quantum current along the interface, per unit energy. In analogy with Ohm's law, it represents the quantum conductivity at energy $\lambda_0$. This explains the physical significance of $\II_e(P)$. 
In \S\ref{sec:20}, we will also interpret $\II_e(P)$ as an algebraic number of traveling waves (with plus or minus count depending on the direction of propagation): a spectral flow -- see e.g. \cite[Proposition 3]{ASV:13}. 

The definition of $\II_e(P)$ requires a standard result:

\begin{lem}\label{lem:1c} The operator $[P,f(x_1)] g'(P)$ is trace-class on $L^2(\R^2)$. 
\end{lem}

While not immediately apparent on \eqref{eq:5u}, $\II_e(P)$ is spectacularly robust.  The first property ensures independence on $f$:

\begin{lem}\label{lem:2a} $\II_e(P)$ is independent of $f$ satisfying \eqref{eq:2a}.
\end{lem}

The second one is independence on the nature of the interface:

\begin{lem}\label{lem:1f} Let $P_1$ and $P_2$ satisfy the assumptions \textbf{(a)}, \textbf{(b)} and \textbf{(c)} above. If $P_1-P_2$ has coefficients supported in a strip $\{ |x_2| \leq L'\}$, then $\II_e(P_1) = \II_e(P_2)$.
\end{lem}

According to Lemma \ref{lem:1f}, $\II_e(P)$ depends only on $P_+$ and $P_-$. We state a last independence property that is perhaps subtler. We will use at the end of the proof, \S\ref{sec:4}. Let $\lambda_1, \lambda_2$ with  $\lambda_1 +2\epsilon \leq \lambda_0 \leq \lambda_2-2\epsilon$ and $\psi \in C^\infty(\R)$ a two-level rearrangement:
\begin{equation}\label{eq:3k}
\psi \ \ \text{nondecreasing} \ \ \text{ and } \ \ 
\psi(\lambda) = \systeme{ \lambda_1 & \ \text{for } \  \lambda \leq \lambda_0-2\epsilon \\
\lambda & \ \text{for } \ |\lambda -\lambda_0| \leq \epsilon \\
\lambda_2& \ \text{for } \ \lambda \geq \lambda_0 + 2\epsilon}.
\end{equation}
See Figure \ref{fig:5}.

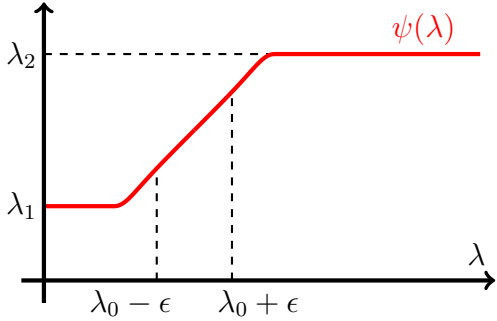
\begin{figure}
\floatbox[{\capbeside\thisfloatsetup{capbesideposition={right,center},capbesidewidth=3in}}]{figure}[\FBwidth]
{\hspace{1cm}\caption{Pictorial representation of $\psi(\lambda)$. We note that $\psi(P_+)$ and $\psi(P_-)$ have no spectrum outside $\{\lambda_1,\lambda_2\}$.}\label{fig:5}}
{\begin{tikzpicture}

\node at (0,0) {\includegraphics{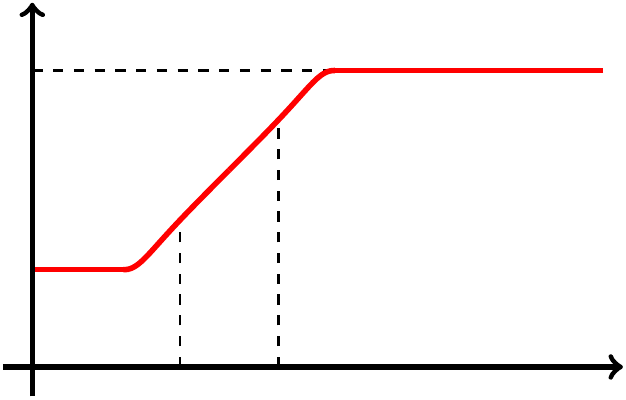}}; 

\node at (-1.7,-2) {$\lambda_0-\epsilon$}; 
\node at (0,-2) {$\lambda_0+\epsilon$}; 
\node at (2.9,-1.35) {$\lambda$};

\node at (-3.15,-.7) {$\lambda_1$}; 
\node at (-3.15,1.3) {$\lambda_2$}; 

\node[red] at (2.2,1.65) {$\psi(\lambda)$}; 

   \end{tikzpicture}}
   \vspace*{-.3cm}
\end{figure}

\begin{lem}\label{lem:1e} If $\psi$ satisfies \eqref{eq:3k} then $\big[ \psi(P), f(x_1) \big] g'\circ \psi(P)$ is trace-class and 
\begin{equation}
\II_e(P) = \Tr_{L^2(\R^2)} \big( \big[ \psi(P), f(x_1) \big] g'\circ \psi(P) \big) = \II_e\big(\psi(P)\big).
\end{equation}
\end{lem}

From Lemma \ref{lem:1e}, $\II_e(P) = \II_e\big(\psi(P)\big)$. Because of Lemma \ref{lem:1c}, it is reasonable to expect that $\II_e\big(\psi(P)\big)$ depends only on $\psi(P_\pm)$. From \eqref{eq:0a} and \eqref{eq:3k},
\begin{equation}
\psi(P_\pm) = \lambda_1 \cdot \Pi_\pm + \lambda_2 \cdot (\Id-\Pi_\pm), \ \ \ \ \Pi_\pm \de \1_{(-\infty,\lambda_0]}(P_\pm). 
\end{equation}
Since $\II_e(P)$ does not depend on $\lambda_1, \lambda_2$, we anticipate that $\II_e(P)$ depends only on~$\Pi_\pm$.

Lemma \ref{lem:1c}-\ref{lem:1e} are morally known; see e.g. \cite{KRS:02,EG:02} for Lemma \ref{lem:1c}, \cite[Theorem 1]{CG:05} for Lemma \ref{lem:1f} and \cite[Lemma A.4]{EG:02} or \cite[Lemma 4.7]{Ba:19a} for Lemma \ref{lem:1e}. We give our own proofs in \S\ref{sec:23}. They rely on pseudodifferential calculus (reviewed in \S\ref{sec:22}). They provide a good introduction to the semiclassical techniques of \S\ref{sec:3}.

\subsection{Dynamics, spectral flow and edge index}\label{sec:20} We give here an interpretation of $\II_e(P)$ as the signed number of independent elementary waves propagating along the interface $|x_2| \leq L$.

Thanks to Lemma \ref{lem:1f}, $\II_e(P)$ depends only on $P_+$ and $P_-$. After a perturbation of $P$ in the strip $\{ |x_2| \leq L\}$, we can assume here that $P$ is periodic w.r.t. $\Z e_1$. In this case, for each $\zeta \in [0,2\pi]$, $P$ acts on the space
\begin{equation}
\LL^2_\zeta \de \left\{ u \in L^2_\loc(\R^2,\C) : \ u(x+e_1) = e^{i\zeta} u(x), \ \int_{[0,1] \times \R} \big|u(x)\big|^2 dx < \infty \right\}.
\end{equation}
We denote the resulting operator by $P_\zeta$ -- the Floquet--Bloch decomposition of $P$ along $\Z e_1$. The essential spectrum of $P_\zeta$ comes from large values of $|x_2|$: we have
\begin{equation}\label{eq:0i}
\sigma_{\LL^2_\zeta,\ess}(P_\zeta) \ = \ \sigma_{\LL^2_\zeta,\ess}(P_{+,\zeta}) \ \textstyle{ \bigcup } \  \sigma_{\LL^2_\zeta,\ess}(P_{-,\zeta}) \ \subset \ \sigma_{L^2(\R^2)}(P_+) \ \textstyle{ \bigcup } \ \sigma_{L^2(\R^2)}(P_-),
\end{equation}
where $P_{\pm,\zeta}$ denote $P_\pm$ acting on $\LL^2_\zeta$. In particular, \eqref{eq:0i} shows that $P_\zeta$ has an essential spectral gap containing $\lambda_0$. 

The spectral flow of $P_\zeta$ is the algebraic number of $\LL^2_\zeta$-eigenvalues that traverse this gap as $\zeta$ sweeps $[0,2\pi]$; see \cite{W:16} for a smooth introduction and Figure \ref{fig:7} for a pictorial representation.  From \cite[Proposition 3]{ASV:13},  $2\pi \cdot \II_e(P)$ coincide with the  spectral flow of $\zeta \mapsto P_\zeta$, when $P$ is $\Z e_1$-invariant.

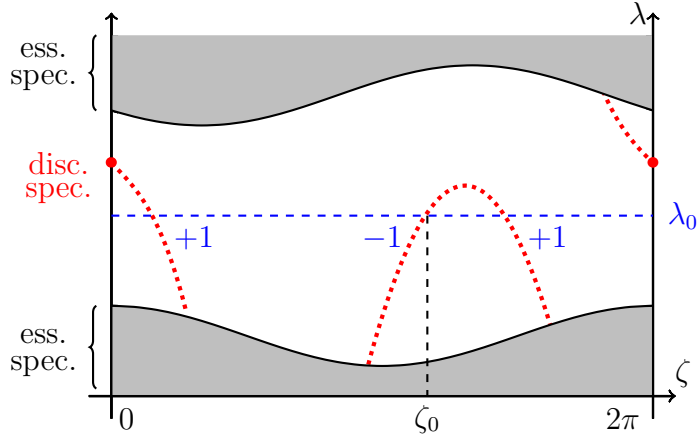
\begin{figure}
\floatbox[{\capbeside\thisfloatsetup{capbesideposition={right,center},capbesidewidth=3in}}]{figure}[\FBwidth]
{\hspace*{-8mm}\caption{Essential  (gray) and discrete (red) spectra of $P_\zeta$ as functions of $\zeta$. The spectral flow is the intersection number of the eigenvalue curves with the energy level $\lambda_0$. Here it equals $1$.}\label{fig:7}}
{\begin{tikzpicture}

\node at (0,0) {\includegraphics{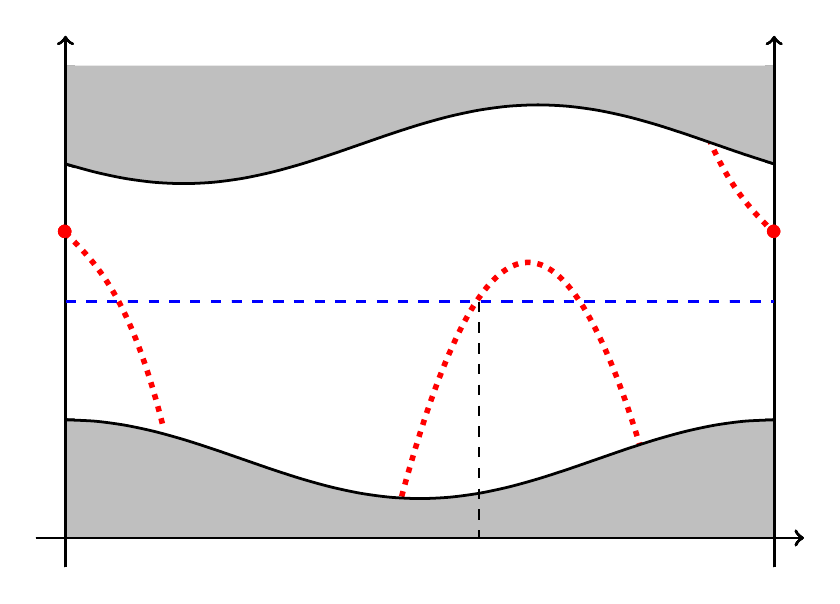}};

  \draw [decorate,decoration={brace},thick] (-3.8,3*.8-1) -- (-3.8,3*.8);
 \draw [decorate,decoration={brace},thick] (-3.8,-2.3) -- (-3.8,-3*.8+1.2);
 \node at (3.4,3*.8+.3) {$\lambda$};
   \node at (-4.5,2.2) {ess.};   
   \node at (-4.5,1.8) {spec.};   

   \node at (-4.5,-1.6) {ess.};   
   \node at (-4.5,-2) {spec.};

    \node[red] at (-4.3,.7) {disc.}; 
    \node[red] at (-4.3,.3) {spec.};

     \node at (4,-3*.8+.3) {$\zeta$};
  \node at (-3.4,-2.7) {$0$};   
\node at (3.2,-2.7) {$2\pi$};   
 
        \node[blue] at (4,.3-.3) {$\lambda_0$};
 \node at (.6,-2.7) {$\zeta_0$};
 
        \node[blue] at (-2.5,-.3) {$+1$};
 \node[blue] at (0,-.3) {$-1$};
 
 \node[blue] at (2.2,-.3) {$+1$};

   \end{tikzpicture}}
   \vspace*{-.3cm}
\end{figure}

From a dynamical point of view, a curve of simple eigenvalues $\lambda(\zeta)$ of $P_\zeta$ with $\lambda(\zeta_0) = \lambda_0$ generates a wave propagating parallel to $\R e_1$, with group velocity $\p_\zeta\lambda(\zeta_0)$. Hence, the spectral flow and $\II_e(P)$ count elementary waves at energy $\lambda_0$ that travel along the interface $|x_2| \leq L$, signed according to the direction of propagation, $\sgn \big( \p_\zeta \lambda(\zeta_0) \big)$.

\subsection{Classical pseudodifferential operators}\label{sec:22}  We review here the classical pseudodifferential calculus. The results below are exposed in Dimassi--Sj\"ostrand \cite[\S7-8]{DS:99} and Zworski \cite[\S4 and \S14]{Zw:12} (set $h=1$). 
For more advanced results, we refer to H\"ormander \cite[\S18-20]{Ho:85}.

 Given $a(x,\xi) \in C^\infty_0\big(\R^2 \times \R^2\big)$, the Weyl quantization of $a$ is defined as
\begin{equation}\label{eq:5n}
\big(\Op(a) u\big)(x) \de  \int_{\R^2 \times \R^2} e^{i\xi(x-x')} \cdot  a \left( \dfrac{x+x'}{2},\xi \right) u(x') \ \dfrac{dx'd\xi}{(2\pi)^2}, \ \ \ \ u \in C^\infty_0(\R^2).
\end{equation}
Operators of the form \eqref{eq:5n} are called pseudodifferential (Pdos); they generalize differential operators. We review here key facts on pseudodifferential calculus, with an emphasis on order functions; composition; resolvents; and trace-class properties. 

\subsubsection{Order functions} See \cite[\S4.4]{Zw:12} and \cite[\S7]{DS:99}. Let $a \in C^\infty(\R^2)$. Conditions so that \eqref{eq:5n} still defines a bounded operator on the Schwartz class  $\SSS(\R^2,\C)$ are typically encoded in order functions, i.e. functions $m(x,\xi) \in C^0(\R^2 \times \R^2)$ with
\begin{equation}
w, \ w' \in \R^2 \times \R^2 \ \ \Rightarrow  \ \ m(w) \leq C \lr{w-w'}^N m(w').
\end{equation}
Specifically, \eqref{eq:5n} defines $\Op(a)$ as a bounded operator on $\SSS(\R^2,\C)$ if for some order function $m$, for all $\az \in \N^4$ there exists $C_\az > 0$ with
\begin{equation}\label{eq:5o}
(x,\xi) \in \R^2 \times \R^2 \ \ \Rightarrow \ \
\left| \p^\az a(x,\xi) \right| \leq C_\az \cdot m(x,\xi).
\end{equation}
Symbols $a$ satisfying \eqref{eq:5o} form the class $S(m)$, naturally equipped with a Frechet space structure. We set $\Psi(m) = \Op\big(S(m)\big)$. Given an order function $m \geq 1$, we say that $a \in S(m^{-\infty})$ if for every $s \in \N$, $a \in S(m^{-s})$.

Standard examples of order functions include $1$, $\lr{x}^s$ and $\lr{\xi}^s$ for any $s \in \R$; lesser known examples are
\begin{equation}\label{eq:2f}
m_{j,\pm}(x,\xi) \de \systeme{ 1 & \text{ for } \pm x_j \geq 0 \\ \lr{x_j}^{-1}  & \text{ for } \pm x_j \leq 0 }.
\end{equation}

\subsubsection{Composition of Pdos}\label{sec:222} See \cite[\S4.4-4.5]{Zw:12} and \cite[\S7]{DS:99}. If $m_1$ and $m_2$ are order functions, then so is $m_1 m_2$. The composition of two Pdos is a Pdo:
\begin{equation}
\Op(a) \in \Psi(m_1), \ \  \Op(b)\in \Psi(m_2) \ \ \Rightarrow \ \ \Op(a)\Op(b) \in \Psi(m_1m_2).
\end{equation}
Moreover, the symbol of $\Op(a) \Op(b)$ in $S(m_1m_2)$ depends continuously on $(a,b) \in S(m_1) \times S(m_2)$.

\subsubsection{Resolvents}\label{sec:223} See \cite[\S8]{DS:99}. We now turn to resolvents. Let $P$ given by \eqref{eq:5s} be elliptic and selfadjoint. We note that $P \in \Psi(\lr{\xi}^2)$. For any $\lambda$ with $\Im \lambda > 0$, the operator $P-\lambda$ is an isomorphism from $H^2(\R^2)$ to $L^2(\R^2)$. A classical result of Beals \cite{Be:77} implies that $(P-\lambda)^{-1} \in \Psi\big( \lr{\xi}^{-2}\big)$:
\begin{equation}
\forall \lambda \in \C^+, \ \ 
\exists r(\cdot;\lambda) \in S\big(\lr{\xi}^{-2}\big), \ \ \ \ 
(P-\lambda)^{-1} = \Op\big(r(\cdot;\lambda)\big).
\end{equation}

In the proofs below, we will need uniform estimates on $r(\cdot;\lambda)$ in $S(1)$: for every $R > 0$, $\az \in \N^4$, there exists $c_{\az,R} > 0$ such that
\begin{equation}\label{eq:5r}
|\lambda| \leq R, \ \ \Im \lambda > 0 \ \ \Rightarrow \ \ \sup_{(x,\xi) \in \R^2}\left| \p^\az r(x,\xi;\lambda) \right| \leq c_{\az,R} \cdot |\Im \lambda|^{-6-|\az|}. \ \ \footnote{The power $6$ is specific to the dimension $n=2$; in general it is $2n+2$.}
\end{equation}
This shows that the constant $C_\az$ for $r(\cdot;\lambda)$ and $m=1$ in \eqref{eq:5o} blow up at worst polynomially in $|\Im \lambda|^{-1}$ when $|\lambda|$ remains bounded.

\subsubsection{Trace-class properties}\label{sec:225} See \cite[\S8]{DS:99}. Assume that $m \in L^1$. Then for any $a \in S(m)$, $\Op(a)$ extends to a trace class operator on $L^2(\R^2)$. Moreover there exists $C > 0$ independent of $a$ such that 
\begin{equation}
\big\| \Op(a) \big\|_\Tr \leq C |m|_{L^1} \cdot \sup_{|\az| \leq 5} C_\az,\footnote{The number $5$ is specific to $n=2$; in general it is $2n+1$.}
\end{equation}
where the constants $C_\az$ are those of \eqref{eq:5o}.

\subsubsection{Functional calculus}\label{sec:224} See \cite[\S3.1 and \S14.3]{Zw:12} and \cite[\S8]{DS:99}. An almost analytic extension of $\rho(\lambda) \in C^\infty_0(\R)$ is a function $\trho(\lambda) \in C^\infty_0(\C^+)$ such that
\begin{equation}
\trho\big|_\R = \rho; \ \ \ \text{and}  \ \ \ \p_\olambda \trho(\lambda) = O\left(|\Im \lambda|^\infty\right) \ \ \text{as} \ \  \Im \lambda  \rightarrow 0.
\end{equation}
Almost analytic extensions always exist. 
If $\trho$ is an almost analytic extension of $\rho$ then $\p_\lambda \trho$ is an almost analytic extension of $\rho'$. Indeed, $(\p_\olambda \trho)\big|_\R = 0$; $\p_\lambda + \p_\olambda$ is tangent to $\R$ and $(\p_\lambda + \p_\olambda) \big|_\R$ is the standard derivative; and $\trho|_\R = \rho$. Thus we have
\begin{equation}
(\p_\lambda \trho) \big|_\R = \big((\p_\lambda + \p_\olambda) \trho\big) \big|_\R = (\p_\lambda + \p_\olambda) \big|_\R  \big(\trho\big|_\R\big) = \big(\trho\big|_\R\big)' =   \rho'.
\end{equation}
Moreover, since $\p_\lambda (\Im \lambda)^s = (2i)^{-1}s \cdot (\Im \lambda)^{s-1}$, we have $\p_\olambda \p_\lambda \trho = O(|\Im \lambda|^\infty)$. This proves that $\p_\lambda \trho$ is an almost analytic extension of $\rho'$.

To emphasize that almost analytic extensions are not analytic, we use the notation $\trho(\lambda) = \trho(\lambda,\olambda)$ in the rest of the paper. A central application of almost analytic extensions is the Helffer--Sj\"ostrand formula. It asserts that for every $z \in \C$,
\begin{equation}\label{eq:2d}
\rho(z) =  \int_{\C^+} \dd{\trho(\lambda,\olambda)}{\olambda} \cdot (z-\lambda)^{-1} \cdot \dfrac{d^2\lambda}{\pi}.
\end{equation}

The identity \eqref{eq:2d} allows for a functional calculus developed in terms of resolvents. If $T$ is a (possibly unbounded) selfadjoint operator then $\big\| (T-\lambda)^{-1} \big\| \leq |\Im \lambda|^{-1}$. In particular we can express $\rho(T)$ as an absolutely convergent integral:
\begin{equation}\label{eq:2e}
\rho(T) =  \int_{\C^+} \dd{\trho(\lambda,\olambda)}{\olambda} \cdot (T-\lambda)^{-1} \cdot \dfrac{d^2\lambda}{\pi}.
\end{equation}
While the functional calculus based on \eqref{eq:2e} goes back to Dyn'kin \cite{Dy:75}, its popular use in the semiclassical literature seems to start with \cite{HS:89}; see \cite{HS:90,SZ:91,Di:93} for subsequent developments.

\subsection{Proofs of Lemma \ref{lem:1c}-\ref{lem:1e}}\label{sec:23} For convenience,  starting now we will use:
\begin{equation}
\JJ_e(P) \de -i \II_e(P) = \Tr_{L^2(\R^2)} \Big( \big[P,f(x_1)\big] g'(P) \Big).
\end{equation} 

\begin{proof}[Proof of Lemma \ref{lem:1c}] 
\textbf{1.} We need to show that $[P,f(x_1)] g'(P)$ is trace-class on $L^2(\R^2)$. Our strategy is to show that $[P,f(x_1)] g'(P)$ is a Pdo whose symbol decays sufficiently. 
We first focus on the term
\begin{equation}
\big[P,f(x_1)\big] = 
\big(1-f(x_1)\big) P f(x_1) - f(x_1) P  \big(1-f(x_1)\big).
\end{equation}
We observe that $f(x_1) \in \Psi(m_{1,+}^{\infty})$ and $1-f(x_1) \in \Psi(m_{1,-}^\infty)$, where $m_{1,\pm}$ are the order functions defined in \eqref{eq:2f}. The Weyl symbol of $P$ belongs to $S\big(\lr{\xi}^2\big)$. Since $m_{1,+}  m_{1,-}  = \lr{x_1}^{-1}$, we deduce from the composition theorem (\S\ref{sec:222}):
\begin{equation}\label{eq:2g}
\big[P,f(x_1)\big] \in \Psi\left(\lr{\xi}^2 \lr{x_1}^{-\infty}\right).
\end{equation}

\textbf{2.} Fix $s\in \N$. We focus on $g'(P)$. Let $\rho(\lambda) \in C_0^\infty(\R,\C)$ such that
\begin{equation}\label{eq:5p}
\lambda \in \sigma_{L^2(\R^2)}(P) \ \ \Rightarrow \ \ 
\rho'(\lambda) = g'(\lambda) (\lambda+i)^s.
\end{equation}
Note that $\rho$ exists: \eqref{eq:5p} specifies $\rho'$ on $\sigma_{L^2(\R^2)}(P)$, which is bounded below; and it suffices to arrange so that $\rho'$ integrates  to $0$ on $\R$. 
Let $\tchi_\pm(x_2) \in C^\infty(\R,[0,1])$ with
\begin{equation}\label{eq:4a}
\tchi_+(x_2) = \systeme{ 0 & \text{ for } x_2 \leq -1 \\ 1 & \text{ for } x_2 \geq 1 \ \ },
\ \ \ \ \tchi_- = 1-\tchi_+.
\end{equation}
Since $\tchi_+ + \tchi_-=1$, we have
\begin{equation}
g'(P) = \sum_\pm \tchi_\pm(x_2) g'(P) =  \sum_\pm \tchi_\pm(x_2) \rho'(P) (P+i)^{-s}.
\end{equation}
Moreover, $P_\pm$ has no spectrum in the support of $g'$; since 
$\sigma_{L^2(\R^2)}(P_\pm) \subset \sigma_{L^2(\R^2)}(P)$, \eqref{eq:5p} implies that $\rho'(P_\pm) = 0$ and
\begin{equation}\label{eq:2c}
g'(P) = \sum_{\pm} \tchi_\pm(x_2) \big( \rho'(P) - \rho'(P_\pm) \big) (P+i)^{-s}.
\end{equation}

Let  $\trho$ be an almost analytic extension of $\rho$.  Then $\p_\olambda \trho$ is an almost analytic extension of $\rho'$.
We write \eqref{eq:2c} using the Helffer--Sj\"ostrand formula \eqref{eq:2e}: 
\begin{equations}\label{eq:5q}
g'(P) =  \sum_{\pm} \int_{\C^+} \dd{^2\trho(\lambda,\olambda)}{\lambda \p\olambda} \cdot \tchi_\pm(x_2) \left( (P-\lambda)^{-1} - (P_\pm-\lambda)^{-1} \right) \cdot \dfrac{d^2\lambda}{\pi} \cdot (P+i)^{-s}
\\
=   \sum_{\pm} \int_{\C^+} \dd{^2 \trho(\lambda,\olambda)}{\lambda \p\olambda} \cdot \tchi_\pm(x_2)  (P-\lambda)^{-1} (P_\pm-P) (P_\pm-\lambda)^{-1}  \cdot \dfrac{d^2\lambda}{\pi} \cdot (P+i)^{-s}.
\end{equations}

\textbf{3.}  We now observe that $\tchi_\pm(x_2) \in \Psi(m_{2,\pm}^s)$ and that $P_\pm-P \in \Psi(m_{2,\mp}^s\lr{\xi}^2)$. Moreover $(P-\lambda)^{-1}$ and $(P_\pm-\lambda)^{-1}$ are in $\Psi(1)$, with symbolic bounds blowing up at worst polynomially in $|\Im \lambda|^{-1}$, see \eqref{eq:5r}. Since $\p^2_{\lambda\olambda} \trho(\lambda,\olambda) = O(|\Im \lambda|^\infty)$ and $m_{2,+} m_{2,-} = \lr{x_2}^{-1}$, we deduce from the composition theorem (\S\ref{sec:222}):
\begin{equation}
\dd{^2 \trho(\lambda,\olambda)}{\lambda \p\olambda} \cdot \tchi_\pm(x_2)  (P-\lambda)^{-1} (P_\pm-P) (P_\pm-\lambda)^{-1} \in \Psi\left(\lr{x_2}^{-s} \lr{\xi}^2\right),
\end{equation}
uniformly in $\lambda$. We integrate this identity on $\C^+$ and multiply by $(P+i)^{-s}$ (which belongs to $\Psi(\lr{\xi}^{-2s})$, see \S\ref{sec:223}). We deduce from \eqref{eq:5q}:
\begin{equation}\label{eq:2h}
\forall s \in \N, \ \ 
g'(P) \in \Psi \left(\lr{x_2}^{-s}\lr{\xi}^{2-2s} \right), \  \ \ \  \text{i.e.} \ \ g'(P) \in \Psi \left(\lr{x_2}^{-\infty}\lr{\xi}^{-\infty} \right).
\end{equation}

\textbf{4.}  We combine \eqref{eq:2g} and \eqref{eq:2h} to obtain
\begin{equation}\label{eq:2r}
[P,F(x_1)] g'(P) \ \in \ \Psi \left( \lr{x_1}^{-\infty} \lr{x_2}^{-\infty} \lr{\xi}^{-\infty} \right) = \Psi \left( \lr{x}^{-\infty} \lr{\xi}^{-\infty} \right).
\end{equation}
Hence $[P,F(x_1)] g'(P)$ is trace class, see \S\ref{sec:225}. 
\end{proof}

\begin{proof}[Proof of Lemma \ref{lem:2a}] It suffices to show that if $f_0(x_1) \in C_0^\infty(\R)$ then 
\begin{equation}
\Tr_{L^2(\R^2)} \big( [P,f_0(x_1)] g'(P) \big) = 0.
\end{equation}
We have $f_0 \in \Psi\big(\lr{x_1}^{-\infty}\big)$. Using \eqref{eq:2h}, we deduce that both $P f_0(x_1) g'(P)$ and $f_0(x_1) P g'(P)$ are in $\Psi\big(\lr{x}^{-\infty}\lr{\xi}^{-\infty}\big)$. Hence both are trace-class. In particular,
 \begin{equation}
0 = \Tr_{L^2(\R^2)} \big( P f_0(x_1)  G'(P) \big) - \Tr_{L^2(\R^2)} \big( f_0(x_1)   G'(P) P  \big) =  \Tr_{L^2(\R^2)} \big( [P, f_0(x_1) ]  G'(P) \big).
\end{equation}
This completes the proof. \end{proof}

As in \cite{CG:05}, the proof of Lemma \ref{lem:1f} requires a preliminary result.

\begin{lem}\label{lem:1d} Let $P_1$ and $P_2$ satisfying \textbf{(a)}, \textbf{(b)} and \textbf{(c)} in \S\ref{sec:21}, and such that $P_1-P_2$ vanishes outside a compact set. Then $\JJ_e(P_1) = \JJ_e(P_2)$.
\end{lem}

\begin{proof} \textbf{1.}   Let $s, \rho$ as in the proof of Lemma \ref{lem:1c} so that for $j=1,2$,
\begin{equation}\label{eq:5v}
[P_j,f(x_1)] g'(P_j) = \int_{\C^+} \dd{^2\trho(\lambda,\olambda)}{\lambda \p \olambda} \cdot [P_j,f(x_1)](P_j-\lambda)^{-1} \cdot \dfrac{d^2\lambda}{\pi} \cdot (P_j+i)^{-s}.
\end{equation}
Our goal is to write the difference of \eqref{eq:5v} for $j=1,2$ in terms of commutators of trace-class operators. This will produce a vanishing trace and complete the proof.

\textbf{2.}  We integrate \eqref{eq:5v} by parts w.r.t. $\lambda$:
\begin{equation}
[P_j,f(x_1)] g'(P_j) = - \int_{\C^+} \dd{\trho(\lambda,\olambda)}{ \olambda} \cdot \big[ P_j,f(x_1) \big](P_j-\lambda)^{-2} \cdot \dfrac{d^2\lambda}{\pi} \cdot (P_j+i)^{-s}.
\end{equation}
We permute $\big[ P_j,f(x_1) \big]$ with one of the terms $(P_j-\lambda)^{-1}$. This allows us to write
\begin{equations}\label{eq:2o}
[P_j,f(x_1)] g'(P_j) = -  \int_{\C^+}  A_j(\lambda) \cdot \dfrac{d^2\lambda}{\pi}, 
\end{equations}
$\text{where} \  A_j(\lambda)  \de \p_\olambda \trho(\lambda,\olambda) \big( B_j(\lambda) + C_j(\lambda) \big),  \text{ with:}$ \begin{align}\label{eq:2l}
B_j(\lambda) & \de  - (P_j-\lambda)^{-1}[P_j,f(x_1)](P_j-\lambda)^{-1} \cdot  (P_j+i)^{-s}
\\
& \  =  \left[(P_j-\lambda)^{-1},f(x_1)\right] \cdot  (P_j+i)^{-s} =  \left[(P_j-\lambda)^{-1},f(x_1)(P_j+i)^{-s} \right];
\\
C_j(\lambda) & \de   \left[(P_j-\lambda)^{-1}, [P_j,f(x_1)]\right] (P_j-\lambda)^{-1}\cdot (P_j+i)^{-s}
\\
& \ = \left[(P_j-\lambda)^{-1}, [P_j,f(x_1)](P_j-\lambda)^{-1}(P_j+i)^{-s}\right].
\end{align}

\textbf{3.}  We derive a formula for $(P_1+i)^{-s} - (P_2+i)^{-s}$, obtained e.g. by taking $s-1$ derivatives with respect to $\mu$ of
\begin{equation}
(P_1-\mu)^{-1} - (P_2-\mu)^{-1} = (P_1-\mu)^{-1} (P_2-P_1) (P_2-\mu)^{-1}
\end{equation}
  and setting $\mu=-i$. 
Leibniz's formula and $\p_\mu^{s_j} (P_1-\mu)^{-1} = s_j! \cdot (P_1-\mu)^{-s_j-1}$ yield
\begin{equations}
(P_1+i)^{-s} - (P_2+i)^{-s} = \sum_{s_1+s_2=s-1} \dfrac{s_1!^2 s_2!^2}{(s-1)!} (P_1+i)^{-s_1-1} (P_2-P_1) (P_2+i)^{-s_2-1}.
\end{equations}

It follows that 
\begin{equations}
B_1(\lambda) - B_2(\lambda)  =  \left[(P_1-\lambda)^{-1} (P_2-P_1) (P_2-\lambda)^{-1},f(x_1)(P_1+i)^{-s}\right]
\\
 + \left[ (P_2-\lambda)^{-1},f(x_1)\sum_{s_1+s_2=s-1} \dfrac{s_1!^2 s_2!^2}{(s-1)!} (P_1+i)^{-s_1-1} (P_2-P_1) (P_2+i)^{-s_2-1}\right].
\end{equations}
Similarly, we find that $C_1(\lambda) - C_2(\lambda)$ is equal to
\begin{equations}
 \left[(P_1-\lambda)^{-1} (P_2-P_1) (P_2-\lambda)^{-1}, [P_1,f(x_1)](P_1-\lambda)^{-1}(P_1+i)^{-s}\right]
\\
+ \left[ (P_2-\lambda)^{-1},[P_1-P_2,f(x_1)](P_1-\lambda)^{-1}(P_1+i)^{-s} \right]
\\
+ \left[ (P_2-\lambda)^{-1},[P_2,f(x_1)](P_1-\lambda)^{-1}(P_2-P_1)(P_2-\lambda)^{-1}(P_1+i)^{-s} \right]
\\
+ \left[ (P_2-\lambda)^{-1},[P_2,f(x_1)](P_1-\lambda)^{-1} \hspace*{-3mm}\sum_{s_1+s_2=s-1} \dfrac{s_1!^2 s_2!^2}{(s-1)!} (P_1+i)^{-s_1-1} (P_2-P_1) (P_2+i)^{-s_2-1}\right].
\end{equations}

\textbf{4.}  The expressions of Step 3 allow us to expand $A_1(\lambda)-A_2(\lambda)$ as a finite sum of commutators $\sum_k\big[D_k(\lambda),E_k(\lambda) \big]$ with the following property. For each $k$, $D_k(\lambda) E_k(\lambda)$ and $E_k(\lambda) D_k(\lambda)$ are finite products of precisely one of each factor $\p_\olambda \trho(\lambda,\olambda)$ and $P_1-P_2$; at most three factors among $(P_j-\lambda)^{-1}$; one factor of the form $f(x_1)$ or $[P_j,f(x_1)]$; and $s$ or $s+1$ factors of the form $(P_j+i)^{-1}$.

We note that $P-Q \in \Psi(\lr{x}^{-1} \lr{\xi}^2)$; that $(P_j+i)^{-1} \in \Psi\big(\lr{\xi}^{-2}\big)$; that  $(P_j-\lambda)^{-1} \in \Psi(1)$ with symbolic bounds blowing up polynomially as $\Im \lambda \rightarrow 0$ -- see \eqref{eq:5r}; and that $\p_\olambda \trho(\lambda,\olambda) \in O(|\Im \lambda|^\infty)$. Therefore we deduce that for any $s \in \N$,
\begin{equation}
D_k(\lambda) E_k(\lambda), \ E_k(\lambda) D_k(\lambda) \in \Psi\big(\lr{x}^{-s} \lr{\xi}^{-2s+4}\big),
\end{equation}
uniformly in $\lambda$. In particular both $D_k(\lambda) E_k(\lambda)$ and $E_k(\lambda) D_k(\lambda)$ are trace class. We deduce that $A_1(\lambda)-A_2(\lambda)$ is (uniformly in $\lambda$) trace class with vanishing trace. The formula \eqref{eq:2o} completes the proof.
\end{proof}

\begin{proof}[Proof of Lemma \ref{lem:1f}]
\textbf{1.} In comparison with Lemma \ref{lem:1d}, the operator $P_1 - P_2$ vanishes now in a (non-compact) strip $|x_2| \leq L'$. We prove lemma \ref{lem:1f} using Lemma \ref{lem:1d} and an approximation argument. 

  Fix $\epsi > 0$, $\chi(x) \in C^\infty_0(\R^2,\R)$ equal to $1$ for $|x| \leq 1$ and $P_3 = \Re\big(P_1 + \chi(\epsi x) (P_2-P_1)\big)$, where we recall that $\Re(T) = \frac{T^*+T}{2}$. We note that $P_3$ is an elliptic selfadjoint operator of order $2$, equal to $P_1$ outside a compact set. From Lemma \ref{lem:1d}, $\JJ_e(P_1) = \JJ_e(P_3)$ thus
\begin{equation}
\JJ_e(P_2) - \JJ_e(P_1) = \JJ_e(P_2) - \JJ_e(P_3).
\end{equation}

\textbf{2.}  Let $s, \rho$ as in the proof of Lemma \ref{lem:1c}. 
We write for $j=2, 3$:
\begin{equation}
[P_j,f(x_1)] g'(P_j) =  \int_{\C^+} \dd{^2\trho(\lambda,\olambda)}{\lambda \p \olambda} \cdot [P_j,f(x_1)](P_j-\lambda)^{-1} \cdot (P_j+i)^{-s} \cdot \dfrac{d^2\lambda}{\pi}.
\end{equation}
We observe that
\begin{equations}\label{eq:2s}
[P_2,f(x_1)](P_2-\lambda)^{-1}  \cdot (P_2+i)^{-s} - [P_3,f(x_1)](P_3-\lambda)^{-1} \cdot  (P_3+i)^{-s}
\\
= [P_2-P_3,f(x_1)](P_2-\lambda)^{-1}  \cdot (P_2+i)^{-s} 
\\
+
[P_3,f(x_1)](P_2-\lambda)^{-1}(P_3-P_2)(P_3-\lambda)^{-1} \cdot  (P_3+i)^{-s}
\\
+
[P_3,f(x_1)](P_2-\lambda)^{-1}  \cdot \sum_{s=s_2+s_3} \dfrac{s_2!^2 s_3!^2}{(s-1)!} (P_2+i)^{-s_2-1}(P_3-P_2)(P_3+i)^{-s_3-1}.
\end{equations}

\textbf{3.}  The bounds that we prove below are all uniform as $\epsi \rightarrow 0$. We observe that
\begin{equation}
P_3-P_2 = \Re\big((\chi(\epsi x)-1\big) (P_2-P_1)\big).
\end{equation}
This vanishes when $|x| \leq \epsi^{-1}$. In particular, $P_3-P_2 \in \epsi \lr{x_1} \cdot \Psi(\lr{\xi}^2)$ with uniform symbolic bounds as $\epsi \rightarrow 0$. 

Since $[P_3-P_2,f]$ and $[P_3,f]$ are in $\Psi(\lr{x_1}^{-\infty} \lr{\xi}^2)$, we deduce from \eqref{eq:2s} that
\begin{equation}
\dfrac{1}{\epsi} \left([P_2,f(x_1)](P_2-\lambda)^{-1}  \cdot (P_2+i)^{-s} - [P_3,f(x_1)](P_3-\lambda)^{-1} \cdot  (P_3+i)^{-s} \right)
\end{equation}
is in $\Psi\big(\lr{\xi}^{4-2s}\big)$. The symbolic bounds blow up polynomially as $\Im \lambda \rightarrow 0$. Thus
\begin{equation}\label{eq:2q}
[P_2,f(x_1)]g'(P_2) - [P_3,f(x_1)]g'(P_3) \ \in \ \epsi \cdot \Psi(\lr{\xi}^{-\infty}).
\end{equation}

\textbf{4.}  From \eqref{eq:2r}, we also have $[P_j,f(x_1)]g'(P_j) \in \Psi(\lr{\xi}^{-\infty} \lr{x}^{-\infty})$. We deduce that \eqref{eq:2q} belongs to $\Psi(\lr{\xi}^{-\infty} \lr{x}^{-\infty})$. Interpolating at the symbolic level, we get 
\begin{equation}
[P_2,f(x_1)]g'(P_2) - [P_3,f(x_1)]g'(P_3) \ \in \ \epsi^{1/2} \cdot \Psi\big(\lr{\xi}^{-\infty}\lr{x}^{-\infty}\big).
\end{equation}
In particular, $[P_2,f(x_1)]g'(P_2) - [P_3,f(x_1)]g'(P_3)$ is trace-class and its trace is $O(\epsi^{1/2})$. We conclude that
\begin{equation}
\JJ_e(P_1) - \JJ_e(P_2) = O(\epsi^{1/2})
\end{equation}
for every $\epsi \in (0,1)$; this completes the proof.\end{proof}

\begin{proof}[Proof of Lemma \ref{lem:1e}] \textbf{1.}  From the properties of $\psi$, $g' \circ \psi = g'$. Moreover, since the spectrum of $P$ is bounded below, there exists $\vp(\lambda) \in C^\infty_0(\R)$ such that $\psi(P) = \lambda_2 + \vp(P)$ and $\psi'(P) = \vp'(P)$.  It follows that 
\begin{equations}\label{eq:3a}
\big[ \psi(P), f(x_1) \big] g' \circ \psi(P) = \big[ \vp(P), f(x_1) \big] g'(P).
\end{equations}

We use the Helffer--Sj\"ostrand formula to write
\begin{equation}\label{eq:3c}
\vp(P) =  \int_{\C^+} \dd{\tvp(\lambda,\olambda)}{\olambda} \cdot  (P-\lambda)^{-1} \cdot \dfrac{d^2\lambda}{\pi}. 
\end{equation}
Since  $(P-\lambda)^{-1} \in \Psi(1)$ with bounds blowing up polynomially with $|\Im \lambda|^{-1}$, $\vp(P) \in \Psi(1)$. As for \eqref{eq:2g}, $\big[ \vp(P), f(x_1) \big] \in \Psi(\lr{x_1}^{-\infty})$. From \eqref{eq:2h}, $g'(P) \in \Psi\big( \lr{x_2}^{-\infty} \lr{\xi}^{-\infty} \big)$. We deduce from \eqref{eq:3a} that
\begin{equation}
 \big[ \psi(P), f(x_1) \big] g' \circ \psi(P) \ \in \ \Psi\left(\lr{x}^{-\infty} \lr{\xi}^{-\infty}\right).
\end{equation}
Hence $\big[ \psi(P), f(x_1) \big] g' \circ \psi(P)$ is trace-class and $\JJ_e\big( \psi(P) \big)$ is properly defined, with
\begin{equation}
\JJ_e\big( \psi(P) \big) = \Tr_{L^2(\R^2)} \big( \big[ \vp(P), f(x_1) \big] g'(P) \big).
\end{equation}

\textbf{2.}  Because of \eqref{eq:3c},
\begin{equations} \label{eq:2x}
[\vp(P),f(x_1)] g'(P) = \int_{\C^+} \dd{\tvp(\lambda,\olambda)}{\olambda} \cdot \left[(P-\lambda)^{-1},f(x_1) \right] g'(P) \cdot \dfrac{d^2\lambda}{\pi}
\\
= -\int_{\C^+} \dd{\tvp(\lambda,\olambda)}{\olambda} \cdot (P-\lambda)^{-1} \big[P,f(x_1)\big] (P-\lambda)^{-1} g'(P) \cdot \dfrac{d^2\lambda}{\pi}.
\end{equations}
Recall that $g'(P) \in \Psi\big(\lr{x_2}^{-\infty} \lr{\xi}^{-\infty}\big)$; $(P-\lambda)^{-1} \in \Psi(1)$ (with bounds blowing up polynomially in $|\Im \lambda|^{-1}$); and $[f(x_1),P] \in \Psi\big(\lr{x_1}^{-\infty} \lr{\xi}^2\big)$. Since $\p_\olambda \tvp(\lambda,\olambda) = O( |\Im \lambda|^\infty )$, we deduce  that
\begin{equation} 
\dd{\tvp(\lambda,\olambda)}{\olambda}  \left[f(x_1),P\right] (P-\lambda)^{-1} g'(P) \ \in  \ \Psi\left( \lr{x}^{-\infty}\lr{\xi}^{-\infty} \right),
\end{equation}
uniformly in $\lambda$. Thus we can trace \eqref{eq:2x} and permute trace and integral. We can also move $(P-\lambda)^{-1}$ cyclically from the left to the right. We end up with
\begin{equations}
\JJ_e\big( \psi(P) \big)
  =   - \int_{\C^+} \dd{\tvp(\lambda,\olambda)}{\olambda} \cdot  \Tr_{L^2(\R^2)} \left( (P-\lambda)^{-1} \big[P,f(x_1)\big] (P-\lambda)^{-1} g'(P) \right) \cdot  \dfrac{d^2\lambda}{\pi}
\end{equations}
\begin{equations}\label{eq:5k}
 =  - \int_{\C^+} \dd{\tvp(\lambda,\olambda)}{\olambda} \cdot  \Tr_{L^2(\R^2)} \left( \big[P,f(x_1)\big] (P-\lambda)^{-2} g'(P) \right) \cdot \dfrac{d^2\lambda}{\pi}.
\end{equations}
We observe that $(P-\lambda)^{-2} = \p_\lambda (P-\lambda)^{-1}$. We integrate \eqref{eq:5k} w.r.t. $\lambda$:
\begin{equations}
\JJ_e\big( \psi(P) \big) =   \int_{\C^+} \dd{^2\tvp(\lambda,\olambda)}{\lambda \p\olambda} \cdot \Tr_{L^2(\R^2)} \left( \big[P,f(x_1)\big] (P-\lambda)^{-1} g'(P) \right) \dfrac{d^2\lambda}{\pi}.
\end{equations}
We permute trace and integral once again and end up with
\begin{equation}\label{eq:3y}
\JJ_e(P) = \Tr_{L^2(\R^2)}\big([P,f(x_1)] \vp'(P)g'(P)\big).
\end{equation}
This completes the proof because $\vp'(P) = \psi'(P)$ and $\psi'(\lambda) = 1$ on the support of $g'$: the RHS of \eqref{eq:3y} is $\JJ_e(P)$. 
\end{proof}

\subsection{Deformation to a semiclassical operator}\label{sec:24} We recall that $\Re(T) = \frac{T+T^*}{2}$. Let $\chi_+(x_2), \ \chi_-(x_2) \in C^\infty(\R)$ and $\chi_0(x_2) \in C^\infty_0(\R)$ such that
\begin{equation}
\chi_+(x_2) = \systeme{1 & \text{ for } x_2 \geq 2 \\ 0 & \text{ for } x_2 \leq 1}, \ \ \ \chi_+(x_2) = \systeme{1 & \text{ for } x_2 \leq - 2 \\ 0 & \text{ for } x_2 \geq -1}, \ \ \ \chi_0 = 1-\chi_--\chi_+. 
\end{equation}
See Figure \ref{fig:6}. 
Given $h > 0$, we introduce
\begin{equations}\label{eq:1c}
P_h \de  \Re\Bigg(\sum_{|\az| \leq 2} b_\az(hx,x) D_x^\az \Bigg) +\Re\Big( \chi_0(hx_2) P_0 \Big)
,  \ \ \   \text{where} \\ 
b_\az(x,y) \de \chi_+(x_2)a_{\az,+}(y) + \chi_-(x_2) a_{\az,-}(y), \ \ \ \ P_0 \de -\Delta + |\lambda_0| + 2.
\end{equations}

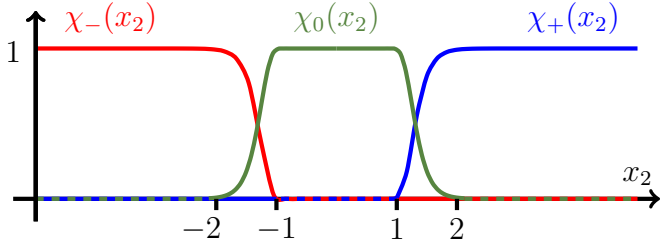
\begin{figure}
\floatbox[{\capbeside\thisfloatsetup{capbesideposition={right,center},capbesidewidth=2.8in}}]{figure}[\FBwidth]
{\caption{Graphs of $\chi_-, \chi_0$ and $\chi_+$.}\label{fig:6}}
{\begin{tikzpicture}

\definecolor{grine}{rgb}{.34, .52, .25}

\node at (0,0) {\includegraphics{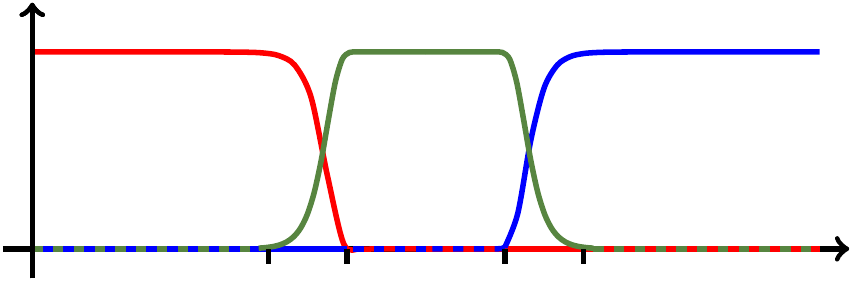}};

\node[red] at (-3,1.3) {$\chi_-(x_2)$}; 
\node[blue] at (3.15,1.3) {$\chi_+(x_2)$}; 
\node[grine] at (0,1.3) {$\chi_0(x_2)$};

\node  at (-1.8,-1.5) {$-2$}; 
\node  at (1.6,-1.5) {$2$}; 
\node  at (-0.8,-1.5) {$-1$}; 
\node  at (0.8,-1.5) {$1$}; 
\node  at (4,-.8) {$x_2$}; 
\node  at (-4.3,.85) {$1$};

   \end{tikzpicture}}
   \vspace*{-.3cm}
\end{figure}

The operator $P_h$ is a symmetric differential operator of order $2$. Below, we write
\begin{equation}
P_h = \sum_{|\az| \leq 2} c_\az(hx,x) D_x^\az,
\end{equation}
rather than \eqref{eq:1c}. The coefficients $c_\az(x,y)$ have a two-scale structure: they belong to $C^\infty_b(\R^2 \times \Tt^2)$. Their dependence in $h$ is polynomial, in particular they remain uniformly bounded as $h \rightarrow 0$.

If $u(x) \in C^\infty(\R^2)$ has support in $\{ \pm hx_2 \geq 2 \}$ then the coefficients $b_\az(hx,x)$ are equal to $a_{\az,\pm}(x)$ on the support of $u$ and 
\begin{equation}
Pu = \Re\Bigg(\sum_{|\az|\leq 2} a_{\az,\pm}(x) D_x^\az\Bigg) u = \Re(P_\pm) u = P_\pm u.
\end{equation}
This implies that
\begin{equation}\label{eq:5w}
c_\az(x,y) = \systeme{a_{\az,+}(y) & \text{ for } x_2  \geq 2 \ \  \\ a_{\az,-}(y) & \text{ for } x_2 \leq -2 } . 
\end{equation}
In other words, $P_h$ is equal to $P$ outside $|hx_2| \leq 2$. We similarly observe that if $u$ has support in $\{ |hx_2| \leq 1 \}$ then $P_h u = P_0u$. Since $\sigma_{L^2(\R^2)} (P_0) = [|\lambda_0|+2,\infty)$, $P_0$ heuristically behaves as a barrier between $P_+$ and $P_-$ at energies below $|\lambda_0|+2$. This can be ignored in \S\ref{sec:3}. It will play a role in \S\ref{sec:4}. 

 Finally, we observe that $P_h$ is elliptic. Indeed, since $P$ is elliptic, $P_\pm$ are elliptic (with, say, ellipticity constant $0 < c \leq 1$) and
\begin{equations}
 \sum_{|\az| = 2} \Re\big( b_\az(hx,x) \big) \xi^\az + \chi_0(hx_2) |\xi|^2 =  \sum_\pm \chi_\pm(hx_2) \sum_{|\az| = 2}  \Re\big(a_{\az,\pm}(x)\big) \xi^\az + \chi_0(hx_2) |\xi|^2
\\
\geq \big( \chi_+(hx_2) + \chi_-(hx_2) + \chi_0(hx_2)\big) \cdot c|\xi|^2 = c|\xi|^2.
\end{equations}

This proves that $P_h$ satisfies the assumptions of \S\ref{sec:21}. From Lemmas \ref{lem:2a} and \ref{lem:1f},
\begin{equation}\label{eq:1e}
\JJ_e(P) = \JJ_e(P_h) = \Tr_{L^2(\R^2)}\Big( \big[ P_h,f(hx_1) \big] g'(P_h) \Big).
\end{equation}
The key observation is that $P_h$ is, in an appropriate sense, a semiclassical operator. We give here a formal explanation and we postpone the rigorous version  \cite{GMS:91} to \S\ref{sec:3}.  Let $U(x,y) \in C^\infty(\R^2 \times \Tt^2)$ and set $u(x) = U(hx,x)$. Then 
\begin{equation} \label{eq:1d}
P_h u(x) = \big(\Pp_h U\big)(x,hx) \ \ \ \ \text{where} \ \ \ \ 
\Pp_h \de \sum_{|\az| \leq 2} c_\az(x,y) (D_y+hD_x)^\az.
\end{equation}
The operator $\Pp_h$ is semiclassical in $x$ with operator valued-symbol $\Pp(x,\xi) + O(h)$, where
\begin{equation}
\Pp(x,\xi) \de \sum_{|\az| \leq 2} c_\az(x,y) (D_y+\xi)^\az + O(h).
\end{equation}

\section{Semiclassical deformation and effective Hamiltonian}\label{sec:3}

In the rest of the paper, we compute $\JJ_e(P)$ using the operator $P_h$ defined in \eqref{eq:1c}. To emphasize that $\JJ_e(P)$ depends only on $P_+$ and $P_-$, we write below
\begin{equation}
\JJ_e(P_-,P_+) \de \JJ_e(P).
\end{equation}

Operators $P_h$ of the form \eqref{eq:1c} first appeared in solid state physics in the 70's. The first mathematical works constructed WKB quasimodes \cite{Bu:87,GRT:88}. Here, a key paper is G\'erard--Martinez--Sj\"otrand \cite{GMS:91}. It establishes a unitary equivalence between $P_h$ acting on $L^2(\R^2)$ and $\Pp_h$ -- see \eqref{eq:1d} -- acting on
\begin{equation}
\HH_1 \de \Bigg\{ \sum_{m \in \Z^2} v(x) \delta(x-hy+hm), \ v \in L^2(\R^2) \Bigg\}  \subset \SSS'\left(\R^2 \times \Tt^2\right).\footnote{This space is denoted $L_0$ in \cite{GMS:91} and \cite[\S13]{DS:99}. It identifies with $L^2(\R^2)$, see \S\ref{sec:314}.}
\end{equation}

This equivalence yield a semiclassical formula for the edge index: 
from \eqref{eq:1e},
\begin{equation}\label{eq:3f}
\JJ_e(P_-,P_+) = \Tr_{L^2(\R^2)} \left( \big[ P_h,f(x_1)\big] g'(P_h) \right) = \Tr_{\HH_1} \left( \big[ \Pp_h,f(x_1)\big] g'(\Pp_h) \right).
\end{equation} 
Another important advance of \cite{GMS:91} is the construction of an effective Hamiltonian $E_{22}(\lambda)$ for $P_h$. This provides a discrete singular value problem whose solutions are precisely the eigenvalues of $P_h$ (within a given spectral window).

Dimassi, Zerzeri and Duong \cite{Di:93,DZ:03,DD:14} used \cite{GMS:91} to provide various semiclassical trace expansions for operators in the form \eqref{eq:1c}, see for instance \eqref{eq:0k}. In principle, the coefficients $b_j$ in \eqref{eq:0k} can be expressed from semiclassical symbols. We expect a similar expansion here:
\begin{equation}\label{eq:4m}
\Tr_{\HH_1} \left( \big[ \Pp_h,f(x_1)\big] g'(\Pp_h) \right) \  \sim   \  \sum_{j \geq 0} a_j \cdot h^{j-2} \ \ \ \ \text{as } h \rightarrow 0,
\end{equation}
with coefficients $a_j$ computable via symbolic calculus. 
However, \eqref{eq:3f} indicates that \eqref{eq:4m} does not depend on $h$. Hence all terms $a_j, j \neq 2$ in the expansion \eqref{eq:4m} must vanish and $a_2 = \JJ_e(P_-,P_+)$.

From the technical point of view, \S\ref{sec:3} is closer to \cite{GMS:91,Di:93} than to previous papers on the bulk-edge correspondence. As in \cite{Di:93}, we will pose a Grushin problem and construct a discrete (finite difference) effective Hamiltonian whose singular values describe accurately $\Pp_h$ near energy $\lambda_0$. 

We will use symbolic calculus to derive a formula for $a_2$.
Specifically, we will adapt calculations of Elgart--Graf--Schenker \cite{EGS:05} from the eigenvalue setting to the singular value setting. This will prove that $\JJ_e(P_-,P_+)$ is (up to summation) a double semiclassical commutator. This proves $a_0=a_1 = 0$; and allows us to compute $a_2$ in terms of the leading symbol of the effective Hamiltonian, $E_{22}(x,\xi;\lambda)$ -- see Theorem \ref{thm:1}. An algebraic manipulation reduces the formula for $a_2$ to an integral involving only asymptotics of $E_\pm(\xi;\lambda)$ of $E_{22}(x,\xi;\lambda)$ as $x_2 \rightarrow \pm \infty$ -- see Theorem \ref{thm:2}. 

We will connect $E_\pm(\xi;\lambda)$ to Chern numbers in \S\ref{sec:4}, completing the proof of~Theorem~\ref{thm:0}.

\subsection{Semiclassical calculus}\label{sec:31} We start this section with a review of semiclassical calculus. While pseudodifferential calculus purely measures regularity, semiclassical calculus allows for the quantitative study of frequencies of order $h^{-1}$, $h \rightarrow 0$. The textbooks \cite{DS:99,Zw:12} provide excellent introductions. The results exposed below are presented in \cite[\S7-8 and \S13]{DS:99}; see also \cite[\S4 and \S13]{Zw:12}.

We say that a symbol $a(x,\xi) \in C^\infty(\R^2 \times \R^2)$ (implicitly depending on $h$) belongs to $S(m)$ if \eqref{eq:5o} holds with bounds $C_\az$ uniform in $h \in (0,1]$. We then define
\begin{equation}
\big(\Op_h(a) u\big)(x) \de \dfrac{1}{(2\pi h)^2} \int_{\R^2 \times \R^2} e^{i\frac{\xi}{h}(x-x')} a \left( \dfrac{x+x'}{2},\xi \right) u(x') dx'd\xi, \ \ \ \ u \in C^\infty_0(\R^2).
\end{equation}
Such operators have bounded extensions on $\SSS(\R^2)$ and we denote the corresponding class by $\Psi_h(m) = \Op_h\big( S(m) \big)$. In the sequel, we will allow for symbols valued in Hilbert spaces, typically $\C^d$ or $L^2(\Tt^2)$. 

\subsubsection{Composition} See \cite[\S7]{DS:99} and \cite[\S4.3-4.4]{Zw:12}. If $a \in S(m_1)$ and $b \in S(m_2)$ then $\Op_h(a) \Op_h(b) \in \Psi_h(m_1m_2)$. We denote its symbol by $a \# b$. One clear advantage of semiclassical over pseudodifferential calculus is the composition formula: for any $K$,
\begin{equation}
a\# b(x,\xi) = \sum_{k=0}^K \dfrac{i^k h^k}{k!} \left. \left(\dfrac{D_\xi D_{x'}-D_xD_{\xi'}}{2}\right)^k \big( a(x,\xi) b(x',\xi') \big) \right|_{\substack{x'=x \\ \xi'=\xi}} + O_{S(m_1m_2)}\left(h^{K+1}\right).
\end{equation}
It implies that $a\#b(x,\xi)$ depends only of $a$ and $b$ locally near $(x,\xi)$, modulo a small remainder, $O(h^\infty)$. We will use the explicit expansion only for $K=0$ and $K=1$:
\begin{itemize}
\item $\Op_h(a) \Op_h(b)$ has symbol $ab + O_{S(m_1m_2)}(h)$;
\item $\big[\Op_h(a), \Op_h(b)\big]$ has symbol
\begin{equation}
\dfrac{h}{2i} \big(\{a,b\} - \{b,a\}\big) + O_{S(m_1m_2)}\left(h^2\right),\ \footnote{This reduces to $\frac{h}{i}\{a,b\}$ when $a$ or $b$ is scalar-valued; however most operators considered below will be matrix and operator-valued.} \ \ \ \ \text{where} \ \ \{a,b\} \de \sum_{j=1}^2 \dd{a}{\xi_j} \dd{b}{x_j} - \dd{a}{x_j} \dd{b}{\xi_j}.
\end{equation}
\end{itemize}

From the composition formula, if $a(x,\xi) \in S(1)$ satisfies $\inf_{\R^2 \times \R^2} \big|a(x,\xi)\big| > 0$ then $\Op_h(a)$ is invertible for $h$ sufficiently small. The semiclassical version of a theorem of Beals \cite{Be:77} implies that its inverse is in $\Psi_h(1)$.

\subsubsection{Resolvents and functional calculus}\label{sec:312} See \cite[\S8]{DS:99}. Let $a \in S(1)$ be Hermitian-valued. Then $\Op_h(a) \in \Psi_h(1)$ is a selfadjoint operator. Moreover, for every $\lambda \in \C^+$, the resolvent $(\Op_h(a) -\lambda)^{-1}$ is also in $\Psi_h(1)$. 

If $r(\cdot;\lambda) \in S(1)$ is such that $\Op_h\big(r(\cdot,\lambda)\big) = (\Op_h(a)-\lambda)^{-1}$, then for any $R > 0$, the following estimates hold uniformly for $\lambda \in \Dd(0,R)$ and $h \in (0,1]$:
\begin{equations}\label{eq:0g}
r(\cdot;\lambda) = \big( a - \lambda)^{-1} + O_{S(1)}\left( h \cdot |\Im \lambda|^{-8}\right);
\\
\sup_{(x,\xi) \in \R^2} \left| \p^\az r(x,\xi;\lambda) \right| \leq c_{\az,R} \cdot \max\left( 1, \dfrac{h^{1/2}}{|\Im \lambda|} \right)^5 \cdot |\Im \lambda|^{-|\az|-1}. \ 
\end{equations}
Using the estimates \eqref{eq:0g} and the Helffer--Sj\"ostrand formula \eqref{eq:2d}, we can develop the functional calculus of selfadjoint semiclassical operators. If $\vp(\lambda) \in C^\infty_0(\R)$, then $\vp\big( \Op_h(a) \big) \in \Psi_h(1)$ and its symbol is 
\begin{equation}
\vp\big( a(x,\xi) \big) + O_{S(1)}(h).
\end{equation}

\subsubsection{Trace class} See \cite[\S8]{DS:99}. Similarly to \S\ref{sec:224}, if $m$ is an order function in $L^1$ then operators in $\Psi_h(m)$ are trace-class. Moreover, there exists $C > 0$ such that for any $a \in S(m)$,
\begin{equation}
\big\| \Op_h(a) \big\|_\Tr \leq Ch^{-2} \cdot |m|_{L^1}  \cdot \sup_{|\az| \leq 5} C_\az \ \footnote{The numbers $2$ and $5$ are specific to dimension $n=2$; in general they are $n$ and $2n+1$, respectively.}
\end{equation}
where the constants $C_\az$ are those of \eqref{eq:5o}.

\subsubsection{Periodic and equivariant classes}\label{sec:314} See \cite[\S13]{DS:99}. We will need to consider classes of operator-valued symbols satisfying certain (pseudo-)periodic conditions. Fix $d \in \N$. We introduce:
\begin{itemize}
\item The class $S^{(22)}(m) \subset S(m)$ of symbols $a(x,\xi) \in C^\infty\big(\R^2 \times \R^2, M_d(\C)\big)$ such that
\begin{equation}
a(x,\xi+2k\pi) = a(x,\xi), \ \ \ k \in \Z^2;
\end{equation}
\item The class $S^{(12)}(m) \subset S(m)$ of symbols $R(x,\xi) \in C^\infty\big(\R^2 \times \R^2, \BB\big(\C^d,L^2(\Tt^2)\big)\big)$ -- i.e. with values in linear operators from $\C^d$ to $L^2(\Tt^2)$ -- such that
\begin{equation}
R(x,\xi+2k\pi) = e^{-2ik\pi y} \cdot R(x,\xi), \ \ \ k \in \Z^2;
\end{equation}
\item The class $S^{(21)}(m)$ of adjoints of symbols in $S^{(12)}(m)$;
\item The class $S^{(11)}(m) \subset S(m)$ of symbols $\Ww(x,\xi) \in C^\infty\big(\R^2 \times \R^2, \BB\big(L^2(\Tt^2)\big)\big)$~with
\begin{equation}\label{eq:1w}
\Ww(x,\xi+2k\pi) = e^{-2ik\pi y} \cdot \Ww(x,\xi) \cdot e^{2ik\pi y}, \ \ \ k \in \Z^2.
\end{equation}
\end{itemize}
We let $\Psi_h^{(jk)}(m) = \Op_h\big(S^{(jk)}(m)\big)$ be the corresponding operator classes; we observe that $\Pp(x,\xi) \in \Psi_h^{(11)}\big(\lr{\xi}^2\big)$. Because of the (pseudo-)periodic conditions, if $m$ decays with $\xi$ then $\Psi^{(jk)}_h(m) = \{0\}$. The order function $m$ may nonetheless decay with $x$.

The classes $\Psi_h^{(jk)}(m)$ appear in relation with the effective Hamiltonian method of \cite{GMS:91}. From the general theory of Pdos, they act on tempered distributions; for instance, operators in $\Psi_h^{(11)}(m)$ act on $\SSS'(\R^2 \times \Tt^2)$. The pseudo-periodic conditions yield additional mapping properties. If $m$ is uniformly bounded in $x$, then operators in $\Psi_h^{(jk)}(m)$ map $\HH_j$ to $\HH_k$, where
\begin{equations}\label{eq:0n}
\HH_1 =  \Bigg\{ \sum_{m \in \Z^2} v(x) \delta(x-hy+hm), \ v \in L^2\big(\R^2\big) \Bigg\}  \subset \SSS'\left(\R^2 \times \Tt^2\right),
\\
\HH_2 \de  \Bigg\{ \sum_{m \in \Z^2} v_m \delta(x-hm), \ v \in \ell^2\big(\Z^2,\C^d\big) \Bigg\}  \subset \SSS'\left(\R^2,\C^d\right).
\end{equations}
The space $\HH_2$ is naturally isomorphic to $\ell^2(\Z^2,\C^d)$. Similarly, $\HH_1$ is isomorphic to $L^2(\R^2)$. Indeed, for any $v(x) \in L^2(\R^2)$, 
\begin{equation}
\int_{\Tt^2} \sum_{m \in \Z^2} v(x) \delta(x-hy+hm) \cdot h^2 dy = v(x),
\end{equation}
where equality holds in the sense of distributions on $\R^2$.

A consequence is that $\Pp_h$ acts on $\HH_1$. In this sense, the elements of $\HH_1$ identify with the two-scale functions $U(hy,y)$ considered in \S\ref{sec:24}: in \eqref{eq:0n}, the Dirac masses constrain $x = hy$ modulo $(h\Z)^2$.

A result due to Dimassi \cite[\S1]{Di:93} -- and fundamental here -- asserts if $a \in S^{(22)}\big(\lr{x}^{-3}\big)$ then $\Op_h(a)$ is trace-class on $\HH_2$ and
\begin{equation}\label{eq:0f}
\Tr_{\HH_2}\big(\Op_h(a)\big) = \dfrac{1}{(2\pi h)^2} \int_{\R^2 \times \Tt_*^2} a(x,\xi)  \ dx d\xi + O(h^\infty). \ \footnote{Strictly speaking, \cite[Remark 1.3a]{Di:93} and \cite[Lemma 13.29]{DS:99} are stated for symbols in $S^{(22)}(1)$ that are compactly supported in $x$; the proof applies (with no change) to sufficiently decaying symbols.}
\end{equation}

\subsubsection{Grushin problem} Here we recall basic properties of 
Grushin problems; see for instance \cite[\S13]{DS:99} and \cite{SZ:07} for various applications.  Assume $Q : \HH_1 \rightarrow \HH_1$, $R_{12} : \HH_2 \rightarrow \HH_1$ and $R_{21} : \HH_2 \rightarrow \HH_1$ are three operators such that for $\lambda$ in a neighborhood of $\C$, the operator
\begin{equation}
\matrice{Q-\lambda & R_{12} \\ R_{21} & 0} \ : \ \HH_1 \oplus \HH_2 \rightarrow \HH_1 \oplus \HH_2
\end{equation}
is invertible. We write the inverse as
\begin{equation}\label{eq:1m}
\matrice{Q-\lambda & R_{12} \\ R_{21} & 0}^{-1} = \matrice{E_{11}(\lambda) & E_{12}(\lambda) \\ E_{21}(\lambda) & E_{22}(\lambda)}.
\end{equation}
Then the operators $E_{jk}(\lambda)$ depend analytically on $\lambda$. Moreover, $Q - \lambda$ is invertible on $\HH_1$ if and only if $E_{22}(\lambda)$ is invertible on $\HH_2$; and
\begin{equations}\label{eq:3x}
(Q - \lambda)^{-1} = E_{11}(\lambda) - E_{12}(\lambda) E_{22}(\lambda)^{-1} E_{21}(\lambda),
\\
E_{22}(\lambda)^{-1} = - R_{21} (Q - \lambda)^{-1} R_{12}.
\end{equations}

\subsection{Review of the effective Hamiltonian method}\label{sec:32} In the sequel, $\Omega$ is a bounded neighborhood in $\C^+$ of $\supp(\tg) \cap \C^+$, and $\Omega' \subset \Omega$ is a neighborhood of $\supp(\tg)$ with
\begin{equation}\label{eq:1z}
\ove{\Omega'} \cap \R \subset [\lambda_0-\epsilon,\lambda_0+\epsilon]. 
\end{equation}
We set $\lambda_+ = \sup\{ 2|\lambda| : \ \lambda \in \Omega \}$.

The idea behind the effective Hamiltonian method is to produce a singular value problem for a discrete Hamiltonian, that describe accurately low-energy spectral aspects of $\Pp_h$.
  We follow the construction of \cite{GMS:91}. It consists in finding $d \in \N$ and a pseudodifferential operator $R_{12} : \HH_1 \rightarrow \HH_2$ with its adjoint $R_{21} : \HH_2 \rightarrow \HH_1$ such that 
\begin{equation}\label{eq:1h}
\matrice{\Pp_h - \lambda & R_{12} \\ R_{21} & 0} \ : \  \HH_1 \oplus \HH_2 \rightarrow \HH_1 \oplus \HH_2
\end{equation}
is invertible for all $\lambda$ in $\Omega$. We refer to \cite[\S13]{DS:99} for a comprehensive presentation.

Following \cite{GMS:91}, there exists $\vp_1(y,\xi), \dots, \vp_d(y,\xi) \in C^\infty(\R^2 \times \R^2)$, satisfying
\begin{equations}\label{eq:1k}
\vp_j(y+\ell,\xi+2\pi k) = e^{-2i\pi ky} \cdot \vp_j(y,\xi); \ \ \ \ \big\langle\vp_m(\cdot, \xi),\vp_n(\cdot, \xi)\big\rangle_{L^2(\Tt^2)} = \delta_{nm},\end{equations}
such that for all $(x,\xi) \in \R^2 \times \R^2$,
\begin{equation}\label{eq:1f}
u \in \big[ \vp_1(\cdot,\xi), \dots, \vp_d(\cdot,\xi)\big]^\perp \ \Rightarrow \ 
 \blr{\big( \Pp(x,\xi) - \lambda_+) u ,u }_{L^2(\Tt^2)} \geq 3 |u|^2_{L^2(\Tt^2)}.
\end{equation}
For technical reasons, we prefer to work with the operator $\Qq_h = \psi(\Pp_h)$, where $\psi$ satisfies \eqref{eq:3k}. We note that $\JJ_e(P_-,P_+) = \JJ_e(P_h) = \JJ_e\big(\psi(P_h)\big)$, see Lemma \ref{lem:1f} and \ref{lem:1e}. Using the unitary equivalence between $P_h$ and $\Pp_h$, we deduce that
\begin{equation}
\JJ_e(P_-,P_+) = \Tr_{\HH_1} \Big( \big[ \Qq_h, f(x_1) \big] g'(\Qq_h) \Big).
\end{equation}

The operator $\Qq_h$ is in $\Psi_h^{(11)}(1)$ and its leading symbol is the bounded operator
\begin{equation}
\Qq(x,\xi) = \psi\big( \Pp(x,\xi) \big) \ : \ L^2(\Tt^2) \rightarrow L^2(\Tt^2),
\end{equation}
because $\psi(\Pp_h) = \lambda_2 + \vp(\Pp_h)$ for some $\vp(\lambda) \in C^\infty_0(\R)$; and because of \S\ref{sec:312}. 
We now extend \eqref{eq:1f} to $\Qq(x,\xi) = \psi\big( \Pp(x,\xi) \big)$.

\begin{lem}\label{lem:1a} If \eqref{eq:1f} holds then there exists $\lambda_2 \geq \lambda_0+2\epsilon$ and $\psi$ satisfying \eqref{eq:3k} such that for every $(x,\xi) \in \R^2 \times \R^2$,
\begin{equation}\label{eq:1i}
u \in \big[ \vp_1(\cdot,\xi), \dots, \vp_d(\cdot,\xi)\big]^\perp  \ \ \Rightarrow \ \ \blr{\left(\psi\big(\Pp(x,\xi)\big) - \lambda_+ \right)u,u}_{L^2(\Tt^2)} \geq |u|_{L^2(\Tt^2)}^2.
\end{equation}
\end{lem}

\begin{proof} \textbf{1.}  We note that $\Qq(x,\xi) \in S^{(11)}(1)$. In particular, it satisfies the pseudoperiodic condition \eqref{eq:1w}. Moreover, $\Pp(x,\xi)$ depends on $x$ only if $x$ is within a compact set $K$. Therefore, it suffices to prove \eqref{eq:1i} for $(x,\xi) \in K \times [0,2\pi]^2$.

 Fix $(x,\xi) \in K \times [0,2\pi]^2$ and $u \in H^2(\Tt^2)$.  We split $u = u_1 + u_2$ where  $u_2 = \Pi(\xi)u \in L^2(\Tt^2)$ is the projection of $u$ to $\big[ \vp_1(\cdot,\xi), \dots, \vp_d(\cdot,\xi)\big]$. In particular $u_1 \in H^2(\Tt^2)$ satisfies the assumption of \eqref{eq:1f} and we have
\begin{equations}
\blr{ \big(\Pp(x,\xi) - \lambda_+ \big) u,u}_{L^2(\Tt^2)} = \sum_{j,k=1}^2
 \blr{ \big(\Pp(x,\xi) - \lambda_+ \big) u_j,u_k}_{L^2(\Tt^2)}
\\
\geq  3|u_1|^2_{L^2(\Tt^2)} - 2\left|\Pp(x,\xi) u_2\right|_{L^2(\Tt^2)} \cdot 3|u_1|_{L^2(\Tt^2)} - \left|\big(\Pp(x,\xi) - \lambda \big) u_2\right|_{L^2(\Tt^2)} \cdot |u_2|_{L^2(\Tt^2)}.
\end{equations}
The space $\big[ \vp_1(\cdot,\xi), \dots, \vp_d(\cdot,\xi)\big]$ is finite dimensional. There exists a constant $C \geq 1$ uniform in $(x, \xi) \in K \times [0,2\pi]^2$  such that 
\begin{equation}
\big| \Pp(x,\xi)  u_2\big|_{L^2(\Tt^2)} + \left|\big(\Pp(x,\xi) - \lambda \big) u_2\right|_{L^2(\Tt^2)}  \leq C |u_2|_{L^2(\Tt^2)}.
\end{equation}
We deduce that
\begin{equations}\label{eq:1j}
\blr{ \big(\Pp(x,\xi) - \lambda_+ \big) u,u}_{L^2(\Tt^2)} \geq  |u_1|^2_{L^2(\Tt^2)} - 2C^2 |u_2|^2_{L^2(\Tt^2)} 
\\
\geq |u|^2_{L^2(\Tt^2)} - 3C^2 |u_2|^2_{L^2(\Tt^2)} =  |u|^2_{L^2(\Tt^2)} - 3C^2 |\Pi(\xi) u|^2_{L^2(\Tt^2)}.
\end{equations}

\textbf{3.}  Fix $\lambda_2 =3C^2+1 + \lambda_+$. We split $u = u_-+u_+$ where $u_- = \1_{(-\infty,\lambda_2)}\big(\Pp(x,\xi)\big) u$ and $u_+ = \1_{[\lambda_2,\infty)}\big(\Pp(x,\xi)\big) u$. Note that by elliptic regularity, $u_- \in H^2(\Tt^2)$.  If $\psi$ satisfies \eqref{eq:3k}, then we have
\begin{equation}
\blr{\big(\psi\big(\Pp(x,\xi)\big) -\lambda_+\big) u,u}_{L^2(\Tt^2)} \geq \blr{\big(\Pp(x,\xi)-\lambda_+\big)u_-,u_-}_{L^2(\Tt^2)} + \big(\lambda_2-\lambda_+ \big) |u_+|^2.
\end{equation}
We obtain from \eqref{eq:1j}:
\begin{equations}
\blr{\big(\psi\big(\Pp(x,\xi)\big) -\lambda_+\big) u,u}_{L^2(\Tt^2)} \geq |u_-|^2_{L^2(\Tt^2)} - 3C^2 |\Pi(\xi) u_-|^2_{L^2(\Tt^2)} + \big(\lambda_2 - \lambda_+ \big) |u_+|^2_{L^2(\Tt^2)}
\\
 \geq |u|^2_{L^2(\Tt^2)} - 3C^2 |\Pi(\xi) u|^2_{L^2(\Tt^2)} + \big(\lambda_2 - \lambda_+ -1 \big) |u_+|^2_{L^2(\Tt^2)} - 3C^2 |\Pi(\xi) u_+|^2_{L^2(\Tt^2)}.
\end{equations}
Since $\lambda_2 = 3C^2+1 + \lambda_+$, we obtain 
\begin{equations}
\blr{\big(\psi\big(\Pp(x,\xi)\big) -\lambda_+\big) u,u}_{L^2(\Tt^2)} \geq |u|^2_{L^2(\Tt^2)} - 3C^2 |\Pi(\xi) u|^2_{L^2(\Tt^2)}.
\end{equations}
This completes the proof: $\Pi(\xi) u = 0$ if $u$ satisfies the condition of \eqref{eq:1i}. \end{proof}

In the rest of the paper we assume given $\vp_1, \dots \vp_d$ satisfying \eqref{eq:1k} and \eqref{eq:1f}; and we fix $\psi$ such that \eqref{eq:1i} holds. Introduce
\begin{equations}\label{eq:1r}
R_{12}(\xi) \de \sum_j t_j \cdot \vp_j(y,\xi); \ \ \ \ \big(R_{21}(\xi)u\big)_j \de \lr{\vp_j(\cdot,\xi),u}_{L^2(\Tt^2)}.
\end{equations}
The symbols $R_{12}(\xi)$ and $R_{21}(\xi)$ are respectively in $S^{(12)}(1)$ and $S^{(21)}(1)$.

A general argument based on \eqref{eq:1i} -- see e.g. \cite[Appendix 13.A]{DS:99} -- implies that
\begin{equation}\label{eq:1l}
\matrice{\Qq(x,\xi) - \lambda & R_{12}(\xi) \\ R_{21}(\xi) & 0} \ : \  L^2(\Tt^2) \oplus \C^d \ \rightarrow \ L^2(\Tt^2) \oplus \C^d
\end{equation}
is invertible for all $(x,\xi) \in \R^2 \times \R^2$ and $\lambda \in \Dd(0,|\lambda_+|)$. Note that this disk contains $\ove{\Omega}$.

The operator
\begin{equation}\label{eq:0b}
\matrice{\Qq_h - \lambda & R_{12} \\ R_{21} & 0} \ : \ \HH_1 \oplus \HH_2 \ \rightarrow \ \HH_1 \oplus \HH_2
\end{equation}
is a bloc-by-bloc semiclassical operator, with blocs in $\Psi_h^{(jk)}(1)$. Its leading symbol is \eqref{eq:1l}, which is invertible. Hence \eqref{eq:0b} is invertible; and its inverse is a semiclassical operator acting on the same space. We write it in the form
\begin{equation}
\matrice{\Qq_h - \lambda & R_{12} \\ R_{21} & 0}^{-1} = \matrice{E_{11}(\lambda) & E_{21}(\lambda) \\ E_{21}(\lambda) & E_{22}(\lambda)}, \ \ \ \ \text{where} \ \ E_{jk}(\lambda) \in \Psi_h^{(jk)}(1).
\end{equation}

\subsection{Reduction}\label{sec:33} We combine the Helffer--Sj\"ostrand formula with the effective Hamiltonian expression \eqref{eq:3x} for $(\Qq_h-\lambda)^{-1}$. This gives
\begin{equations}
g'(\Qq_h) = \int_{\C^+} \dd{^2\tg(\lambda,\olambda)}{\olambda \p\lambda} \cdot (\Qq_h-\lambda)^{-1} \cdot \dfrac{d^2\lambda}{\pi}
\\
= \int_{\C^+} \dd{^2\tg(\lambda,\olambda)}{\olambda \p\lambda} \cdot \left( E_{11}(\lambda) - E_{12}(\lambda) E_{22}(\lambda)^{-1} E_{21}(\lambda) \right) \cdot \dfrac{d^2\lambda}{\pi}
\end{equations}
This integral splits in two parts, one of them involving $E_{11}(\lambda)$. This term is holomorphic in $\Omega$, which is a neighborhood of $\supp(\tg)$ in $\C^+$. An integration by parts with respect to $\olambda$ removes $E_{11}(\lambda)$ and we end up with:
 \begin{equations}\label{eq:3s}
g'(\Qq_h) = - \int_{\C^+} \dd{^2\tg(\lambda,\olambda)}{\olambda \p\lambda} \cdot E_{12}(\lambda) E_{22}(\lambda)^{-1} E_{21}(\lambda)  \cdot \dfrac{d^2\lambda}{\pi}.
\end{equations}

Let $\Phi_0(x_1) \in C_0^\infty(\R)$ such that $\Phi_0(x_1) = 1$ on $[-1,1]$; define $\Phi(x) = \Phi_0(x_1) \Phi_0(x_2)$. We insert $\Phi$ in \eqref{eq:3s} to write $\JJ_e(P_-,P_+) = \Tr_{\HH_1} (\TT_\Phi + \TT_{1-\Phi})$, where 
\begin{equations}\label{eq:3o}
\TT_\Phi \de   - \int_{\C^+} \dd{^2\tg(\lambda,\olambda)}{\olambda \p\lambda} \cdot \big[ \Qq_h, f(x_1)\big] \cdot E_{12}(\lambda) E_{22}(\lambda)^{-1} \Phi(x) E_{21}(\lambda)  \cdot \dfrac{d^2\lambda}{\pi}.
\end{equations}

\begin{lem}\label{lem:1k} The operator $\TT_{1-\Phi}$ is trace-class on $\HH_1$ and  $\| \TT_{1-\Phi} \|_\Tr = O(h^\infty)$.
\end{lem}

\begin{proof} \textbf{1.}  Let $\chi_\pm(x_2) \in C^\infty(\R)$ satisfying \eqref{eq:4a}.  We define $\Qq_{h,\pm} = \psi(\Pp_{h,\pm})$, where $\Pp_{h,\pm}$ are the asymptotic Hamiltonians of $\Pp_h$:
\begin{equation}
\Pp_{h,\pm} \de \sum_{|\az| \leq 2} a_{\az,\pm}(y) (D_x+hD_y)^\az.
\end{equation}
Then we have
\begin{equation}
(\Qq_h-\lambda)^{-1} = \sum_{\pm} \chi_\pm(x_2)(\Qq_{h,\pm}-\lambda)^{-1} + \sum_{\pm} \chi_\pm(x_2) (\Qq_{h,\pm}-\lambda)^{-1} (\Qq_h-\Qq_{h,\pm}) (\Qq_h-\lambda)^{-1}.
\end{equation}
The term $(\Qq_{h,\pm}-\lambda)^{-1}$ is analytic for $\lambda \in \Omega'$. Recall that $\psi(\Pp_h) = \lambda_2 + \vp(\Pp_h)$, where $\vp(\lambda) \in C_0^\infty(\R)$. Arguing as in Steps 3-4 in the proof of Lemma \ref{lem:1c}, we see that
\begin{equation}
\Qq_h-\Qq_{h,\pm} = \vp(\Pp_h) - \vp(\Pp_{h,\pm}) 
\in \Psi_h^{(11)}\left( m_{2,\mp}^{\infty}  \right). 
\end{equation}
Moreover $\chi_\pm(x_2) \in \Psi_h^{(11)}(m_{2,\pm}^{\infty})$. We deduce that $(\Qq_h-\lambda)^{-1}$ is a sum of two terms: the first one is analytic in $\lambda \in \Omega'$; the second one belongs to $\Psi_h^{(11)} ( \lr{x_2}^{-\infty} )$, with bounds blowing up at worst polynomially in $|\Im \lambda|^{-1}$ .

\textbf{2.}  We recall that $E_{22}(\lambda)^{-1} = -R_{21} (\Qq_h-\lambda)^{-1} R_{12}$. Hence $E_{22}(\lambda)^{-1} = T_1(\lambda) + T_2(\lambda)$, with $T_1(\lambda)$ is analytic in $\lambda \in \Omega'$; and $T_2(\lambda)$ in $\Psi_h^{(22)} \left( \lr{x_2}^{-\infty}  \right)$: 
\begin{equations}\label{eq:4d}
T_1(\lambda) \de -\sum_{\pm} R_{21} \cdot \chi_\pm(x_2)(\Qq_{h,\pm}-\lambda)^{-1} \cdot R_{12},
\\
T_2(\lambda) \de -\sum_{\pm} R_{21} \cdot \chi_\pm(x_2) (\Qq_{h,\pm}-\lambda)^{-1} (\Qq_h-\Qq_{h,\pm}) (\Qq_h-\lambda)^{-1}\cdot R_{12}.
\end{equations}

Pairings of analytic terms with almost-analytic functions vanish. Therefore, 
\begin{equations}\label{eq:0d}
\TT_{1-\Phi} =
\int_{\C^+} \dd{^2\tg(\lambda,\olambda)}{\olambda \p\lambda} \cdot \big[ \Qq_h, f(x_1)\big] \cdot  E_{12}(\lambda) T_2(\lambda) \big( 1-\Phi(x) \big) E_{21}(\lambda)  \cdot \dfrac{d^2\lambda}{\pi}.
\end{equations}

\textbf{3.}  As in \eqref{eq:2g}, $[\Qq_h,f(x_1)] \in \Psi_h^{(11)}\big(\lr{x_1}^{-\infty}\big)$. We deduce that for any $s \in \N$,
\begin{equation}\label{eq:0c}
\hspace*{-2mm} \big[\Qq_h, f(x_1)\big]   E_{12}(\lambda) = \big[\Qq_h, f(x_1)\big]  E_{12}(\lambda) \lr{x_1}^s \cdot \lr{x_1}^{-s} \ \in \ \Psi_h^{(12)}(1) \cdot \Psi_h^{(22)}\big(\lr{x_1}^{-s}\big),
\end{equation}
with uniform bounds as $|\Im \lambda| \rightarrow 0$. 

As in \eqref{eq:2h}, $T_2(\lambda) \in \Psi_h^{(22)}\big(\lr{x_2}^{-\infty}\big)$, with bounds blowing up polynomially with $|\Im \lambda|^{-1}$. Combining with \eqref{eq:0c}, we deduce that for any $s \in \N$,
\begin{equation}\label{eq:0e}
\big[ \Qq_h, f(x_1)\big] \cdot  E_{12}(\lambda) T_2(\lambda) \big( 1-\Phi(x) \big) \ \in \ \Psi^{(12)}_h(1) \cdot \Psi_h^{(22)}\big(\lr{x}^{-s}\big),
\end{equation}
with bounds blowing up polynomially with $|\Im \lambda|^{-1}$. For $s \geq 3$, operators in $\Psi_h^{(22)}\big(\lr{x}^{-s}\big)$ are trace-class on $\HH_2$. Moreover, $\p_\olambda \tg(\lambda,\olambda) = O(|\Im \lambda|^\infty)$. We deduce from \eqref{eq:0d} that $\TT_{1-\Phi}$ is trace class on $\HH_1$.

\textbf{4.}  We take the trace of \eqref{eq:0d} and permute trace and integral. The identity \eqref{eq:0e} allows us to move cyclically  $E_{21}(\lambda)$ from the right to the left. We obtain
\begin{equation}
\Tr_{\HH_1}\big(\TT_{1-\Phi}\big) =
\int_{\C^+} \dd{^2\tg(\lambda,\olambda)}{\olambda \p\lambda} \cdot \Tr_{\HH_2}\Big(E_{21}(\lambda)\big[ \Qq_h, f(x_1)\big] \cdot  E_{12}(\lambda) T_2(\lambda) \big( 1-\Phi(x) \big)  \Big) \cdot \dfrac{d^2\lambda}{\pi}.
\end{equation}
We show that this trace is $O(h^\infty)$ using \eqref{eq:0f}: we prove that the symbol is $O(h^\infty)$.

Fix $N \in \N$. Recall that $f'$ has support in $[-1,1]$. We use
\begin{equation}\label{eq:3q}
[\Qq_h,f(x_1)] = [\vp(\Pp_h),f(x_1)] = \big( 1 -f(x_1) \big)\vp(\Pp_h)f(x_1) - f(x_1) \vp(\Pp_h) \big(1-f(x_1)\big)
\end{equation}
with the composition theorem. This shows that there exists $a_N \in S^{(11)}\big(\lr{x_1}^{-\infty}  \big)$, with support in $[-1, 1] \times \R^3$, such that 
\begin{equation}\label{eq:3t}
\big[\Qq_h,f(x_1)\big] = \Op_h(a_N) + h^N \cdot \Psi_h^{(11)} \left( \lr{x_1}^{-\infty}  \right).
\end{equation}
Via a similar argument, there exists $b_N(\cdot;\lambda) \in S^{(22)}(\lr{x_2}^{-\infty})$, with support in $\R \times [-1, 1] \times \R^2$ and seminorms blowing up at worst polynomially with $|\Im \lambda|^{-1}$,  with
\begin{equation}\label{eq:3u}
T_2(\lambda) = \Op_h\big(b_N(\cdot;\lambda)\big) + h^N \cdot \Psi_h^{(22)} \left( \lr{x_2}^{-\infty} \right).
\end{equation}
It uses the Helffer--Sj\"ostrand formula for $\Qq_h - \Qq_{h,+} = \vp(\Pp_h) - \vp(\Pp_{h,+})$; $\supp \chi_+ \subset [-1,+\infty)$; and that the coefficients of   $\Pp_h-\Pp_{h,+}$ have support in $\R \times (-\infty,1]$.  

We note that the (three-way) intersection of the supports of $a_N, b_N$ and $1-\Phi$ is empty. Using \eqref{eq:3t}, \eqref{eq:3u} and the composition theorem, we deduce that
\begin{equation}\label{eq:3v}
E_{21}(\lambda) \big[ \Qq_h, f(x_1)\big] \cdot  E_{12}(\lambda) T_2(\lambda) \big( 1-\Phi(x) \big)  \ \in \ h^N \cdot \Psi_h^{(22)} \left( \lr{x}^{-\infty}   \right)
\end{equation}
with symbolic bounds blowing up polynomially with $|\Im \lambda|^{-1}$. In particular the $\HH_2$-trace of \eqref{eq:3v} is $O(h^N |\Im\lambda|^{-\az_N})$ for some $\az_N> 0$. This completes the proof because $\p^2_{\olambda\lambda} \tg(\lambda,\olambda) = O(|\Im \lambda|^\infty)$.
\end{proof}

Applying Lemma \ref{lem:1k}, we split
\begin{equation}
[\Qq_h,f(x_1)]g'(\Qq_h) = \TT_\Phi + \TT_{1-\Phi}
\end{equation}
where both $\TT_\Phi$ and $\TT_{1-\Phi}$ are trace-class; it proves that $\JJ_e(P_-,P_+) = \Tr_{\HH_1}(\TT_\Phi) + O(h^\infty)$. 

The operator $\TT_\Phi$ is an integral involving $\p^2_{\olambda\lambda} \tg(\lambda,\olambda)$ and
\begin{equation}\label{eq:3p}
\big[\Qq_h,f(x_1)\big]E_{12}(\lambda) E_{22}(\lambda)^{-1} \Phi(x) E_{21}(\lambda).
\end{equation}
We observe that $\Phi(x) \in \Psi_h^{(22)} \left( \lr{x}^{-\infty}  \right)$, thus it is trace-class on $\HH_2$. The other operators in \eqref{eq:3p} are bounded with bounds blowing up polynomially with $|\Im \lambda|^{-1}$. Since $\p^2_{\olambda\lambda} \tg(\lambda,\olambda) = O(|\Im \lambda|^\infty)$, we can permute trace and integral in $\Tr_{\HH_1}(\TT_\Phi)$. Thus,
\begin{equation}
\Tr_{\HH_1}(\TT_\Phi) = 
 - \int_{\C^+} \dd{^2\tg(\lambda,\olambda)}{\olambda \p\lambda} \cdot \Tr_{\HH_1} \Big( \big[\Qq_h,f(x_1)\big] \cdot E_{12}(\lambda) E_{22}(\lambda)^{-1} \Phi(x) E_{21}(\lambda) \Big) \cdot \dfrac{d^2\lambda}{\pi}.
\end{equation} 
We move the term $E_{21}(\lambda)$ cyclically and end up with
\begin{equation}\label{eq:4h}
 - \int_{\C^+} \dd{^2\tg(\lambda,\olambda)}{\olambda \p\lambda} \cdot \Tr_{\HH_2} \Big( E_{21}(\lambda) \big[\Qq_h,f(x_1)\big]E_{12}(\lambda) \cdot E_{22}(\lambda)^{-1} \Phi(x) \Big) \cdot \dfrac{d^2\lambda}{\pi}.
\end{equation} 
To summarize, $\JJ_e(P_-,P_+)$ equals \eqref{eq:4h} modulo $O(h^\infty)$. In a sense, the next result extends the definition \eqref{eq:5u}  of $\JJ_e(P)$ to singular value problems:

\begin{thm}\label{thm:1} We have
\begin{equations}
\JJ_e(P_-,P_+)
 =
 \int_{\C^+} \dd{^2\tg(\lambda,\olambda)}{\olambda \p\lambda} \cdot \Tr_{\HH_2} \Big( \big[E_{22}(\lambda),f(x_1)\big] E_{22}(\lambda)^{-1} \Phi(x) \Big) \cdot \dfrac{d^2 \lambda}{\pi} + O(h^\infty).
\end{equations} 
\end{thm}

\begin{proof} \textbf{1.}  The starting point is \eqref{eq:4h}. We use the matrix identity \eqref{eq:1m}  for the $(1,2)$ and $(2,1)$ components. It yields
\begin{equations}
E_{21}(\lambda) \big[\Qq_h,f(x_1)\big]E_{12}(\lambda) = E_{21}(\lambda) (\Qq_h-\lambda) f(x_1) E_{12}(\lambda) - E_{21}(\lambda) f(x_1) (\Qq_h-\lambda) E_{12}(\lambda)
\\
 = -E_{22}(\lambda) R_{21} f(x_1) E_{12}(\lambda) + E_{21}(\lambda) f(x_1) R_{12} E_{22}(\lambda).
\end{equations}
Then we use \eqref{eq:1m} for the $(2,2)$ component. 
 This gives
\begin{equations}
E_{21}(\lambda) \big[\Qq_h,f(x_1)\big]E_{12}(\lambda) 
\\
-E_{22}(\lambda)  f(x_1)  +  f(x_1)  E_{22}(\lambda) -E_{22}(\lambda) \big[ R_{21}, f(x_1)\big] E_{12}(\lambda) + E_{21}(\lambda) \big[f(x_1), R_{12}\big] E_{22}(\lambda)
\\
=
\big[f(x_1), E_{22}(\lambda)\big] -E_{22}(\lambda) \big[ R_{21}, f(x_1)\big] E_{12}(\lambda) + E_{21}(\lambda) \big[f(x_1), R_{12}\big] E_{22}(\lambda).
\end{equations}
We multiply on both sides by $E_{22}(\lambda)^{-1} \Phi(x)$ to end up with
\begin{equations}\label{eq:4b}
E_{21}(\lambda) \big[\Qq_h,f(x_1)\big]E_{12}(\lambda) E_{22}(\lambda)^{-1} \Phi(x) 
=
\big[f(x_1), E_{22}(\lambda)\big] E_{22}(\lambda)^{-1} \Phi(x)
 \\
-
E_{22}(\lambda) \big[ R_{21}, f(x_1)\big] E_{12}(\lambda) E_{22}(\lambda)^{-1} \Phi(x) + E_{21}(\lambda) \big[f(x_1), R_{12}\big] \Phi(x).
\end{equations}

\textbf{2.}  The function $\Phi$ has compact support. Therefore it induces a trace-class operator on $\HH_2$. This allows us to separately trace each term in \eqref{eq:4b}. The third trace is
\begin{equation}\label{eq:4i}
\Tr_{\HH_2} \Big( E_{21}(\lambda) \big[f(x_1), R_{12}\big] \Phi(x) \Big).
\end{equation}
It is analytic for $\lambda \in \Omega$; an integration by parts with respect to $\olambda$  gets rid of it.
 
\textbf{3.}  We focus on the second trace:
\begin{equation}\label{eq:4c}
\Tr_{\HH_2} \Big( E_{22}(\lambda) \big[ R_{21}, f(x_1)\big] E_{12}(\lambda) E_{22}(\lambda)^{-1} \Phi(x) \Big).
\end{equation}
We move $E_{22}(\lambda)$ cyclically to the right and commute it with $\Phi(x)$. The term $E_{22}(\lambda)^{-1}$ cancels out with $E_{22}(\lambda)$, producing an analytic term. Only the commutator produces non-analytic terms. In other words,  \eqref{eq:4c} equals 
\begin{equation}
\Tr_{\HH_2} \Big(  \big[ R_{21}, f(x_1)\big] E_{12}(\lambda) E_{22}(\lambda)^{-1} \big[\Phi(x), E_{22}(\lambda)\big]\Big),
\end{equation}
modulo an analytic function. 

We recall that $E_{22}(\lambda)^{-1}$ splits as $T_1(\lambda) + T_2(\lambda)$, where $T_1(\lambda)$, $T_2(\lambda)$ are defined in \eqref{eq:4d}. Since  $T_1(\lambda)$ is analytic in $\lambda$, an integration by parts w.r.t. $\olambda$ replaces \eqref{eq:4c} by
\begin{equation}
\Tr_{\HH_2} \Big(  \big[ R_{21}, f(x_1)\big] E_{12}(\lambda) T_2(\lambda) \big[\Phi(x), E_{22}(\lambda)\big]\Big).
\end{equation}

\textbf{4.}  Fix $N \in \N$. Since $\Phi$ has compact support and $\Phi'$ vanishes in $[-1,1]^2$, there exists $c_N(\cdot;\lambda) \in \Psi_h^{(22)}(1)$ with compact support, vanishing in $[-1,1]^2$, analytic in $\lambda$ such that
\begin{equations}\label{eq:4e}
\big[\Phi(x), E_{22}(\lambda)\big] = \Op_h\big(c_N(\cdot;\lambda)\big) + h^N \cdot \Psi_h^{(22)} \left( \lr{x}^{-\infty} \right).
\end{equations}
From \eqref{eq:3u}, there exists $b_N(\cdot;\lambda) \in  S^{(22)}(1)$, with support in $\R \times [-1, 1] \times \R^2$ and with seminorms blowing up at worst polynomially with $|\Im \lambda|^{-1}$,  such that 
\begin{equation}\label{eq:4f}
T_2(\lambda) = \Op_h\big(b_N(\cdot;\lambda)\big) + h^N \cdot \Psi_h^{(22)} (1).
\end{equation}
Finally, there exists $d_N(\cdot;\lambda) \in \Psi_h^{(22)}(1)$ with support in $[-1,1] \times \R^3$ such that
\begin{equation}\label{eq:4g}
\big[ R_{21}, f(x_1)\big] E_{12}(\lambda) = \Op_h\big(d_N(\cdot;\lambda)\big) + h^N \cdot \Psi_h^{(22)} (1).
\end{equation}

We remark that $b_N(\cdot;\lambda)$, $c_N(\cdot;\lambda)$ and $d_N(\cdot,\lambda)$ have disjoint supports. 
From the composition theorem applied to \eqref{eq:4e}, \eqref{eq:4f} and \eqref{eq:4g} we deduce that
\begin{equation}
\big[ R_{21}, f(x_1)\big] E_{12}(\lambda) T_2(\lambda) \big[\Phi(x), E_{22}(\lambda)\big] \in h^N \cdot \Psi_h^{(22)} \left( \lr{x}^{-\infty} \right)
\end{equation}
with bounds blowing up at worst polynomially with $|\Im \lambda|^{-1}$. Therefore, there exists $\az_N > 0$ such that
\begin{equation}\label{eq:4j}
\Tr_{\HH_2} \Big(  \big[ R_{21}, f(x_1)\big] E_{12}(\lambda) T_2(\lambda) \big[\Phi(x), E_{22}(\lambda)\big]\Big) = O\left( \dfrac{h^N}{|\Im \lambda|^{\az_N}} \right).
\end{equation}

\textbf{6.}  We go back to \eqref{eq:4b} and we recall that \eqref{eq:4i} is analytic. Moreover, \eqref{eq:4c} equals $O\left( h^N|\Im \lambda|^{-\az_N} \right)$ modulo an analytic term. Since $\p_\lambda\tg$ is almost analytic, we deduce that for any $N$, $\JJ_e(P_-,P_+)$ equals 
\begin{equations}
 - \int_{\C^+} \dd{^2\tg(\lambda,\olambda)}{\olambda \p\lambda}\cdot \Tr_{\HH_2} \Big( E_{21}(\lambda) \big[\Qq_h,f(x_1)\big]E_{12}(\lambda) \cdot E_{22}(\lambda)^{-1} \Phi(x) \Big) \cdot \dfrac{d^2 \lambda}{\pi} + O(h^\infty)
 \\
 =
 \int_{\C^+} \dd{^2\tg(\lambda,\olambda)}{\olambda \p\lambda} \cdot \Tr_{\HH_2} \Big( \big[E_{22}(\lambda),f(x_1)\big] E_{22}(\lambda)^{-1} \Phi(x) \Big) \cdot \dfrac{d^2 \lambda}{\pi} + O(h^N).
\end{equations} 
This completes the proof.
\end{proof}

We next use Theorem \ref{thm:1} to express $\JJ_e(P_-,P_+)$ in terms of asymptotic quantities. We first introduce the asymptotic leading symbols of the effective Hamiltonian:
\begin{equations}\label{eq:1y}
E_\pm(\xi;\lambda) \de - R_{21}(\xi) \big( \Qq_\pm(\xi) - \lambda\big)^{-1} R_{12}(\xi), \ \ \ \ \text{where} 
\\
\Qq_\pm(\xi) \de \psi\big(\Pp_\pm(\xi)\big), \ \ \ \ \Pp_\pm(\xi) \de \sum_{|\az| \leq 2} a_{\az,\pm}(y) (D_y+\xi)^\az.
\end{equations}
We define an index for $E_\pm(\xi;\lambda)$:
\begin{equation}\label{eq:5d}
\JJ(E_\pm) = -\int_{\p\Omega} \int_{\Tt^2_*}\Tr_{\C^d} \left(\left(\dd{E_\pm}{\xi_1}  E_\pm^{-1} \dd{ \big(\p_\lambda E_\pm \cdot E_\pm^{-1}\big)}{\xi_2} \right)(\xi;\lambda) \right)  \dfrac{d\xi}{(2\pi)^2} \dfrac{d^1\lambda}{2i\pi},
\end{equation}
where we recall that $\Tt^2_* = (\Tt^2)^*$ is the two-torus $\R^2/(2\pi\Z)^2$.

\begin{theorem}\label{thm:2} We have
\begin{equation}
\JJ_e(P_-,P_+) = \JJ(E_+) - \JJ(E_-).
\end{equation}
\end{theorem}

\begin{proof} \textbf{1.}  Theorem \ref{thm:1} shows that modulo lower order terms, $\JJ_e(P_-,P_+)$ is equal to
\begin{equation}
\int_{\C^+}\dd{^2\tg(\lambda,\olambda)}{\lambda \p \olambda} \cdot \Tr_{\HH_2} \Big( \big[E_{22}(\lambda),f(x_1)\big] E_{22}(\lambda)^{-1} \Phi(x) \Big) \cdot \dfrac{d^2\lambda}{\pi}.
\end{equation}
We integrate by parts with respect to $\lambda$. This produces the term
\begin{equations}
\dd{}{\lambda} \Tr_{\HH_2} \Big( \big[E_{22}(\lambda),f(x_1)\big] E_{22}(\lambda)^{-1} \Phi(x) \Big) = t_1(\lambda) + t_2(\lambda),  \ \ \ \ \  \ \text{where:}
\\
t_1(\lambda) \de  \Tr_{\HH_2} \Big( \big[E_{22}'(\lambda),f(x_1)\big] E_{22}(\lambda)^{-1} \Phi(x) \Big) 
\\
t_2(\lambda) \de - \Tr_{\HH_2} \Big( \big[E_{22}(\lambda),f(x_1)\big] E_{22}(\lambda)^{-1} E_{22}'(\lambda)E_{22}(\lambda)^{-1} \Phi(x) \Big).
\end{equations}
In $t_2(\lambda)$, we commute $E_{22}'(\lambda)E_{22}(\lambda)^{-1}$ with $\Phi(x)$, then move it cyclically to the left. This shows that $t_2(\lambda) = t_3(\lambda) + t_4(\lambda)$, where
\begin{equations}
t_3(\lambda) \de - \Tr_{\HH_2} \Big( \big[E_{22}(\lambda),f(x_1)\big] E_{22}(\lambda)^{-1} \big[E_{22}'(\lambda)E_{22}(\lambda)^{-1}, \Phi(x)\big] \Big),
\\
t_4(\lambda) \de - \Tr_{\HH_2} \Big( E_{22}'(\lambda)E_{22}(\lambda)^{-1}\big[E_{22}(\lambda),f(x_1)\big] E_{22}(\lambda)^{-1} \Phi(x)\big] \Big) 
\\
= \Tr_{\HH_2} \Big( E_{22}'(\lambda) \big[E_{22}(\lambda)^{-1},f(x_1)\big]  \Phi(x) \Big).
\end{equations}
In particular, $t_1(\lambda) + t_4(\lambda) = 0$ since it is the trace of trace-class commutators: 
\begin{equations}
t_1(\lambda) + t_4(\lambda) = \Tr_{\HH_2} \Big(  \big[E_{22}'(\lambda)E_{22}(\lambda)^{-1},f(x_1)\big]  \Phi(x) \Big) 
\\
= \Tr_{\HH_2} \Big(  \big[E_{22}'(\lambda)E_{22}(\lambda)^{-1}\Phi(x),f(x_1)\big]   \Big) = 0.
\end{equations}
We conclude that modulo $O(h^\infty)$, $\JJ_e(P_-,P_+)$ is equal to
\begin{equation}\label{eq:4z}
\int_{\C^+}\dd{\tg(\lambda,\olambda)}{\olambda} \cdot \Tr_{\HH_2} \Big( \big[E_{22}(\lambda),f(x_1)\big] E_{22}(\lambda)^{-1} \big[E_{22}'(\lambda)E_{22}(\lambda)^{-1}, \Phi(x)\big] \Big) \cdot \dfrac{d^2\lambda}{\pi}.
\end{equation}

\textbf{2.}  We use Dimassi's formula \eqref{eq:0f}. The equation \eqref{eq:4z} becomes 
\begin{equation}\label{eq:4o}
\JJ_e(P_-,P_+) = 
 \int_{\C^+}\dd{\tg(\lambda,\olambda)}{\olambda} \cdot \int_{\R^2 \times \Tt^2_*}\Tr_{\C^d} \big( \sigma(x,\xi;\lambda) \big)  \dfrac{dx d\xi}{(2\pi h)^2} \cdot \dfrac{d^2\lambda}{\pi} + O(h^\infty),
\end{equation}
where $\sigma(\cdot;\lambda) \in S^{(22)}(1)$ is the symbol of 
\begin{equation}\label{eq:4n}
\big[E_{22}(\lambda),f(x_1)\big] E_{22}(\lambda)^{-1} \big[E_{22}'(\lambda)E_{22}(\lambda)^{-1}, \Phi(x)\big].
\end{equation}

\textbf{3.}  Because of \eqref{eq:0g} and \eqref{eq:3x}, we can write 
\begin{equations}
E_{22}(\lambda) = \Op_h\big(E(\cdot;\lambda)\big) + O_{\Psi^{(22)}_h(1)}\big(h |\Im \lambda|^{-8}\big), \ \ \ \ \text{where} 
\\
 E(x,\xi;\lambda) \de -R_{21}(\xi) \big( \Qq(x,\xi)-\lambda \big)^{-1} R_{12}(\xi).
\end{equations}
The composition theorem applied to \eqref{eq:4n} shows that \begin{equations}\label{eq:5a}
\sigma(\cdot;\lambda) = h^2\sigma_0(\cdot;\lambda)+ O_{S^{(22)}(1)}\big(h^3 |\Im \lambda|^{-16}\big) \ \ \ \ \text{where}
\\
\sigma_0(x,\xi;\lambda) \de
\left(\dfrac{1}{i} \big\{ E, f  \big\} E^{-1} \cdot \dfrac{1}{i} \big\{ \p_\lambda E \cdot E^{-1}, \Phi \big\}\right)(x,\xi;\lambda).
\end{equations}
We observe that neither $\sigma_0$ not $\JJ_e(P_-,P_+)$ depend on $h$. Hence, \eqref{eq:4o} reduces to a $h$-independent formula:
\begin{equation}\label{eq:4p}
\JJ_e(P_-,P_+) = 
 \int_{\C^+}\dd{\tg(\lambda,\olambda)}{\olambda} \cdot \int_{\R^2 \times \Tt^2_*}\Tr_{\C^d} \big( \sigma_0(x,\xi;\lambda) \big)  \dfrac{dx d\xi}{(2\pi)^2} \cdot \dfrac{d^2\lambda}{\pi}.
\end{equation}

\textbf{4.}  We simplify the expression \eqref{eq:5a} for $\sigma_0$. Recall that $\Phi(x) = \Phi_0(x_1) \Phi_0(x_2)$ where $\Phi_0$ is equal to $1$ on $[-1,1]$. For $x \in \supp(f')$, $\p_{x_1} \Phi (x) = \Phi_0(x_2) \Phi_0'(x_1) = 0$. The support of $\p_{x_2} \Phi$ does not intersect the strip $\R \times (-1,1)$. Therefore we can write
\begin{equation}
\supp(\p_{x_2} \Phi) = S_+ \cup S_-, \ \ \ \ S_\pm \de \supp(\p_{x_2} \Phi) \cap \{ \pm x_2 \geq 1\}. 
\end{equation}
For $x \in S_\pm$, $E(x,\xi;\lambda) =  E_\pm(\xi;\lambda)$.
We deduce that for $\pm x_2 > 0$, 
\begin{equations}
\sigma_0(x,\xi;\lambda) = - \left(\dd{E_\pm}{\xi_1}  E_\pm^{-1} \dd{ \big(\p_\lambda E_\pm \cdot E_\pm^{-1}\big)}{\xi_2} \right)(\xi;\lambda) \cdot \dd{f(x_1)}{x_1}  \dd{ \Phi(x)}{x_2}.
\end{equations}

\textbf{5.}  On the support of $f'$, $\Phi(x) = \Phi_0(x_1) \Phi_0(x_2) = \Phi_0(x_2)$. Therefore
\begin{equation}
\int_{\R \times \R^\pm} \dd{f(x_1)}{x_1} \dd{\Phi(x)}{x_2} dx = \int_{\R^\pm} \dd{\Phi_0(x_2)}{x_2} dx_2 = \mp 1.
\end{equation}
It follows that
\begin{equations}\label{eq:4q}
F(\lambda) \de \int_{\R^2 \times \Tt^2_*}\Tr_{\C^d} \big( \sigma_0(x,\xi;\lambda) \big) \dfrac{dx d\xi}{(2\pi)^2}
\\
 = \sum_\pm \pm \int_{\Tt^2_*}\Tr_{\C^d} \left(\left(\dd{E_\pm}{\xi_1}  E_\pm^{-1} \dd{ \big(\p_\lambda E_\pm \cdot E_\pm^{-1}\big)}{\xi_2} \right)(\xi;\lambda) \right)  \dfrac{d\xi}{(2\pi)^2}.
\end{equations}

\textbf{6.}  From \eqref{eq:4p} and the definition \eqref{eq:4q} of $F$,
\begin{equation}\label{eq:4r}
\JJ_e(P_-,P_+) = 
 \int_{\C^+}\dd{\tg(\lambda,\olambda)}{\olambda} \cdot F(\lambda) \cdot \dfrac{d^2\lambda}{\pi}.
\end{equation}
We remove the dependence in $\tg$. We observe that $F$ is meromorphic in $\lambda \in \Omega$. Assume now that $\lambda_\star \in \Omega$ is a pole of $F$. Then there exists $\xi_\star$ such that $E_+^{-1}$ or $E_-^{-1}$ has a pole at $(\xi_\star;\lambda_\star)$. From \eqref{eq:1y}, $\lambda_\star \in \sigma_{L^2(\Tt^2)}\big(\Qq_+(\xi)\big) \cup \sigma_{L^2(\Tt^2)}\big(\Qq_-(\xi)\big)$. From the spectral gap assumption, $g(\lambda) = 1$ near $\lambda_\star$. 

We integrate \eqref{eq:4r} by parts with respect to $\olambda$:
\begin{equation}\label{eq:4s}
\JJ_e(P_-,P_+) = 
 -\int_{\C^+}\tg(\lambda,\olambda) \cdot \dd{F(\lambda)}{\olambda} \cdot \dfrac{d^2\lambda}{\pi}
\end{equation}
where we see $\p_\olambda F$ as a distribution, whose singular support is within poles of $F$. To simplify \eqref{eq:4s}, we expand $F$ as a Lorenz series near the pole $\lambda_\star$. From the identities
\begin{equation}
\dfrac{(-1)^j}{\pi}\dd{}{\olambda}(\lambda-\lambda_\star)^{-j-1}  =   \delta^{(j)}(\lambda-\lambda_\star) \ \ \text{on} \ \ \DD'(\C^+); \ \ \ \ g^{(j)}(\lambda_\star) = 
 \systeme{1  \ \text{ if } \ j =0 \\ 0 \ \text{ if } \ j \geq 1},
\end{equation}
we see that only terms of the form $(\lambda-\lambda_\star)^{-1}$ in the Lorenz development of $F$ contribute to \eqref{eq:4s}; and more precisely,
\begin{equation}\label{eq:1v}
\JJ_e(P_-,P_+) = -\sum_{\lambda_\star} \Res(F,\lambda_\star) = -\int_{\p\Omega} F(\lambda) \dfrac{d^1\lambda}{2i\pi}.
\end{equation}
This completes the proof: \eqref{eq:5d} appear in the RHS of \eqref{eq:1v}. \end{proof}

\section{Relation to the Chern index}\label{sec:4}

In this section we complete the proof of Theorem \ref{thm:0}. Deriving the Chern number $c_1(\EE_\pm)$ from $\JJ(E_\pm)$ turns out to be surprisingly involved. We first prove
\begin{equation}
\JJ_e(P_-,P_+) = \JJ_e(P_-,P_0) + \JJ_e(P_0,P_+),
\end{equation}
where $P_0$ was defined in \S\ref{sec:24}.
This formula explicitly splits $\JJ_e(P_-,P_+)$ in components for $x_2 \geq 1$ and $x_2 \leq -1$. This allows us to pick separate effective Hamiltonians for $x_2 \leq -1$ and $x_2 \geq 1$. Hence, Theorem \ref{thm:0} follows from the separate (and similar) computation of $\JJ_e(P_-,P_0)$ and $\JJ_e(P_0,P_+)$.

Theorem \ref{thm:2} computes $\JJ_e(P_0,P_+)$ in terms of an effective Hamiltonian $E_{22}(\lambda)$. There remains quite a bit of  flexibility in the choice of $E_{22}(\lambda)$. We design an effective Hamiltonian $\tE_{22}(\lambda)$ that is suitable for the computation of $\JJ_e(P_0,P_+)$. This tremendously simplifies the derivation of $2i\pi \cdot \JJ_e(P_0,P_+) = c_1(\EE_+)$. 

However $\tE_{22}(\lambda)$ may not be suitable for the calculation of $\JJ_e(P_-,P_0)$. But the same approach will design another suitable effective Hamiltonian and prove $2i\pi \cdot \JJ_e(P_-,P_0) = -c_1(\EE_-)$. This will complete the proof of Theorem \ref{thm:0}.

\subsection{Chern number}\label{sec:41} We review how to define the bulk index of $P_+$. We recall that $\lambda_0 \notin \sigma_{L^2(\R^2)}(P_+)$. By Floquet--Bloch theory \cite{RS:78,Ku:16}, for any $\xi \in \R^2$, $\lambda_0 \notin \sigma_{L^2(\Tt^2)}\big( P_+(\xi)\big)$, where
\begin{equation}
P_+(\xi) \de \sum_{|\az| \leq 2} a_{\az,+}(y) (D_y+\xi)^\az \ : \ L^2(\Tt^2) \rightarrow L^2(\Tt^2).
\end{equation}
Let $\FF_+$ be the smooth vector bundle over $\R^2$ whose fiber at $\xi \in \R^2$ is
\begin{equation}
\FF_+(\xi) \de \Range\big( \Pi_+(\xi) \big), \ \ \ \ \Pi_+(\xi) \de \1_{(-\infty,\lambda_0]} \big( P_+(\xi)\big).
\end{equation}

For any $k \in \Z^2$ and $\xi \in \R^2$, 
\begin{equations}\label{eq:1g}
e^{-2ik\pi y} P_+(\xi) e^{2ik\pi y} = P_+(\xi+2k\pi), \ \ \ \ e^{-2ik\pi y} \Pi_+(\xi) e^{2ik\pi y} = \Pi_+(\xi+2k\pi), \ \ \ \ \text{thus}
\\
k \in \Z^2, \ \ \xi \in \R^2 \ \ \Rightarrow \  \ \FF_+(\xi+2k\pi) = e^{-2ik\pi y} \cdot\FF_+(\xi). 
\end{equations}
These relations show that $\FF_+ \rightarrow \R^2$ induces   a bundle $\EE_+ \rightarrow \Tt^2_*$, defined as
\begin{equation}\label{eq:0h}
\EE_+ \de \FF_+/_\sim, \ \ \ \ \text{where} \ \ (\xi,v) \sim (\xi',v') \ \ \Leftrightarrow \ \ \systeme{ \xi-\xi' \in (2\pi \Z)^2 \\ e^{i\xi y}v = e^{i\xi'y}v },
\end{equation}
see e.g. \cite[\S2]{Pa:07}. Another way to define $\EE_+$ consists of looking at $P_+$ on spaces of pseudoperiodic functions, see e.g. \cite[\S2]{Dr:19b}.

Complex vector bundles over $\Tt^2_*$ are classified by their rank and their Chern number, $c_1(\EE_+)$. This integer is defined by integrating a curvature on $\EE_+$. Analogously to the Gauss--Bonnet theorem, the final result does not depend on the choice of curvature: it is a topological invariant. Taking for instance the Berry curvature \cite{Be:84,Si:83},
\begin{equation}
c_1(\EE_+) = \dfrac{i}{2\pi} \int_{\Tt^2_*} \Tr_{L^2(\Tt^2)} \Big( \Pi_+(\xi) \big[ \p_1\Pi_+(\xi), \p_2 \Pi_+(\xi) \big] \Big) d\xi.
\end{equation}

If $c_1(\EE_+) = 0$, then $\EE_+$ is trivial: it admits a smooth orthonormal equivariant frame -- i.e. transforming like \eqref{eq:0h}. In other words, there exist smooth functions $\vp_1(y,\xi), \dots, \vp_n(y,\xi) \in C^\infty(\Tt^2 \times \R^2)$, such that for any $\xi$,
\begin{equations}
\EE_+(\xi) = \big[ \vp_1(\cdot,\xi), \dots, \vp_n(\cdot,\xi) \big] \ \ \ \text{and} 
\\ 
\vp_j(y,\xi+2k\pi) = e^{-2ik\pi y}, \ \ k \in \Z^2; \ \ \ \ \big\langle \vp_j(\cdot,\xi), \vp_\ell(\cdot,\xi) \big\rangle_{L^2(\Tt^2)} = \delta_{jl}.
\end{equations}
We refer to \cite[\S3]{Mo:17} for the proof.

\subsection{A concatenation formula}\label{sec:6.1} Let $\rho_0(x_2) \in C^\infty(\R,\R)$ be independent of $x_1$, with
\begin{equation}
\rho_0(x) = \systeme{ 0 & \text{ for }  x_2 \leq 0 \\ x_2 & \text{ for } x_2 \geq 1}.
\end{equation}
We define $\rho(x) = \big(0,\rho_0(x_2)\big)$ and we set
\begin{equation}
\tP_h \de \Re\Bigg( \sum_{|\az| \leq 2} c_\az(\rho(hx),x) D_x^\az \Bigg) \ : \ L^2(\R^2) \rightarrow L^2(\R^2)
\end{equation}
The same arguments as \S\ref{sec:24} show that $\tP_h$ is a selfadjoint elliptic operator of order $2$. We observe that $c_\az(\rho(hx),x) = c_\az(0,x)$ for $x_2 \leq 0$; and $c_\az(\rho(hx),x) = c_\az(hx,x)$ for $hx_2 \geq 1$. Thus, the asymptotics of $\tP_h$ for $hx_2 \geq 1$ and $hx_2 \leq 0$ are respectively $P_0$ -- see \eqref{eq:1c} -- and $P_+$. We deduce that $\JJ_e(P_0,P_+) = \JJ_e(\tP_h)$.

Moreover, $\tP_h$ is unitarily equivalent to
\begin{equation}
\tPp_h \de \Re\Bigg(\sum_{|\az| \leq 2} c_\az(\rho(x),y) (D_y+hD_x)^\az \Bigg) \ : \ \HH_1 \rightarrow \HH_1.
\end{equation}
This operator is semiclassical, with leading symbol
\begin{equation}
\tPp(x,\xi) \de \Re\Bigg(\sum_{|\az| = 2} c_\az(\rho(x),y) (D_y+\xi)^\az \Bigg) = \Re\big( \Pp(\rho(x),\xi) \big) = \Pp\big(\rho(x),\xi\big).
\end{equation}

In particular, if $\vp_1, \dots, \vp_d$ satisfy \eqref{eq:1k} and \eqref{eq:1f}, then they also satisfy
\begin{equation}\label{eq:1q}
u \in \big[ \vp_1(\cdot,\xi), \dots, \vp_d(\cdot,\xi)\big]^\perp \ \Rightarrow \ 
 \big\langle \big( \tPp(x,\xi) - \lambda_+) u ,u \big\rangle_{L^2(\Tt^2)} \geq 3 |u|^2_{L^2(\Tt^2)}.
\end{equation}
Thus (a) we can construct an effective Hamiltonian $\tE_{22}(\lambda)$ for $\tPp$, with  leading symbol $\tE(x,\xi;\lambda) = E\big(\rho(x),\xi;\lambda\big)$; (b) we can apply Theorem \ref{thm:2} and get 
\begin{equations}\label{eq:1o}
\JJ_e(P_0,P_+) = \JJ\big(\tE_+\big) - \JJ\big(\tE_-\big), \ \ \ \ \text{where}
\\
\tE_-(\xi;\lambda) \de E(0,\xi;\lambda), \ \ \ \ \tE_+(\xi;\lambda) \de E_+(\xi;\lambda).
\end{equations}

The key point is $\tE_-(\xi;\cdot) = (\lambda-\lambda_2) \cdot \Id_{\C^d}$. Indeed, from \eqref{eq:1c}, $\Pp(0,\xi) = (D_y+\xi)^2 + |\lambda_0| + 2$. In particular, 
\begin{equation}
\sigma_{L^2(\R^2)}\big(\Pp(0,\xi)\big) \subset \big[ |\lambda_0|+2,\infty \big).
\end{equation}
This implies $\psi\big(\Pp(0,\xi)\big) = \lambda_2 \cdot \Id_{L^2(\Tt^2)}$. We deduce from \eqref{eq:3x} that
\begin{equation}
\tE_-(\xi;\lambda)^{-1} = -R_{21}(\xi) \big( \psi\big(\Qq(0,\xi)\big)-\lambda \big)^{-1} R_{12}(\xi) = (\lambda-\lambda_2)^{-1} \cdot \Id_{\C^d}.
\end{equation}
Hence $\tE_-(\xi;\cdot) = (\lambda-\lambda_2) \cdot \Id_{\C^d}$.

From \eqref{eq:5d}, $\JJ(\tE_-) = 0$. From \eqref{eq:1o}, $\JJ(\tE_+) = \JJ(E_+)$. Therefore, $\JJ_e(P_0,P_+) = \JJ(E_+)$. The same analysis applies to the pair $(P_-,P_0)$ and yields $\JJ_e(P_-,P_0) = -\JJ(E_-)$. We conclude that
\begin{equation}\label{eq:1s}
 \JJ_e(P_-,P_+) =  \JJ(E_+) -  \JJ(E_-) = \JJ_e(P_-,P_0) + \JJ(P_0,P_+).
\end{equation}

In \S\ref{sec:43} we design a good effective Hamiltonian for the pair $(P_0,P_+)$. Thanks to \eqref{eq:1s}, it does not need to  be also good for the pair $(P_-,P_0)$. The choice of \S\ref{sec:43} tremendously simplifies the derivation of the Chern numbers in \S\ref{sec:44}:
\begin{equation}
2i\pi \cdot \JJ_e(P_0,P_+) = 2i\pi \cdot \JJ\big(\tE_+\big) = c_1(\EE_+). 
\end{equation}
The same argument proves  $2i\pi \cdot \JJ_e(P_-,P_0) = -c_1(\EE_-)$ and ends the proof of~Theorem~\ref{thm:0}.

\subsection{A convenient effective Hamiltonian}\label{sec:43}

\begin{lem}\label{lem:1l} There exist $\vp_1, \dots, \vp_d$ satisfying \eqref{eq:1k} such that for all $(x,\xi) \in \R^2$,
\begin{equations}\label{eq:1p}
\Range\big( \Pi_+(\xi) \big) \ \subset \ \big[ \vp_1(\xi), \dots, \vp_d(\xi) \big];
\\
u \in \big[ \vp_1(\xi), \dots, \vp_d(\xi) \big]^\perp \ \ \Rightarrow \ \ \blr{\big(P(x,\xi) -\lambda_+\big)u,u}_{L^2(\Tt^2)} \geq 3|u|_{L^2(\Tt^2)}^2.
\end{equations}
\end{lem}

\begin{proof} \textbf{1.}  Since the fibers of $\FF_+$ have (constant) finite dimensions and are contained in $H^2(\Tt^2)$, there exists $C > 0$ such that
\begin{equation}\label{eq:4y}
\xi \in [0,2\pi]^2, \ \ 
u \in \FF_+(\xi) \ \ \Rightarrow \ \  |\Delta u|_{L^2(\Tt^2)} \leq C |u|_{L^2(\Tt^2)}.
\end{equation}

Let $\nu = c_1(\EE_+)$. We construct $u_0(\xi,y) \in C^\infty(\R^2,\Tt^2)$ such that the line bundle $\FF_0 \rightarrow \R^2$ satisfies:
\begin{itemize}
\item[\textbf{(i)}] $\FF_0(\xi+2k\pi) = e^{-2ik\pi y} \cdot \FF_0(\xi)$ when $\xi \in \R^2, k \in \Z^2$ -- i.e. $\FF_0$ is equivariant;
\item[\textbf{(ii)}] The induced bundle $\EE_0 \rightarrow \Tt^2_*$ (see \S\ref{sec:41}) has Chern number $-\nu = -c_1(\EE_+)$;
\item[\textbf{(iii)}] $\FF_0$ and $\FF_+$ are in direct sum.
\end{itemize}

To this end, let $a(\xi) \in C^\infty(\R^2,\C^2)$ such that the line $\C a(\xi)$ induces a vector bundle $\C a \rightarrow \Tt^2_*$, with Chern number $-\nu$. We prove the existence of $a$ in Appendix \ref{app:1}. Fix $v_1(y,\xi), v_2(y,\xi) \in C^\infty(\R^2 \times \R^2)$ such that
\begin{equation}
v(y+\ell,\xi+2\pi k) = e^{-2i\pi k y} \cdot v_j(y,\xi), \ \ \ \ (\ell,k)\in \Z^2.
\end{equation}
Let $m_0 \in \N^2$ sufficiently large (in a sense specified below); and define 
\begin{equation}
u_0(y,\xi) \de e^{2i\pi m_0 y} \cdot \sum_{j=1}^2 a_j(\xi) \cdot v_j(y,\xi).
\end{equation}
We note that $|u_0|_{L^2(\Tt^2)} = 1$. The bundle $\FF_0  = \C u_0 $ over $\R^2$ is equivariant and isomorphic to $\C a \rightarrow \R^2$. Thus it induces a bundle $\EE_0 \rightarrow \Tt^2_*$ with Chern number $-\nu$; this proves \textbf{(i)} and \textbf{(ii)} above. 

We note that $|\Delta u_0|_{L^2_0} = (2\pi m_0)^2 + O(|m_0|)$ for large $m_0$. We adjust $m_0$ so that $|\Delta u_0|_{L^2(\Tt^2)} \geq 2C$, where $C$ appears in \eqref{eq:4y}. Then $u_0(\xi) \notin \FF_+(\xi)$. This shows \textbf{(iii)}: $\FF_0$ and $\FF_+$ are in direct sum.

From the additivity properties of Chern numbers, $c_1(\FF_0 \oplus \FF_+)= 0$. Therefore the bundle $\FF_0 \oplus \FF_+$ admits a smooth equivariant frame, see e.g. \cite[\S3]{Mo:17}. Moreover, as in \eqref{eq:4y}, there exists $C' > 0$ such that
\begin{equation}\label{eq:4u}
\xi \in [0,2\pi]^2, \ \ 
u \in \FF_0(\xi) \oplus \FF_+(\xi) \ \ \Rightarrow \ \  |\Delta u|_{L^2(\Tt^2)} \leq C' |u|_{L^2(\Tt^2)}.
\end{equation}

\textbf{2.}  Fix  $\tvp_1, \dots, \tvp_D$ satisfying \eqref{eq:1k} and such that
\begin{equation}
(x,\xi) \in \R^2, \ \ u \in \big[ \tvp_1(\xi), \dots, \tvp_D(\xi) \big]^\perp \ \ \Rightarrow \ \ \blr{\big(\Pp(x,\xi)-\lambda_+\big)u,u}_{L^2(\Tt^2)} \geq 3 |u|_{L^2(\Tt^2)}^2.
\end{equation}
Let $\vp_0 \in C^\infty_0(\R^2)$ with support in $(0,1)^2$ such that $\Delta \vp_0 \not\equiv 0$. We fix $t > 0$ (large enough in a sense progressively specified below), $y_1, \dots, y_D \in (0,1)^2$ pairwise distinct, and we set
\begin{equation}
\vp_j(y,\xi) \de \tvp_j(y,\xi) + \sum_{m \in \Z} e^{i\xi m} \vp_0\big(t (y-y_j-m)\big).
\end{equation}
We observe that $\big|\vp_j(\xi)-\tvp_j(\xi)\big|_{L^2(\Tt^2)} = O(t^{-1/2})$: for $t$ sufficiently large, $\vp_j(\xi)$ is a small perturbation of $\tvp_j(\xi)$. \cite[Proposition A.3]{DS:99} implies that for every $(x,\xi) \in \R^2 \times \R^2$,
\begin{equation}\label{eq:4v}
 u \in \big[ \vp_1(\xi), \dots, \vp_D(\xi) \big]^\perp  \ \ \Rightarrow  \ \  \blr{\big(\Pp(x,\xi)-\lambda_+\big)u,u}_{L^2(\Tt^2)} \geq 3 |u|_{L^2(\Tt^2)}^2.
\end{equation}

\textbf{3.}  Fix $\xi \in [0,2\pi]^2$. Let $u$ in $\big[ \vp_1(\xi), \dots, \vp_D(\xi) \big]$, with $|u|_{L^2(\Tt^2)} \leq 1$. We write
\begin{equation}\label{eq:4t}
u(y) = \sum_{j=1}^D a_j \cdot \vp_j(y,\xi) = \sum_{j=1}^D a_j \cdot \tvp_j(y,\xi) +  \sum_{m \in \Z}  e^{i\xi m}\sum_{j=1}^D a_j \cdot  \vp_0\big(t (y-y_j-m)\big).
\end{equation}
We note that $\lr{u,\tvp_j(\xi)}_{L^2(\Tt^2)} = a_j + O(t^{-1/2})$, uniformly in $u$. In particular, after increasing $t$, we can assume that $|a_j| \leq 2$. We take the Laplacian of \eqref{eq:4t} and bound below the $L^2(\Tt^2)$-norm:
\begin{equations}
|\Delta u|_{L^2(\Tt^2)} \geq \left| \sum_{j=1}^D a_j t^2 \cdot (\Delta \vp_0)\big(t(\cdot-y_j)\big) \right|_{L^2(\Tt^2)} -\left| \sum_{j=1}^D a_j \Delta \tvp_j(\xi) \right|_{L^2(\Tt^2)}
\\
\geq \left( \sum_{j=1}^D a_j^2 t^2 \cdot |\Delta \vp_0|_{L^2(\Tt^2)}^2 \right)^{1/2} - 2\sum_{j=1}^D  |\Delta \tvp_j(\xi)|_{L^2(\Tt^2)}
\\
 \geq t  \cdot |\Delta \vp_0|_{L^2(\Tt^2)} \cdot |a|_{\C^D} - 2\sum_{j=1}^D  |\Delta \tvp_j(\xi)|_{L^2(\Tt^2)}.  
\end{equations}
In the second line we used that for $t$ sufficiently large and $j=1, \dots, D$,  the supports of $\vp_0\big(t (\cdot-y_j)\big)$ do not intersect. 
 
The functions $\tvp_j(\xi)$ do not depend on $t$. Therefore, if we pick $t$ sufficiently large, $|\Delta u|_{L^2(\Tt^2)} \geq 2C'$, where $C'$ is the constant in \eqref{eq:4u}. It follows that for all $\xi \in [0,2\pi]^2$, the vector spaces $\FF_0(\xi) \oplus \FF_+(\xi)$ and $\big[ \vp_1(\xi), \dots, \vp_D(\xi) \big]$ are in direct sum. The equivariance property extends this relation to all $\xi \in \R^2$. 

We set $d = D+n+1$ and denote a smooth equivariant section of $\FF_0 \oplus \FF_+$ by $\vp_{D+1}, \dots, \vp_d$ (it exists by Step 1). Then 
\begin{equation}
\Range\big(\Pi_+(\xi)\big) = \FF_+(\xi) \subset \big[\vp_1(\xi), \dots, \vp_d(\xi)\big].
\end{equation}
The equation \eqref{eq:4v} a fortiori implies \eqref{eq:1p}. After performing a Gran--Schmidt process on $\vp_1(\xi), \dots, \vp_d(\xi)$, redefining these vectors if necessary, the proof is complete. 
\end{proof}

\subsection{Conclusion}\label{sec:44} Let $\vp_1, \dots, \vp_d$ be given by Lemma \ref{lem:1l} and $R_{12}(\xi)$, $R_{21}(\xi)$ defined according to \eqref{eq:1r}. Because of Lemma \ref{lem:1l}, \eqref{eq:1q} holds. Lemma \ref{lem:1a} implies 
\begin{equation}
u \in \big[ \vp_1(\cdot,\xi), \dots, \vp_d(\cdot,\xi)\big]^\perp  \ \ \Rightarrow \ \ \big \langle \big( \tQq(x,\xi)  - \lambda_+ \big)u,u \big\rangle_{L^2(\Tt^2)} \geq |u|_{L^2(\Tt^2)}^2,
\end{equation}
where $\tQq(x,\xi)= \psi\big(\tPp(x,\xi)\big)$. We note that for $x_2$ sufficiently large, $\tPp(x,\xi) = P_+(\xi)$.

Following \S\ref{sec:6.1}, we have
\begin{equation}\label{eq:5b}
\JJ_e(P_0,P_+) = \JJ(E_+), \ \ \ \ E_+(\xi;\lambda) = -R_{21}(\xi) \big( Q_+(\xi) - \lambda \big)^{-1} R_{12}(\xi),
\end{equation}
where $Q_+(\xi) = \psi\big(P_+(\xi)\big)$. The operator $P_+$ has an $L^2(\R^2)$-spectral gap $[\lambda_0-2\epsilon,\lambda_0+2\epsilon]$. Moreover, we recall that for some $\lambda_1 \leq \lambda_0 -2\epsilon$, $\lambda_2 \geq \lambda_0+2\epsilon$,
\begin{equation}
\psi(\lambda) = \systeme{\lambda_1 \ \text{ if  } \ \lambda \leq \lambda_0-2\epsilon \\ 
\lambda_2 \ \text{ if } \ \lambda \geq \lambda_0+2\epsilon}.
\end{equation}
It follows that
\begin{equations}\label{eq:5c}
Q_+(\xi) = \psi\big(P_+(\xi) \big) = \lambda_1 \Pi_+(\xi) + \lambda_2 \big( \Id - \Pi_+(\xi)\big).
\end{equations}
From \eqref{eq:5b} and \eqref{eq:5c}, we obtain
\begin{equation}
E_+(\xi;\lambda)^{-1} = R_{21}(\xi) \left( \dfrac{\Pi_+(\xi)}{\lambda-\lambda_1}  + \dfrac{\Id-\Pi_+(\xi)}{
\lambda-\lambda_2}\right) R_{12}(\xi).
\end{equation}

Let $u_1(y,\xi), \dots, u_n(y,\xi) \in C^\infty(\Tt^2 \times \R^2)$ be an orthonormal frame of the bundle $\FF_+ \rightarrow \R^2$. It exists because $\R^2$ is contractible \cite[\S1]{Mo:01}. In general, this frame is not equivariant.
Because of \eqref{eq:1r}, $R_{21}(\xi) R_{12}(\xi) = \Id_{\C^d}$. Moreover, (a) $\vp_1(\xi), \dots, \vp_d(\xi)$ form an orthonormal system in $L^2(\Tt^2)$; (b) $u_{1}(\xi), \dots, u_{n}(\xi)$ form an orthonormal system in $L^2(\Tt^2)$; (c) $\big[u_{1}(\xi), \dots, u_{n}(\xi)\big] \subset \big[ \vp_1(\xi), \dots, \vp_d(\xi)\big]$. Hence $R_{21}(\xi)u_1(\xi), \dots, R_{21}(\xi)u_n(\xi)$ form an orthonormal system in $\C^d$. There are fundamental consequences. First,
\begin{equation}
\Pi_1(\xi)\de R_{21}(\xi) \Pi_+(\xi) R_{12}(\xi) = \sum_{j=1}^n R_{21}(\xi)u_{j}(\xi) \otimes R_{21}(\xi)u_{j}(\xi)
\end{equation}
is an orthogonal projection in $\C^d$, which depends periodically on $\xi$. This allows us to define a bundle $\GG_+ \rightarrow \Tt^2_*$ with fibers $\Range \big( \Pi_1(\xi) \big) \subset \C^d$. Second, the map $R_{12}(\xi)$ induces a bundle isomorphism between $\EE_+ \rightarrow \Tt^2_*$ and $\GG_+ \rightarrow \Tt^2$.

In particular, these two bundles have the same topology. This expresses $c_1(\EE_+)$ using the Berry curvature associated to $\Pi_1(\xi)$:
\begin{equation}\label{eq:5f}
c_1(\EE_+) = \dfrac{i}{2\pi} \int_{\Tt^2_*} \Tr_{\C^d} \Big( \Pi_1(\xi) \big[ \p_1 \Pi_1(\xi), \p_2 \Pi_2(\xi) \big] \Big) d\xi.
\end{equation}
We now set $\Pi_2(\xi) = \Id - \Pi_1(\xi)$ and we obtain
\begin{equations}\label{eq:5e}
E_+(\xi;\lambda)^{-1} = \dfrac{\Pi_1(\xi)}{\lambda-\lambda_1} + \dfrac{\Pi_2(\xi)}{\lambda-\lambda_2}, \ \ \ \ 
E_+(\xi;\lambda) = (\lambda-\lambda_1) \Pi_1(\xi) + (\lambda-\lambda_2) \Pi_2(\xi).
\end{equations}
In particular, $\p_\lambda E_+ = \Id_{\C^d}$. The formula \eqref{eq:5d} for $\JJ(E_+)$ simplifies substantially:
\begin{equations}
\JJ(E_+) \de -\int_{\Tt^2_*}\int_{\p\Omega} \Tr_{\C^d} \left(\left(\dd{E_+}{\xi_1}  E_+^{-1} \dd{ \big(\p_\lambda E_+ \cdot E_+^{-1}\big)}{\xi_2} \right)(\xi;\lambda) \right)   \dfrac{d^1\lambda}{2i\pi} \dfrac{d\xi}{(2\pi)^2}
\\
= -\int_{\Tt^2_*}\int_{\p\Omega} \Tr_{\C^d} \left(\left(\dd{E_+}{\xi_1}  E_+^{-1} \dd{ E_+^{-1}}{\xi_2} \right)(\xi;\lambda) \right) \dfrac{d^1\lambda}{2i\pi} \dfrac{d\xi}{(2\pi)^2}.
\end{equations}

Thanks to \eqref{eq:5e}, 
\begin{equations}
 \Tr_{\C^d} \left(\left( \dd{E_+}{\xi_1} E_+^{-1} \dd{E_-^{-1}}{\xi_2} \right)(\xi;\lambda) \right) 
 = \sum_{j,k,\ell = 1,2} \dfrac{\lambda-\lambda_j}{(\lambda-\lambda_k) (\lambda-\lambda_\ell)} \cdot \Tr_{\C^d}\Big( \big(\p_1 \Pi_j \cdot \Pi_k \cdot \p_2 \Pi_\ell\big)(\xi)\Big).
\end{equations}
We recall that $\p\Omega$ encloses $\lambda_1$ but not $\lambda_2$. Hence, the integral
\begin{equation}
\int_{\p\Omega} \dfrac{\lambda-\lambda_j}{(\lambda-\lambda_k) (\lambda-\lambda_\ell)} \dfrac{d\lambda}{2i\pi}
\end{equation}
equals $1$ if $j=k=\ell=1$; $j=2$, $k=\ell=1$; $j=k=2$, $\ell=1$; and $j=\ell=2$, $k=1$. It vanishes in all other cases. We deduce that
\begin{equations}\label{eq:0o}
\int_{\p\Omega} \Tr_{\C^d}\left(\left( \dd{E_+}{\xi_1} E_+^{-1} \dd{(\p_\lambda E_+ \cdot E_+^{-1})}{\xi_2}\right)(\xi;\lambda) \right) \dfrac{d^1\lambda}{2i\pi}  \vspace*{-2mm}
\end{equations}
\begin{equations}
= \Tr_{\C^d}\Big( \Big(\p_1 \Pi_1 \cdot \Pi_1 \cdot \p_2 \Pi_1 + \p_1 \Pi_2 \cdot \Pi_1 \cdot \p_2 \Pi_1 + \p_1 \Pi_2 \cdot \Pi_2 \cdot \p_2 \Pi_1 + \p_1 \Pi_2 \cdot \Pi_1 \cdot \p_2 \Pi_2\Big)(\xi)\Big).
\end{equations}
The first and second term cancel out: $\p_1 (\Pi_1 + \Pi_2) = 0$. We use $\p_j \Pi_2 = -\Pi_1$ in the third and fourth term. Thus \eqref{eq:0o} equals 
\begin{equations}
\Tr_{\C^d}\Big( \Big(-\p_1 \Pi_1 \cdot \Pi_2 \cdot \p_2 \Pi_1 + \p_1 \Pi_1 \cdot \Pi_1 \cdot \p_2 \Pi_1\Big) (\xi) \Big).
\end{equations}
Since $\Pi_1 \Pi_2 = 0$, we get $\p_1 \Pi_1 \cdot \Pi_2 = -\Pi_1 \p_1 \Pi_2 = \Pi_1 \p_1 \Pi_1$. Thus we end up with
\begin{equations}
\Tr_{\C^d}\Big(  \Big(-\Pi_1 \cdot \p_1\Pi_1 \cdot \p_2 \Pi_1 + \p_1 \Pi_1 \cdot \Pi_1 \cdot \p_2 \Pi_1\Big)(\xi) \Big) = - \Tr_{\C^d}\Big( \Pi_1(\xi) \big[\p_1\Pi_1(\xi),\p_2 \Pi_1(\xi)\big]\Big).
\end{equations}
From this identity and \eqref{eq:5f}, we conclude that
\begin{equation}
\JJ_e(P_0,P_+) = \JJ(E_+) = \int_{\Tt^2_*} \Tr_{\C^d}\Big( \Pi_1(\xi) \big[\p_1\Pi_1(\xi),\p_2 \Pi_1(\xi)\big]\Big) \dfrac{d\xi}{(2\pi)^2} = -\dfrac{i}{2\pi} c_1(\EE_+).
\end{equation}

In particular, $2i\pi \cdot \JJ_e(P_0,P_+) = c_1(\EE_+)$. The same procedure shows that $2i\pi \cdot \JJ_e(P_-,P_0) = -c_1(\EE_-)$. The formula \eqref{eq:1s} and $\II_e(P) = i \JJ_e(P)$ end the proof of Theorem \ref{thm:0}:
\begin{equation}
2i\pi \cdot \JJ_e(P_-,P_+) = c_1(\EE_+) - c_1(\EE_-).
\end{equation}

\appendix

\section{}\label{app:1}

\begin{lem} There exists $a(\xi) \in C^\infty(\R^2,\C^2)$ such that the line $\C a(\xi)$ is $(2\pi\Z)^2$-periodic in $\xi$; and such that the vector bundle $\C a \rightarrow \Tt^2_*$ has Chern number $-\nu$.
\end{lem}

\begin{proof} \textbf{1.}  Fix $\epsi > 0$, $\az(\xi_1), \beta(\xi_1) \in C^\infty(\R,\R)$, both $2\pi$-periodic, with
\begin{equation}
\xi_1 \in [-1,1] \ \ \Rightarrow \ \ \az(\xi_1) = \xi_1
, \ \ \beta(\xi_1) = 0; \ \ \xi_1 \in [-\pi,\pi] \setminus [-1,1] \ \ \Rightarrow \beta(\xi_1) > 0.
\end{equation}
Let $M_\epsi(\xi) \in C^\infty\big(\R^2,M_2(\C)\big)$ be given by
\begin{equations}
M_\epsi(\xi) \de \matrice{\az(\xi_1) & \beta(\xi_1) + \epsi e^{-i\nu \xi_2} \\ \beta(\xi_1) + \epsi e^{i\nu\xi_2} & -\az(\xi_1)}.
\end{equations}
For any $\xi \in \R^2$, $M_\epsi (\xi)$ has a unique negative eigenvalue. Since $\R^2$ is contractible \cite[\S1]{Mo:01}, $M_\epsi(\xi)$ admits a normalized negative-energy eigenvector $a_\epsi(\xi) \in C^\infty(\R^2,\C^2)$. Since $M_\epsi(\xi)$ is $(2\pi\Z)^2$-periodic, the eigenspace $\C a_\epsi(\xi)$ is $(2\pi\Z)^2$-periodic. Thus it induces a vector bundle $\C a_\epsi \rightarrow \Tt^2_*$.

\textbf{2.}  The eigenprojector of $M$ associated to the negative eigenvalue is
\begin{equation}
\pi_\epsi = \Id - \dfrac{M_\epsi}{\sqrt{-\det M_\epsi}}.
\end{equation}
Thus the Berry curvature of $\C a_\epsi \rightarrow \Tt^2_*$ is
\begin{equation}
B_\epsi(\xi) \de -\Tr_{\C^2} \left( \dfrac{M_\epsi}{\sqrt{-\det M_\epsi}} \left[ \dd{}{\xi_1} \dfrac{M_\epsi}{\sqrt{-\det M_\epsi}}, \dd{}{\xi_2} \dfrac{M_\epsi}{\sqrt{-\det M_\epsi}} \right]\right).
\end{equation}
We observe that  as $\epsi \rightarrow 0$, the convergences
\begin{equation}
-\det M_\epsi = \az(\xi_1)^2 + |\beta(\xi_1)+\epsi e^{i\nu \xi_2}|^2 \rightarrow \az(\xi_1)^2+\beta(\xi_1)^2; \ \ \ \ \text{and} \ \  \p_2  M_\epsi \rightarrow 0
\end{equation}
are uniform. Moreover, for $\xi_1 \in [-1,1]$, $\az(\xi_1)^2+\beta(\xi_1)^2$ is bounded below by a positive constant. We deduce that $B_\epsi(\xi) \rightarrow 0$ uniformly away from $[-1,1] \times [-\pi,\pi]$.

\textbf{3.}  When $\xi_1 \in [-1,1]$, we have
\begin{equation}
M_\epsi(\xi) \de \matrice{\xi_1 & \epsi e^{-i\nu \xi_2} \\ \epsi e^{i\nu\xi_2} & -\xi_1}.
\end{equation}
A direct calculation, see e.g. \cite[Lemma 6.3]{Dr:18} shows that
\begin{equation}
B_\epsi(\xi_1) = \dfrac{i\epsi^2\nu}{2(\xi_1^2 + \epsi^2)^{3/2}}, \ \ \ \ \dfrac{i}{2\pi}\int_{[-1,1] \times [-\pi,\pi]} B_\epsi(\xi) d\xi \rightarrow -\nu. 
\end{equation}
It follows that if $\epsi$ is sufficiently small, then the Chern number of  $\C a_\epsi \rightarrow \Tt^2_*$ is $-\nu$, as claimed. This completes the proof. 
\end{proof}

\bibliographystyle{plain}
\bibliography{TI}

\begin{thebibliography}{10}

\bibitem[AFL18]	{Aea:18}
H.~Ammari, B.~Fitzpatrick, H.~Lee, E.~Orvehed Hiltunen and S.~Yu,
\newblock {\em Honeycomb-lattice {M}innaert bubbles.}
\newblock Preprint, \href{http://arxiv.org/abs/arXiv:1811.03905}{\nolinkurl{arXiv:1811.03905}}.

\bibitem[ASV13]{ASV:13}
J.~C. Avila, H.~Schulz-Baldes and C.~Villegas-Blas,
\newblock {\em Topological invariants of edge states for periodic two-dimensional
  models.}
\newblock Mathematical Physics, Analysis and Geometry,
  \textbf{16}:136--170, 2013.

\bibitem[B19]{Ba:19a}
G.~Bal,
\newblock {\em Continuous bulk and interface description of topological insulators.}
\newblock J. Math. Phys., \textbf{60}(8):20pp, 2019.

\bibitem[B77]{Be:77}
R.~Beals,
\newblock {\em Characterization of pseudodifferential operators and applications.}
\newblock Duke Math. J., \textbf{44}:45--57, 1977.

\bibitem[BES94]{BES:94}
J.~Bellissard, A.~van Elst and H.~Schulz-Baldes,
\newblock {\em The noncommutative geometry of the quantum {H}all effect. {T}opology
  and physics.}
\newblock J. Math. Phys., \textbf{35}(10):5373--5451, 1994.

\bibitem[BC18]{BC:18}
G.~Berkolaiko and A.~Comech,
\newblock {\em Symmetry and {D}irac points in graphene spectrum.}
\newblock Journal of Spectral Theory, \textbf{8}(3):1099---1147, 2018.

\bibitem[B84]{Be:84}
M.~V. Berry,
\newblock {\em Quantal phase factors accompanying adiabatic changes.}
\newblock Proc. Roy. Soc. London Ser. A, \textbf{392}(1802):45–57,
  1984.

\bibitem[BKR17]{BKR:17}
C.~Bourne, J.~Kellendonk and A.~Rennie,
\newblock {\em The {K}-theoretic bulk-edge correspondence for topological
  insulators.}
\newblock Ann. Henri Poincar\'e, \textbf{18}(5):1833--1866, 2017.

\bibitem[BR18]{BR:18}
C.~Bourne and A.~Rennie,
\newblock{\em  Chern numbers, localisation and the bulk-edge correspondence for
  continuous models of topological phases.}
\newblock Math. Phys. Anal. Geom., \textbf{21}(3), 2018.

\bibitem[B19]{Br:19}
M.~Braverman,
\newblock {\em Spectral flows of {T}oeplitz operators and bulk-edge correspondence.}
\newblock Lett. Math. Phys., \textbf{109}(10):2271--2289, 2019.

\bibitem[B87]{Bu:87}
V.S. Buslaev,
\newblock {\em Semiclassical approximation for equations with periodic coefficients.}
\newblock Russ. Math. Surv., \textbf{42}:97--125, 1987.

\bibitem[CG05]{CG:05}
J.-M. Combes and F.~Germinet.
\newblock {\em Edge and impurity effects on quantization of hall currents.}
\newblock Comm. Math. Phys., \textbf{256}(1):159--180, 2005.

\bibitem[DMV17]{DMV:17}
P.~Delplace, J.~B. Marston and A.~Venaille,
\newblock {\em Topological origin of equatorial waves.}
\newblock Science, \textbf{358}(6366):1075--1077, 2017.

\bibitem[D93]{Di:93}
M.~Dimassi.
\newblock {\em Asymptotic expansions of slow perturbations of the periodic
  {S}chr\"odinger operator.}
\newblock Comm. Partial Differential Equations,
  \textbf{18}(5-6):771--803, 1993.

\bibitem[DD14]{DD:14}
M.~Dimassi and A.~T. Duong,
\newblock {\em Trace asymptotics formula for the {S}chr\"odinger operators with
  constant magnetic fields.}
\newblock J. Math. Anal. Appl., \textbf{416}(1):427--448, 2014.

\bibitem[DGR02]{DGR:02}
M.~Dimassi, J.~C. Guillot and J.~Ralston,
\newblock {\em Semiclassical asymptotics in magnetic {B}loch bands.}
\newblock J. Phys. A, 35(35):7597--7605, 2002.

\bibitem[DGR04]{DGR:04}
M.~Dimassi, J.~C. Guillot and J.~Ralston,
\newblock {\em On effective {H}amiltonians for adiabatic perturbations of magnetic
  {S}chr\"odinger operators.}
\newblock Asymptot. Anal., \textbf{40}(2):137--146, 2004.

\bibitem[DS99]{DS:99}
M.~Dimassi and J.~Sj\"ostrand,
\newblock {\em Spectral Asymptotics in the Semi-Classical Limit.} 
\newblock London Mathematical Society, Lecture Note Series
  \textbf{268}(1999).

\bibitem[DZ03]{DZ:03}
M.~Dimassi and M.~Zerzeri,
\newblock {\em A local trace formula for resonances of perturbed periodic
  {S}chr\"odinger operators.}
\newblock J. Funct. Anal., \textbf{198}(1):142--159, 2003.

\bibitem[D18]{Dr:18}
A.~Drouot,
\newblock {\em The bulk-edge correspondence for continuous dislocated systems.}
\newblock Preprint, \href{http://arxiv.org/abs/arXiv:1810.10603}{\nolinkurl{arXiv:1810.10603}}.

\bibitem[D19a]{Dr:19a}
A.~Drouot,
\newblock {\em Characterization of edge states in perturbed honeycomb structures.}
\newblock Pure Appl. Anal., \textbf{1}(3):385--445, 2019.

\bibitem[D19b]{Dr:19b}
A.~Drouot,
\newblock {\em The bulk-edge correspondence for continuous honeycomb lattices.}
\newblock Comm. Partial Differential Equations,
  \textbf{44}(12):1406--1430, 2019.

\bibitem[DW19]{DW:19} 
A.~Drouot and M.~Weinstein,
\newblock {\em Edge states and the Valley Hall Effect.}
\newblock Preprint, \href{http://arxiv.org/abs/arXiv:1910.03509}{\nolinkurl{arXiv:1910.03509}}.

\bibitem[DFW18]{DFW:18}
A.~Drouot, C.~L. Fefferman and M.~I. Weinstein,
\newblock {\em Defect modes for dislocated periodic media.}
\newblock Preprint, \href{http://arxiv.org/abs/arXiv:1810.05875}{\nolinkurl{arXiv:1810.05875}}.



\bibitem[D75]{Dy:75}
E.~M. Dyn'kin,
\newblock {\em An operator calculus based upon the {C}auchy--{G}reen formula.}
\newblock J. Soviet Math., \textbf{4}(4), 329--34 (1975).


\bibitem[EG02]{EG:02}
P.~Elbau and G.~M. Graf,
\newblock {\em Equality of bulk and edge {H}all conductances revisted.}
\newblock Comm. Math. Phys., \textbf{229}:415--432, 2002.

\bibitem[EGS05]{EGS:05}
A.~Elgart, G.~M. Graf and J.~H. Schenker,
\newblock {\em Equality of the bulk and the edge {H}all conductances in a mobility
  gap.}
\newblock Comm. Math. Phys., \textbf{259}:185--221, 2005.


\bibitem[F19]{F:19} F. Faure, 
\newblock {\em Manifestation of the topological index formula in quantum waves and geophysical waves.}
\newblock Preprint, \href{http://arxiv.org/abs/arXiv:1901.10592}{\nolinkurl{arXiv:1901.10592}}.



\bibitem[F70]{Fe:70} B. V. Fedosov,  
\newblock {\em A direct proof of the formula for the index of an elliptic system in Euclidean space.}
\newblock Funkcional. Anal. i Prilozen., \textbf{4}(4): 83--84, 1970.

\bibitem[FLW16]{FLW:16}
C.~L. Fefferman, J.~P. Lee-Thorp and M.~I. Weinstein,
\newblock {\em Edge states in honeycomb structures.}
\newblock Annals of PDE, \textbf{2}(12), 2016.

\bibitem[FLW17]{FLW:17}
C.~L. Fefferman, J.~P. Lee-Thorp and M.~I. Weinstein,
\newblock {\em Topologically protected states in one-dimensional systems.}
\newblock Memoirs of the American Mathematical Society,
  \textbf{247}(1173), 2017.

\bibitem[FLW18]{FLW:18}
C.~L. Fefferman, J.~P. Lee-Thorp and M.~I. Weinstein,
\newblock {\em Honeycomb {S}chroedinger operators in the strong-binding regime.}
\newblock Comm. Pure Appl. Math., \textbf{71}(6), 2018.

\bibitem[FW12]{FW:12}
C.~L. Fefferman and M.~I. Weinstein,
\newblock {\em Honeycomb lattice potentials and {D}irac points.}
\newblock J. Amer. Math. Soc., \textbf{25}(4):1169--1220, 2012.

\bibitem[FT16]{FT:16}
S.~Freund and S.~Teufel,
\newblock {\em Peierls substitution for magnetic {B}loch bands.}
\newblock Anal. PDE, \textbf{9}(4):773--811, 2016.

\bibitem[GMS91]{GMS:91}
C.~G\'erard, A.~Martinez and S.~Sj\"ostrand,
\newblock {\em A mathematical approach to the effective {H}amiltonian in perturbed
  periodic problems.}
\newblock Comm. Math. Phys., \textbf{142}:217--244, 1991.

\bibitem[GP13]{GP:13}
G.~M. Graf and M.~Porta,
\newblock {\em Bulk-edge correspondence for two-dimensional topological insulators.}
\newblock Comm. Math. Phys., \textbf{324}(3):851--895, 2012.

\bibitem[GS18]{GS:18}
G.~M. Graf and J.~Shapiro.
\newblock {\em The bulk-edge correspondence for disordered chiral chains.}
\newblock Comm. Math. Phys., \textbf{363}(3):829--846, 2018.

\bibitem[GT18]{GT:18}
G.~M. Graf and C.~Tauber.
\newblock {\em Bulk-edge correspondence for two-dimensional {F}loquet topological
  insulators.}
\newblock Ann. Henri Poincar\'e, \textbf{19}(3):709--741, 2018.

\bibitem[GRT88]{GRT:88}
J.~C. Guillot, J.~Ralston and E.~Trubowitz,
\newblock {\em Semi-classical methods in solid state physics.}
\newblock Comm. Math. Phys., \textbf{116}:401--415, 1988.

\bibitem[HR07]{HR:07}
F.~D.~M. Haldane and S.~Raghu,
\newblock {\em Possible realization of directional optical waveguides in photonic
  crystals with broken time-reversal symmetry.}
\newblock Phys. Rev. Lett., \textbf{100}(1):013904, 2008.

\bibitem[H82]{Ha:82}
B.~I. Halperin,
\newblock {\em Quantized {H}all conductance, current-carrying edge states, and the
  existence of extended states in a two-dimensional disordered potential.}
\newblock Phys. Rev. B, \textbf{25}(4):2185--2190, 1982.

\bibitem[H93]{Ha:93}
Y.~Hatsugai,
\newblock {\em The {C}hern number and edge states in the integer quantum hall
  effect.}
\newblock Phys. Rev. Lett., \textbf{71}:3697--3700, 1993.


\bibitem[HS89]{HS:89}
B.~Helffer and J.~Sj\"ostrand
\newblock {\em Equation de {S}chr\"odinger avec champ magn\'etique et \'equation de {H}arper.}
\newblock Springer Lect. Notes in Physics \textbf{345}, 118--197, Springer, Berlin (1989).


\bibitem[HS90]{HS:90} B. Helffer and J. Sj\"ostrand, 
\newblock {\em On diamagnetism and de {H}aas-van {A}lphen effect.}
\newblock  Ann. Inst. H. Poincaré Phys. Th\'eor. \textbf{52}(1990), no. 4, 303--375. 

\bibitem[H85]{Ho:85}
L.~H\"ormander,
\newblock {\em The analysis of linear partial differential operators. III.
  Pseudodifferential operators.} 
\newblock Springer-Verlag,  Berlin \textbf{274}(1985).

\bibitem[KRS02]{KRS:02}
J.~Kellendonk, T.~Richter and H.~Schulz-Baldes,
\newblock {\em Edge current channels and {C}hern numbers in the integer quantum
  {H}all effect.}
\newblock Rev. Math. Phys., \textbf{14}(1):87--119, 2002.

\bibitem[KS04a]{KS:04a}
J.~Kellendonk and H.~Schulz-Baldes,
\newblock {\em Quantization of edge currents for continuous magnetic operators.}
\newblock J. Funct. Anal., \textbf{209}(2):388--413, 2004.

\bibitem[KS04b]{KS:04b}
J.~Kellendonk and H.~Schulz-Baldes,
\newblock {\em Boundary maps for crossed products with an application to the quantum
  {H}all effect.}
\newblock Comm. Math. Phys., \textbf{3}(3):611--637, 2004.

\bibitem[KMT13]{Kea:13}
A.~B. Khanikaev, S.~H. Mousavi, W.-K. Tse, M.~Kargarian, A.~H. MacDonald and
  G.~Shvets,
\newblock {\em Photonic topological insulators.}
\newblock Nature materials, \textbf{12}(3):233--239, 2013.

\bibitem[K17]{Ku:17}
Y.~Kubota,
\newblock {\em Controlled topological phases and bulk-edge correspondence.}
\newblock Comm. Math. Phys., \textbf{349}(2):493--525, 2017.

\bibitem[K16]{Ku:16}
P.~A. Kuchment,
\newblock {\em An overview of periodic elliptic operators.}
\newblock Bull. Amer. Math. Soc., \textbf{53}(3):343--414, 2016.

\bibitem[LWZ18]{LWZ:18}
J.~P. Lee-Thorp, M.~I. Weinstein and Y.~Zhu,
\newblock {\em Elliptic operators with honeycomb symmetry; {D}irac points, edge
  states and applications to photonic graphene.}
\newblock Arch. Rational Mech. Anal., 2018.

\bibitem[M17]{Mo:17}
D.~Monaco,
\newblock {\em Chern and {F}u--{K}ane--{M}ele invariants as topological
  obstructions.}
\newblock Advances in quantum mechanics, Springer, \textbf{18}:201--222,
  2017.

\bibitem[M01]{Mo:01}
J.~D. Moore,
\newblock {\em Lectures on {S}eiberg--{W}itten invariants.}, 
\newblock Lecture Notes in Mathematics, Springer–Verlag, Berlin \textbf{1629}(2001).

\bibitem[P07]{Pa:07}
G.~Panati,
\newblock {\em Triviality of {B}loch and {B}loch-{D}irac bundles.}
\newblock Ann. Henri Poincar\'e, \textbf{8}(5):995--1011, 2007.

\bibitem[PST03a]{PST:03a}
G.~Panati, H.~Spohn and S.~Teufel,
\newblock {\em Space-adiabatic perturbation theory.}
\newblock Adv. Theor. Math. Phys., \textbf{7}(1):145--204, 2003.

\bibitem[PST03b]{PST:03b}
G.~Panati, H.~Spohn and S.~Teufel,
\newblock {\em Effective dynamics for {B}loch electrons: Peierls substitution and
  beyond.}
\newblock Comm. Math. Phys., \textbf{242}(3):547--578, 2003.



\bibitem[PDV19]{PDV:19}
M.~Perrot, P. Delplace and A. Venaille,
\newblock {\em Topological transition in stratified fluids.}
\newblock Nature Physics \textbf{15}(2019), 781--784.

\bibitem[RH08]{RH:08}
S.~Raghu and F.~D.~M. Haldane,
\newblock {\em Analogs of quantum-{H}all-effect edge states in photonic crystals.}
\newblock Phys. Rev. A, \textbf{78}(3):033834, 2008.

\bibitem[RS78]{RS:78}
M.~Reed and B.~Simon,
\newblock {\em Methods of Modern Mathematical Physics: Analysis of Operators,
  Volume IV.}
\newblock Academic Press, 1978.

\bibitem[ST19]{ST:19}
J.~Shapiro and C.~Tauber,
\newblock {\em Strongly disordered {F}loquet topological systems.}
\newblock Ann. Henri Poincar\'e, \textbf{20}(6):1837--1875, 2019.

\bibitem[S83]{Si:83}
B.~Simon.
\newblock {\em Holonomy, the quantum adiabatic theorem, and {B}erry's phase.}
\newblock Phys. Rev. Lett. \textbf{51}(24):2167, 1983.


\bibitem[SZ91]{SZ:91} J.~Sj\"ostrand and M. Zworski, 
\newblock {\em Complex scaling and the distribution of scattering poles.}
\newblock J. Amer. Math. Soc., \textbf{4}(4), 729--769 (1991).

\bibitem[SZ07]{SZ:07}
J. Sj\"ostrand and M. Zworski,
\newblock {\em Elementary linear algebra for advanced spectral problems.}
\newblock Ann. Inst. Fourier \textbf{57}(2007), no. 7, 2095–2141.


\bibitem[T14]{Ta:14}
A.~Taarabt,
\newblock {\em Equality of bulk and edge {H}all conductances for continuous magnetic
  random {S}chr\"odinger operators.}
\newblock Preprint, \href{http://arxiv.org/abs/arXiv:1403.7767}{\nolinkurl{arXiv:1403.7767}}.




\bibitem[TKN82]{TKNN:82}
D.~J. Thouless, M.~Kohmoto, M.~P. Nightgale and M.~Den Nijs,
\newblock {\em Quantized {H}all conductance in a two-dimensional periodic potential.}
\newblock Phys. Rev. Lett. \textbf{49}:405, 1982.

\bibitem[W16]{W:16}
N.~Waterstraat,
\newblock {\em Fredholm Operators and Spectral Flow.}
\newblock Lecture notes, \href{http://arxiv.org/abs/arXiv:1603.02009}{\nolinkurl{arXiv:1603.02009}}.

\bibitem[YGS15]{Yea:15}
Z.~Yang, F.~Gao, X.~Shi, X.~Lin, Z.~Gao, Y.~Chong and B.~Zhang,
\newblock {\em Topological acoustics.}
\newblock Phys. Rev. Lett., 114:114301, 2015.

\bibitem[Z12]{Zw:12}
M.~Zworski,
\newblock {\em Semiclassical analysis}, volume \textbf{138}.
\newblock Graduate Studies in Mathematics, American Mathematical Society, 2012.

\end{thebibliography}
\end{document}